\pdfoutput=1
\documentclass[11pt]{article}
\usepackage{xcolor}
\usepackage{graphicx}
\usepackage{amsmath,amsfonts,amssymb,graphics,amsthm}
\usepackage{hyperref}
\usepackage{comment}
\usepackage{tabularx}
\usepackage[protrusion=true,expansion=true]{microtype}
\usepackage{enumerate}
\usepackage{bbm}
\usepackage{mathrsfs}
\usepackage[margin=1in]{geometry}

\hypersetup{
    colorlinks=false,
    linktocpage,
    }

\numberwithin{equation}{section}

\newtheorem*{thma}{Theorem A}
\newtheorem{theorem}{Theorem}[section]
\newtheorem{corollary}[theorem]{Corollary}
\newtheorem{lemma}[theorem]{Lemma}
\newtheorem{proposition}[theorem]{Proposition}
\newtheorem{remark}[theorem]{Remark}
\newtheorem{definition}[theorem]{Definition}
\theoremstyle{remark}

\def\@rst #1 #2other{#1}
\newcommand\MR[1]{\relax\ifhmode\unskip\spacefactor3000 \space\fi
  \MRhref{\expandafter\@rst #1 other}{#1}}
\newcommand{\MRhref}[2]{\href{http://www.ams.org/mathscinet-getitem?mr=#1}{MR#2}}

\def\MR#1{\href{http://www.ams.org/mathscinet-getitem?mr=#1}{MR#1}}

\newcommand{\C}{\mathbbm{C}}
\newcommand{\D}{\mathbbm{D}}
\newcommand{\E}{\mathbbm{E}}
\newcommand{\N}{\mathbbm{N}}

\newcommand{\Q}{\mathbbm{Q}}

\newcommand{\R}{\mathbbm{R}}
\renewcommand{\P}{\mathbbm{P}}
\newcommand{\bbH}{\mathbbm{H}}
\newcommand{\IG}{\mathrm{IG}}

\newcommand{\eps}{\varepsilon}

\let\Re\undefined
\DeclareMathOperator{\Re}{Re}
\let\Im\undefined
\DeclareMathOperator{\Im}{Im}

\DeclareMathOperator{\GFF}{GFF}
\DeclareMathOperator{\Cov}{Cov}
\DeclareMathOperator{\Var}{Var}
\DeclareMathOperator{\SLE}{SLE}

\def\cW{\mathcal{W}}

\def\cS{\mathcal{S}}

\def\cQ{\mathcal{Q}}

\def\cL{\mathcal{L}}
\def\cK{\mathcal{K}}

\def\cH{\mathcal{H}}

\def\cD{\mathcal{D}}
\def\cC{\mathcal{C}}

\def\alb#1\ale{\begin{align*}#1\end{align*}}
\def\allb#1\alle{\begin{align}#1\end{align}}

\newcommand{\aryb}{\begin{eqnarray*}}
\newcommand{\arye}{\end{eqnarray*}}
\def\alb#1\ale{\begin{align*}#1\end{align*}}
\newcommand{\eqb}{\begin{equation}}
\newcommand{\eqe}{\end{equation}}
\newcommand{\eqbn}{\begin{equation*}}
\newcommand{\eqen}{\end{equation*}}

\newcommand{\BB}{\mathbbm}
\newcommand{\ol}{\overline}
\newcommand{\ul}{\underline}
\newcommand{\op}{\operatorname}

\newcommand{\im}{\operatorname{Im}}
\newcommand{\re}{\operatorname{Re}}
\newcommand{\frk}{\mathfrak}

\newcommand{\ep}{\varepsilon}
\newcommand{\rta}{\rightarrow}

\newcommand{\wt}{\widetilde}
\newcommand{\wh}{\widehat} 
\newcommand{\mcl}{\mathcal}

\newcommand{\bdy}{\partial}

\let\originalleft\left
\let\originalright\right
\renewcommand{\left}{\mathopen{}\mathclose\bgroup\originalleft}
\renewcommand{\right}{\aftergroup\egroup\originalright}

\DeclareMathAlphabet{\mathpzc}{OT1}{pzc}{m}{it}

\begin{document}

\title{Liouville quantum gravity surfaces with boundary as matings of trees}
\author{
\begin{tabular}{c}Morris Ang\\[-5pt]\small MIT\end{tabular}\; 
\begin{tabular}{c}Ewain Gwynne\\[-5pt]\small Cambridge\end{tabular}
} 
\date{  }

\maketitle

\begin{abstract}
For $\gamma \in (0,2)$, the quantum disk and $\gamma$-quantum wedge are two of the most natural types of Liouville quantum gravity (LQG) surfaces with boundary. These surfaces arise as scaling limits of finite and infinite random planar maps with boundary, respectively.
We show that the left/right quantum boundary length process of a space-filling SLE$_{16/\gamma^2}$ curve on a quantum disk or on a $\gamma$-quantum wedge is a certain explicit conditioned two-dimensional Brownian motion with correlation $-\cos(\pi\gamma^2/4)$. 
This extends the mating of trees theorem of Duplantier, Miller, and Sheffield (2014) to the case of quantum surfaces with boundary (the disk case for $\gamma \in (\sqrt 2 , 2)$ was previously treated by Duplantier, Miller, Sheffield using different methods).
As an application, we give an explicit formula for the conditional law of the LQG area of a quantum disk given its boundary length by computing the law of the corresponding functional of the correlated Brownian motion. 
\end{abstract}

\tableofcontents

\section{Introduction}
\label{sec-intro}

\subsection{Overview}
\label{sec-overview}

Let $h$ be an instance of the Gaussian Free Field (GFF) on a planar domain $D$, and fix $\gamma \in (0,2)$. Informally, the $\gamma$-Liouville quantum gravity (LQG) surface associated with $(D,h)$ is the random surface conformally parametrized by $D$, with metric tensor $e^{\gamma h} \, (dx^2+dy^2)$, where $dx^2+dy^2$ is the Euclidean metric tensor. LQG surfaces are expected (and in some cases proven) to be the scaling limits of random planar maps. The case $\gamma = \sqrt{8/3}$, sometimes called \emph{pure gravity}, corresponds to uniform random planar maps, and other values correspond to random planar maps weighted by the partition function of an appropriate statistical mechanics model (sometimes called ``gravity coupled to matter"). For example, $\gamma = \sqrt 2$ corresponds to random planar maps weighted by the number of spanning trees they admit and $\gamma=\sqrt{4/3}$ corresponds to random planar maps weighted by the number of bipolar orientations~\cite{kmsw-bipolar} they admit. 

The GFF $h$ does not have well-defined pointwise values, so the above definition of LQG does not make rigorous sense. However, one can define LQG rigorously using various regularization procedures. For example, it is possible to define the LQG area measure $\mu_h$ on $D$ as a limit of regularized versions of $e^{\gamma h(z)} \,dz$, where $dz$ denotes Lebesgue measure~\cite{kahane,shef-kpz,rhodes-vargas-review}. In a similar vein, one can define the LQG boundary length measure $\nu_h$ on $\bdy D$ (in the case when $D$ has a boundary) and on certain curves in $D$, including SLE$_\kappa$-type curves for $\kappa = \gamma^2$~\cite{shef-zipper}. 
The measures $\mu_h$ and $\nu_h$, respectively, are expected to be the scaling limits of the counting measure on vertices and the counting measure on boundary vertices for random planar maps. This convergence has been proven for a few types of planar maps conformally embedded in the plane~\cite{gms-tutte,hs-cardy-embedding} and for various types of uniform planar maps in the Gromov-Hausdorff-Prokhorov topology (see, e.g.,~\cite{legall-uniqueness,miermont-brownian-map,bet-mier-disk,gwynne-miller-uihpq,bmr-uihpq}).

The measures $\mu_h$ and $\nu_h$ satisfy a conformal covariance relation which leads to a natural rigorous definition of LQG surfaces. 
Suppose $D, \widetilde D$ are planar domains and $\varphi: D \to \widetilde D$ is a conformal map. If $\widetilde h$ is a GFF on $\widetilde D$ and 
\begin{equation}\label{eqn: quantum surface defn}
h = \widetilde h \circ \varphi + Q \log |\varphi'| \quad \text{ where } Q = \frac2\gamma + \frac\gamma2,
\end{equation}
then by~\cite[Proposition 2.1]{shef-kpz} the LQG area and boundary length measures satisfy $\mu_{\wt h} = \varphi_* \mu_h$ and $\nu_{\wt h} = \varphi_* \nu_h$, where $\varphi_*$ denotes the pushforward. This leads us to define an equivalence relation on pairs $(D,h)$ by saying that $(D,h) \sim (\widetilde D, \widetilde h)$ if there exists some $\varphi$ for which \eqref{eqn: quantum surface defn} holds. Following~\cite{shef-kpz,shef-zipper,wedges}, we define an equivalence class of such pairs $(D,h)$ to be a \emph{quantum surface}. We will often want to decorate a quantum surface by one or more marked points in $D\cup \bdy D$ or paths. In this situation, we define equivalence classes via \eqref{eqn: quantum surface defn}, and further require that the conformal map $\varphi$ maps decorations on the first surface to corresponding decorations on the second surface.

There are many deep results concerning $\gamma$-LQG surfaces decorated by Schramm-Loewner Evolution (SLE$_\kappa$)~\cite{schramm0,schramm-sle} curves for $\kappa \in \{\gamma^2,16/\gamma^2\}$. Such results are the continuum analogs of special symmetries which arise for random planar maps decorated by the ``right" type of statistical mechanics model, whose partition function matches up with the weighting of the random planar map. 

One of the most important connections between SLE and LQG is the \emph{mating of trees} or \emph{peanosphere} theorem of Duplantier, Miller, and Sheffield~\cite{wedges,sphere-constructions}. 
The whole-plane version of this theorem concerns a special type of $\gamma$-LQG surface parametrized by the whole plane, called a \emph{$\gamma$-quantum cone}, decorated by an independent space-filling SLE$_\kappa$ curve $\eta$ for $\kappa = 16/\gamma^2$ (see Section~\ref{section: preliminaries} for more background on these objects). 
The theorem states that if we parametrize $\eta$ so that it traverses one unit of LQG mass in one unit of time, then the net change in the LQG boundary lengths of the left and right outer boundaries of $\eta$ relative to time 0 evolve as a pair of correlated Brownian motions, with correlation $-\cos(\pi\gamma^2/4)$. 
Roughly speaking, the space-filling SLE-decorated $\gamma$-quantum cone can be obtained by gluing together, or ``mating" the continuum random trees (CRT's) associated with these two Brownian motions, and the curve $\eta$ corresponds to the peano curve which snakes between the two mated trees. See~\cite[Section 1.3]{wedges} and Figure~\ref{fig: boundary_lengths_disk_bottom} for more detail on this point.
See also~\cite{sphere-constructions} for an analog of this result on the sphere rather than the whole plane. 

The mating of trees theorem is a continuum analog of so-called \emph{mating of trees bijections} for random planar maps, such as the Mullin bijection and its generalization the Hamburger-Cheeseburger bijection~\cite{mullin-maps,bernardi-maps,shef-burger,gkmw-burger}.
Such bijections encode a random planar map decorated by a statistical mechanics model (a spanning tree in the case of the Mullin bijection, or an instance of the FK cluster model~\cite{fk-cluster} in the case of the Hamburger-Cheeseburger bijection) in terms of a pair of discrete random trees, or equivalently a random walk on $\BB Z^2$ with a certain increment distribution. 
In many cases it is possible to show that the encoding walk for the decorated random planar map converges in the scaling limit to the pair of correlated Brownian motions arising in the continuum mating of trees theorem (this type of convergence is called ``peanosphere convergence"). 
This constitutes the first rigorous connection between random planar maps and LQG.

The mating of trees theorem has proven to be an extremely fruitful tool in the study of random planar maps, LQG, and SLE. 
For a few examples, it is used in the proof of the equivalence between $\sqrt{8/3}$-LQG and the Brownian map~\cite{lqg-tbm1,lqg-tbm2,lqg-tbm3}, to study various fractal properties of SLE~\cite{ghm-kpz,gp-sle-bubbles}, to study random planar maps embedded in the plane~\cite{gms-tutte}, and to compute exponents for graph distances and for various processes on random planar maps (see, e.g.,~\cite{ghs-map-dist,gm-spec-dim}). 
See~\cite{ghs-mating-survey} for a survey of results proven using mating-of-trees theory. 

The goal of this paper is to prove extensions of the mating of trees theorem to the two most natural LQG surfaces with boundary: the quantum disk and the $\gamma$-quantum wedge. See Theorems~\ref{thm: peanosphere disk} and~\ref{thm: peanosphere gamma wedge} for precise statements. (For $\gamma \in (\sqrt 2 , 2)$, the quantum disk case was previously treated in~\cite{wedges,sphere-constructions} using different techniques.)
One reason why these surface are natural is that they are expected to arise as the scaling limits of planar maps with the topology of the disk and the half-plane, respectively (see, e.g.,~\cite[Section 5]{hrv-disk} for a precise conjecture in the disk case). 
As in the case of random planar maps without boundary, our results are continuum analogs of mating of trees bijections for random planar maps with boundary.
We will not go into detail about this here since our focus is on the continuum theory, but see, e.g.,~\cite{gp-dla,bhs-site-perc,kmsw-bipolar} for some discussion of such bijections. 

Our results are useful for identifying the scaling limits of random planar maps with boundary, both in the sense of ``peanosphere convergence" discussed above and in the setting of random planar maps embedded in the plane. For example, in~\cite{gms-tutte}, the scaling limit of the so-called \emph{mated-CRT map} with boundary, embedded into the plane via the Tutte embedding (a.k.a.\ the harmonic embedding) is not explicitly described in the case when $\gamma \in (0,\sqrt 2]$ (see~\cite[Footnote 3]{gms-tutte}). Our results immediately imply that this scaling limit is a quantum disk decorated by an independent SLE$_{16/\gamma^2}$ loop based at a boundary point, as one would expect. 

Our results also have applications to proving exact formulas for LQG, since the mating of trees theorem allows us to reduce LQG calculations to much easier calculations for a correlated two-dimensional Brownian motion. 
In particular, we will explicitly identify the law of the area of a quantum disk given its boundary length modulo a single unknown constant (the variance of the peanosphere Brownian motion); see Theorem~\ref{thm: intro area disk}. 
This gives a new approach to proving exact formulas for LQG which is completely different from (but probably less general than) the conformal field theory techniques used to prove other exact formulas for LQG in~\cite{krv-dozz,remy-fb-formula,rz-gmc-interval}.

\subsection{Main results}
\label{sec-results}

Here and throughout the rest of the paper, we fix $\gamma \in (0,2)$ and define
\begin{equation}
\kappa = \gamma^2, \quad \kappa' = \frac{16}{\gamma^2}, \quad Q = \frac\gamma2 + \frac2\gamma.
\end{equation}
 
There is an important one-parameter family of quantum surfaces with two marked boundary points called \emph{$\alpha$-quantum wedges} for $\alpha \in (-\infty,Q + \frac{\gamma}{2})$. For the parameter range $\alpha \in (-\infty ,Q]$, the surface is called a \emph{thick quantum wedge}. Thick quantum wedges are typically parametrized by $\BB H$ with marked points at 0 and $\infty$. Every bounded neighborhood of 0 has finite total LQG mass but the complement of every such neighborhood has infinite LQG mass. 
For $\alpha \in (Q, Q + \frac\gamma2)$, the surface is called a \emph{thin quantum wedge}. Informally, it is an infinite Poissonian ``chain'' (concatenation) of finite volume quantum surfaces, called \emph{beads}, each with two marked boundary points. See Section~\ref{subsection: quantum wedges and disks} for a more comprehensive review of these random surfaces.

For a simply connected domain $D$ with marked boundary points $a,b$, for $\kappa' > 4$ one can define a random space-filling curve called \emph{space-filling $\SLE_{\kappa'}$} from $a$ to $b$. For $\kappa' \geq 8$, this is just ordinary chordal SLE$_{\kappa'}$. For $\kappa' \in (4,8)$, space-filling $\SLE_{\kappa'}$ can be obtained from ordinary chordal SLE$_{\kappa'}$ by iteratively ``filling in the bubbles" which it disconnects from its target point by SLE$_{\kappa'}$-type curves. By taking the limit as $b \to a$ in the counterclockwise definition, we can define \emph{counterclockwise} space-filling $\SLE_{\kappa'}$ rooted at the point $a$. See Section~\ref{subsection: space-filling SLE} for a discussion on space-filling SLE. In this paper we will be concerned with random surfaces decorated by independent space-filling $\SLE_{\kappa'}$ curves. This is easy to define for surfaces parametrized by simply connected domains (such as quantum disks or thick quantum wedges): we just sample the space-filling SLE$_{\kappa'}$ independently from the GFF-type distribution which describes the quantum surface.
In the case of a thin wedge, a space-filling SLE$_{\kappa'}$ between the two marked points is defined to be a concatenation of independent space-filling SLE$_{\kappa'}$s in the beads of the thin wedge, each going between the two marked points of its corresponding bead; see Figure~\ref{fig: alpha_wedge}, right.

We first briefly explain the mating of trees theorem for the $\frac{3\gamma}{2}$-quantum wedge (which is an immediate consequence of the main result of \cite{wedges}), then state new mating-of-trees theorems for the unit boundary-length quantum disk and the $\gamma$-quantum wedge. 
We note that a $\frac{3\gamma}{2}$-quantum wedge is thick if and only if $\gamma \leq \sqrt 2$. 

\begin{thma}[\cite{wedges}] \label{thm: peanosphere alpha wedge}
Let $\gamma \in (0,2)$.
Consider a $\frac{3\gamma}{2}$-quantum wedge $(\bbH, h, 0, \infty)$ decorated by an independent space-filling $\SLE_{\kappa'}$ curve $\eta'$ from $0$ to $\infty$. Parametrize $\eta'$ by LQG area so that $\mu_h(\eta'([s,t])) = t-s$ for each $0\leq s  \leq t  <\infty$. Let $(L_t, R_t)_{t \geq 0}$ be the left/right boundary length process as defined in Figure \ref{fig: alpha_wedge}. Then $(L_t, R_t)_{t \geq 0}$ evolves as a Brownian motion with variances and covariances given by 
\begin{equation}\label{eqn: covariance}
\Var(L_t) = {\BB a}^2 t, \quad \Var(R_t) = {\BB a}^2 t, \quad \Cov(L_t, R_t) = - \cos (\pi \gamma^2 /4) {\BB a}^2 t \quad \text{ for } t \geq 0 ,
\end{equation}
where $\BB a > 0$ is a deterministic constant which depends only on $\gamma$ (and is not made explicit in~\cite{wedges}). 
Moreover, $(L_t, R_t)_{t \geq 0}$ a.s.\ determines the $\frac{3\gamma}{2}$-quantum wedge decorated by the space-filling SLE, viewed as a curve-decorated quantum surface (i.e., viewed modulo conformal coordinate changes as in~\eqref{eqn: quantum surface defn} which fix 0 and $\infty$). 
\end{thma}

This theorem was proved\footnote{See Section~\ref{sec-MOT} for details.} in \cite[Theorem 1.9, Theorem 1.11]{wedges}, except for the explicit form of the covariances \eqref{eqn: covariance} for $\gamma \in (0,\sqrt 2)$ which was later established in \cite{kappa8-cov}. We emphasize that although here the boundary length process $(L_t, R_t)_{t \geq 0}$ has specified initial value $(L_0,R_0)=(0,0)$, we only care about the changes in $(L_t, R_t)_{t \geq 0}$ over time rather than the exact values, so we will sometimes modify boundary length processes by an additive constant. 

\begin{figure}[ht!]
\begin{center}
  \includegraphics[scale=0.85]{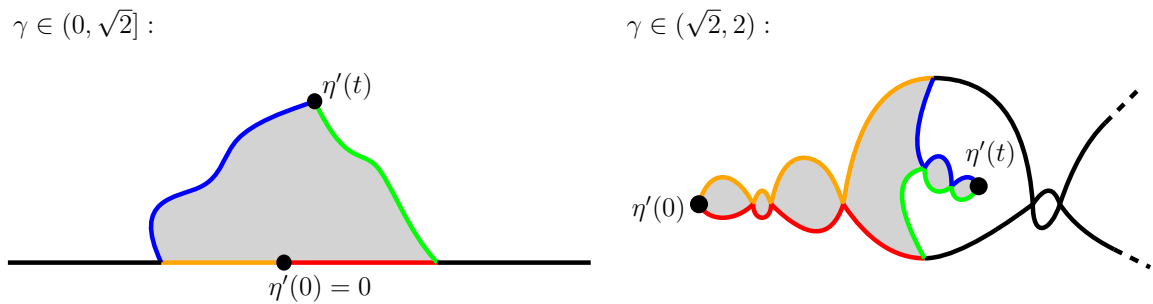}%
 \end{center}
\caption{\label{fig: alpha_wedge} Illustration of Theorem~\hyperref[thm: peanosphere alpha wedge]{A}. Consider a $\frac{3\gamma}{2}$-quantum wedge with field $h$, decorated by an independently drawn space-filling $\SLE_{\kappa'}$ curve $\eta'$ from $0$ to $\infty$ parametrized by quantum area. We define $L_t = \nu_h(\text{blue}) - \nu_h(\text{orange})$, and $R_t = \nu_h(\text{green}) - \nu_h(\text{red})$. Then $(L_t, R_t)$ evolves as Brownian motion with covariances given by \eqref{eqn: covariance}. \textbf{Left: }For $\gamma \in (0,\sqrt2]$, the $\frac{3\gamma}{2}$-quantum wedge is thick, so we can parametrize it by $\bbH$. \textbf{Right: }For $\gamma \in (\sqrt2, 2)$, the $\frac{3\gamma}{2}$-quantum wedge is thin, so it is a countable collection of doubly-marked disk-homeomorphic quantum surfaces together with a total ordering on the set of such surfaces.}
\end{figure}

The unit boundary length quantum disk is a kind of quantum surface with the topology of the disk which has (random) finite area and boundary length one, first introduced in \cite{wedges}. It typically comes with one or more marked boundary points, which are sampled independently from the quantum boundary length measure. The unit boundary length quantum disk is equivalent to the quantum disk considered in~\cite{hrv-disk}. This will be proved in the forthcoming paper~\cite{cercle-quantum-disk}; see~\cite{ahs-sphere} for the sphere case. See Section~\ref{subsection: disk} for more on the quantum disk.

It is known that in the regime $\gamma \in (\sqrt2, 2)$, if one decorates a quantum disk with an independent counterclockwise space-filling $\SLE_{\kappa'}$ from a marked boundary point to itself and defines the left/right boundary length process appropriately, then the boundary length process is a two-dimensional Brownian motion conditioned to remain in the first quadrant. This is proved in \cite{wedges} and elaborated upon in \cite[Theorem 2.1]{sphere-constructions}. The reason why the proof is easier for $\gamma \in (\sqrt 2,2)$ is as follows. When $\kappa' = 16/\gamma^2 \in (4,8)$, space-filling SLE$_{\kappa'}$ surrounds ``bubbles" (regions with the topology of the disk) and subsequently fills them in (this is related to the fact that the quantum wedge in Theorem~\ref{thm: peanosphere alpha wedge} is thin for $\gamma \in (\sqrt 2,2)$). The quantum surfaces obtained by restricting the field to these bubbles are quantum disks, so one can deduce the quantum disk case from Theorem \hyperref[thm: peanosphere alpha wedge]{A} by restricting to one of the bubbles.
In this paper, we extend the result to the full range $\gamma \in (0, 2)$ (see Corollary~\ref{cor: restate disk peanosphere} for an explanation of the equivalence of the results for $\gamma \in (\sqrt2,2)$).

\begin{theorem}\label{thm: peanosphere disk}
Suppose that $\gamma \in (0,2)$, and that $(\D, \psi, -1)$ is a unit boundary length quantum disk with random quantum area $\mu_\psi(\D)$ and one marked boundary point. Let $\eta':[0,\mu_\psi(\D)]\to \overline\D$ be a counterclockwise space-filling $\SLE_{\kappa'}$ process from $-1$ to $-1$ sampled independently from $\psi$ and then parametrized by $\mu_\psi$-mass. Let $L_t$ and $R_t$ denote the quantum lengths of the left and right sides of $\eta'([0,t])$, with additive constant normalized so that $L_0 = 0$ and $R_0 = 1$; see Figure \ref{fig: boundary_lengths_disk}. Then $(L_t,R_t)_{ 0 \leq t \leq \mu_\psi(\D)}$ is a finite-time Brownian motion started from $(0,1)$ and conditioned to stay in the first quadrant $\R^+ \times \R^+$ until it hits $(0,0)$, with variances and covariances as in~\eqref{eqn: covariance} (Figure~\ref{fig: boundary_lengths_disk_bottom}).
Moreover, the function $(L_t, R_t)_{0\leq t \leq \mu_\psi(\D)}$ a.s.\ determines $(\D , \psi,\eta' , -1)$ as a curve-decorated quantum surface (i.e., viewed modulo conformal coordinate changes as in~\eqref{eqn: quantum surface defn} which fix $-1$). 
\end{theorem}

The Brownian motion of Theorem \ref{thm: peanosphere disk} is conditioned on a probability zero event; we discuss the precise definition of this process in Section \ref{section: cone excursions}. 
The statement that $(L_t, R_t)_{0\leq t \leq \mu_\psi(\D)}$ a.s.\ determines $(\D, \psi,\eta',-1)$ can equivalently be phrased as follows.
If we fix some canonical choice of equivalence class representative of the curve-decorated quantum surface $(\D , \psi , \eta', -1)$ (e.g., we require that the $\gamma$-LQG lengths of the arcs separating $-1$, $-e^{2\pi i /3}$, and $e^{-4\pi i/3}$ are each equal to $1/3$) then $(L_t, R_t)_{0\leq t \leq \mu_\psi(\D)}$ a.s.\ determines $(\psi,\eta')$. 
As in~\cite{wedges,sphere-constructions}, our proof does not give an explicit description of the functional which takes in $(L_t, R_t)_{0\leq t \leq \mu_\psi(\D)}$ and outputs $(\psi, \eta')$. However, this functional can be made explicit using the results of~\cite{gms-tutte}; see, in particular,~\cite[Remark 1.4]{gms-tutte}. 

\begin{figure}[ht!]
\begin{center}
\includegraphics[scale=0.85]{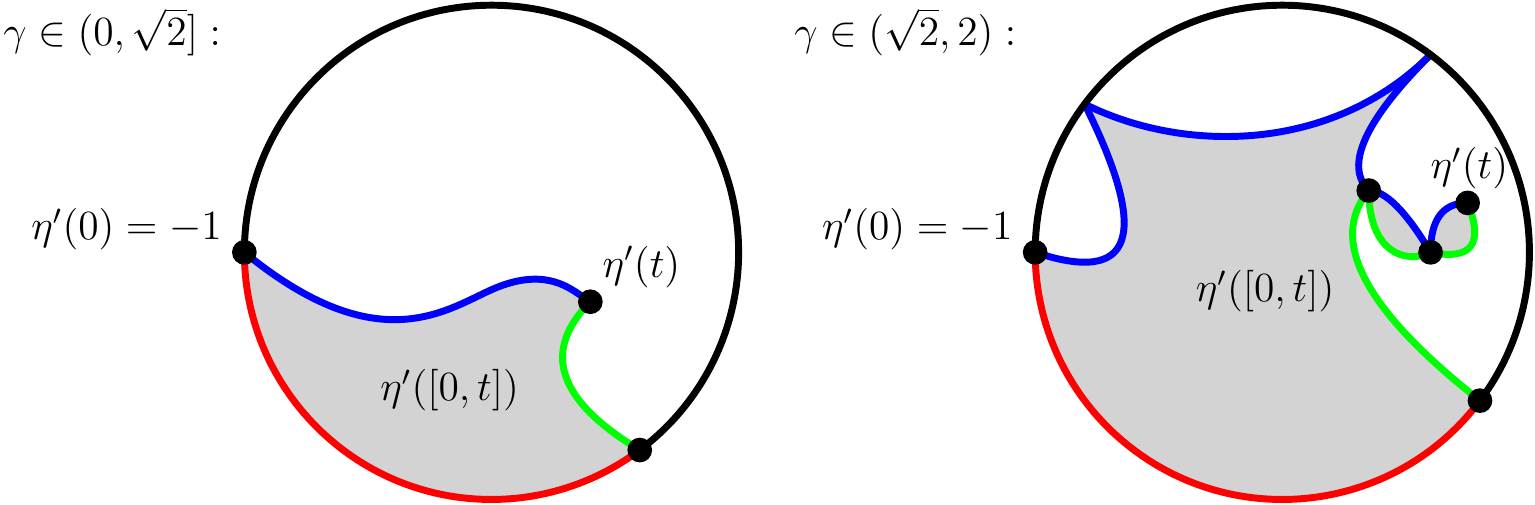}
\end{center}
\caption{\label{fig: boundary_lengths_disk} For $\gamma \in (0, 2)$, consider a unit boundary length quantum disk $(\D,\psi ,-1)$ with an independently drawn counterclockwise space-filling $\SLE_{\kappa'}$ curve $\eta'$ from $-1$ to $-1$ parametrized by quantum area. For $t > 0$, we define $L_t = \nu_\psi(\text{blue})$, and $R_t = 1 + \nu_\psi(\text{green}) - \nu_\psi(\text{red})$. Note that for the case $\gamma \in (\sqrt2,2)$, the curve stopped at time $t$ contains (at most countably many) connected components joined at pinch points. Each such component has a left and right boundary, and in this description of $(L_t, R_t)$, one should take the sum over the appropriate left/right boundaries of the components.}
\end{figure}

\begin{figure}[ht!]
\begin{center}
\includegraphics[scale=0.65]{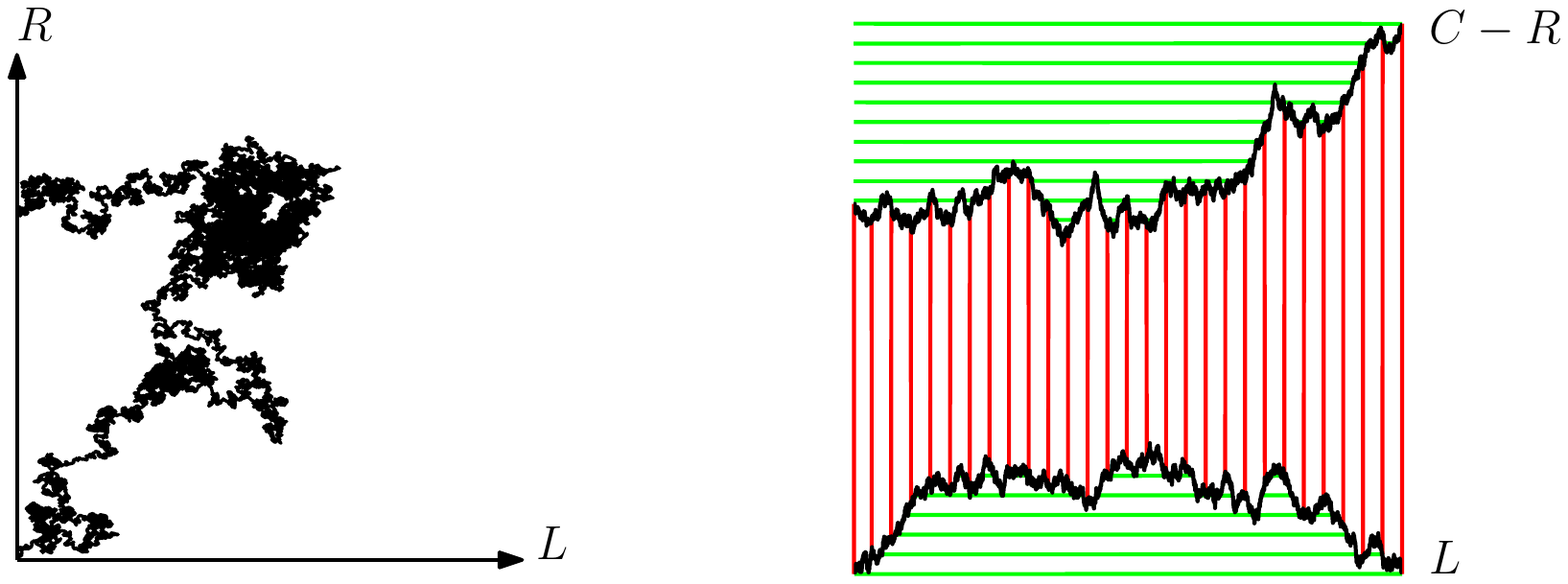}
\end{center}
\caption{\label{fig: boundary_lengths_disk_bottom} \textbf{Left: }Theorem~\ref{thm: peanosphere disk} tells us that $(L_t, R_t)$ evolves as a Brownian motion with covariances given by \eqref{eqn: covariance}, starting at $(0,1)$ and ending at $(0,0)$ and conditioned to stay in the positive quadrant. Pictured is a sample for $\gamma = \sqrt{2}$. \textbf{Right: } As in the whole-plane and sphere cases (see~\cite[Section 1.3]{wedges}), one can recover the curve-decorated topological space $(\ol{\BB D} ,\eta')$ from  $(L,R)$ explicitly as follows. 
We first plot the graphs of $L_t $ and $C - R_t$ against $t$ (with $C$ chosen sufficiently large so the graphs are disjoint), as in the figure. 
We then identify all points in the rectangle $[0,\mu_\psi(\BB D)] \times [0,C]$ which lie on the same vertical line segment between the graphs (several such segments are shown in red) or the same horizontal line segment above the graph of $C-R$ or below the graph of $L$ (green). 
The resulting topological quotient space, decorated by the curve obtained by tracing along the graph of $L$ (equivalently, $C-R$) from left to right is homeomorphic to $(\ol{\BB D} ,\eta')$ via a curve-preserving homeomorphism. This can be seen using Theorem~\ref{thm: peanosphere disk} and exactly the same arguments as in the whole-plane and sphere cases. The boundary of the disk is the image under the quotient map of the vertical segment between the points $(0, C-1)$ and $(0,C)$.} 
\end{figure}

Since $\eta'$ is parametrized by $\mu_\psi$-mass, the area $\mu_\psi(\BB D)$ of the quantum disk in Theorem~\ref{thm: peanosphere disk} is the random time that the Brownian motion of Theorem~\ref{thm: peanosphere disk} hits $(0,0)$. 
Theorem~\ref{thm: peanosphere disk} along with a Brownian motion calculation will therefore allow us to prove the following theorem.

\begin{theorem}\label{thm: intro area disk}
Recall the unknown constant $\BB a$ from~\eqref{eqn: covariance}. The area of the unit boundary length quantum disk is distributed according to the law
\begin{equation*}
  \P[\mu_\psi(\D) \in dt] = \frac{1}{c t^{1+4/\gamma^2}}  \exp\left(-\frac{1}{2 ( \BB a \sin(\pi\gamma^2/4) )^2  t } \right)  \, dt,
\end{equation*}
where
\[c =  2^{4/\gamma^2} \Gamma(4/\gamma^2) (\BB a \sin(\pi\gamma^2/4) )^{ 8/\gamma^2 } .\]
\end{theorem}
The exact formula for the law of $\mu_\psi(\BB D)$ does not appear elsewhere in the existing literature. 
However, Guillame Remy and Xin Sun [Private communication] have informed us of a work in progress in which they prove the same formula as in Theorem~\ref{thm: intro area disk} without the unknown constant $\BB a$. This is done using techniques which are similar to those in~\cite{krv-dozz,remy-fb-formula,rz-gmc-interval} and completely different from those in the present paper.
Comparing the two formulas will lead to a computation of the unknown constant $\BB a$ of~\eqref{eqn: covariance}. 

The quantum wedge with $\alpha = \gamma$ is particularly special. Informally, when one zooms in on a typical boundary point of a $\gamma$-quantum surface from the perspective of the $\gamma$-LQG boundary length measure and simultaneously re-scales so that LQG areas remain of constant order, then the resulting surface is a $\gamma$-quantum wedge. See~\cite[Proposition 1.6]{shef-zipper} for a precise statement of this form. 
 Since $\gamma \in (0,Q)$ for $\gamma \in (0,2)$, the $\gamma$-quantum wedge is always thick, so we can parametrize it by $\bbH$. By zooming in near a boundary-typical point of a $\frac{3\gamma}{2}$-quantum wedge, we will prove the following mating of trees theorem for the $\gamma$-quantum wedge (which is new for all $\gamma \in (0,2)$). 

\begin{theorem}\label{thm: peanosphere gamma wedge}
Suppose that $\gamma \in (0,2)$, and that $(\bbH, h, 0, \infty)$ is a $\gamma$-quantum wedge. Let $\eta' : \BB R\rta \ol{\BB H}$ be a counterclockwise space-filling $\SLE_{\kappa'}$ process from $\infty$ to $\infty$ sampled independently from $h$ and then reparametrized by quantum area, and with time recentered so that $\eta'(0) = 0$. Let $L_t$ and $R_t$ be defined as in Figure \ref{fig: boundary_lengths_gamma_wedge}. Then the law of $(L_t,R_t)_{t \in \R}$ can be described as follows: 
\begin{itemize}
\item The process $(L_t, R_t)_{t \geq 0}$ is a two-dimensional Brownian motion with covariances given by \eqref{eqn: covariance};
\item The process  $(L_{-t}, R_{-t} )_{t \geq 0}$ is independent of $(L_t, R_t)_{t \geq 0}$ and is a Brownian motion with the same covariance structure~\eqref{eqn: covariance}, with the additional conditioning that $R_{-t} \geq 0$ for all $t \geq 0$.  
\end{itemize}
Moreover, $(L_t, R_t)_{t\in\BB R}$ a.s.\ determines $(\bbH , h,  \eta' , 0, \infty)$ viewed as a curve-decorated quantum surface.
\end{theorem}

As in the discussion just after Theorem~\ref{thm: peanosphere disk} we can say that $(L_t, R_t)_{t\in\BB R}$ a.s.\ determines $(h, \eta')$ if we fix a canonical choice of equivalence class representative, e.g., if we require that $\mu_h(\BB D\cap\BB H)  =1$. 
In the setting of Theorem~\ref{thm: peanosphere gamma wedge}, we can explicitly identify the curve-decorated surface parametrized by $\eta'((-\infty, 0])$ and the curve-decorated surface parametrized by $\eta'([0,\infty))$. These are independent quantum wedges decorated by space-filling $\SLE_{\kappa'}(\ul\rho)$ curves; see Theorem~\ref{thm: wedge laws}. 

\begin{figure}[ht!]
\begin{center}
\includegraphics[scale=0.65]{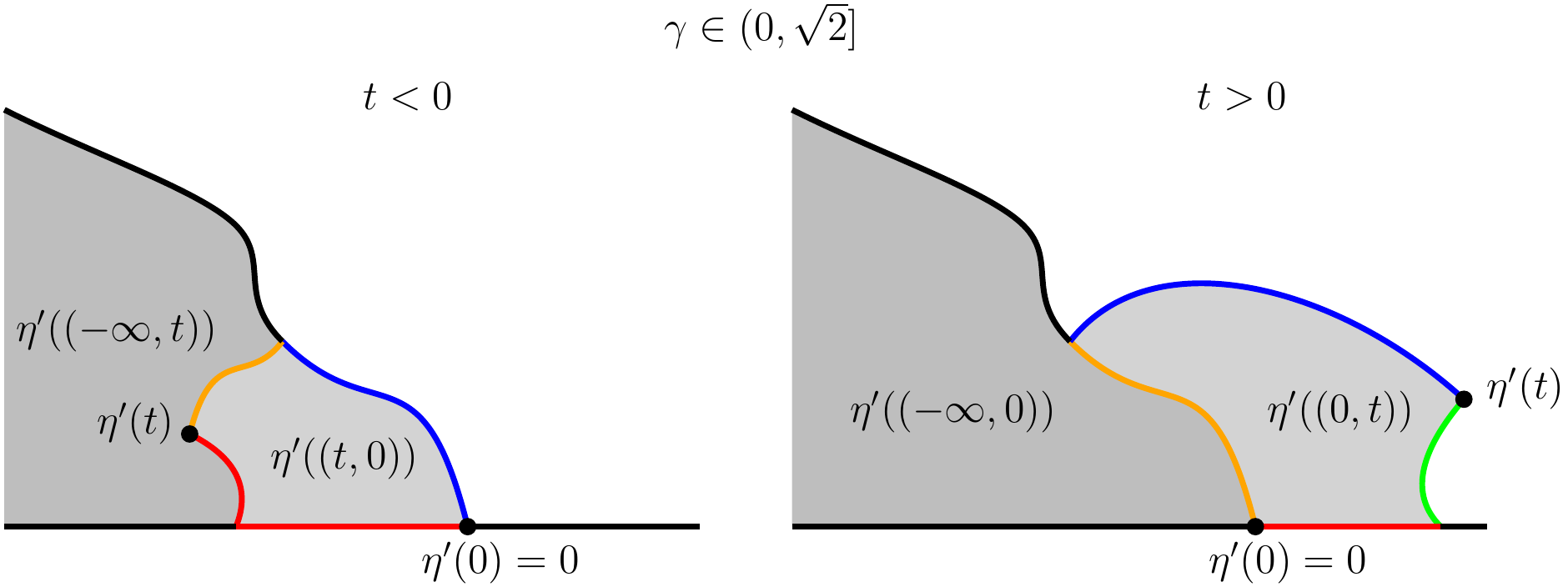}
\end{center}
\caption{\label{fig: boundary_lengths_gamma_wedge} Illustration for Theorem~\ref{thm: peanosphere gamma wedge}. For $\gamma \in (0,2)$, consider a $\gamma$-quantum wedge $(\bbH,h,0,\infty)$ with an independently drawn counterclockwise space-filling $\SLE_{\kappa'}$ curve $\eta'$ from $-\infty$ to $-\infty$ parametrized by quantum area. Note that for the case $\gamma \in (\sqrt2,2)$ (not illustrated here), the region $\eta'((-\infty, t])$ typically has multiple components joined at pinch points, and in the following description of $(L_t, R_t)$, one should take the corresponding sum over the left/right boundaries of the components. \textbf{Left:} For $t < 0$, we define $L_t = \nu_h(\text{orange}) - \nu_h(\text{blue})$, and $R_t = \nu_h(\text{red})$. \textbf{Right:} For $t > 0$, we define $L_t = \nu_h(\text{blue}) - \nu_h(\text{orange})$, and $R_t = \nu_h(\text{green}) - \nu_h(\text{red})$.}

\end{figure}

\subsection{Proof outlines and paper structure}
\label{sec-outline}

The first main result we prove in this paper is Theorem~\ref{thm: peanosphere gamma wedge}. We outline its proof below. For $\gamma \in (0,\sqrt2]$, consider a $\frac{3\gamma}{2}$-quantum wedge parametrized by $\bbH$, decorated by an independent space-filling curve $\eta'$ from $0$ to $\infty$.
\begin{itemize}
\item Theorem~\hyperref[thm: peanosphere alpha wedge]{A} gives us the boundary length process of $\eta'$ on the $\frac{3\gamma}{2}$-quantum wedge;
\item \cite[Proposition 5.5]{shef-zipper} tells us that when we zoom in on a quantum-typical boundary point of the $\frac{3\gamma}{2}$-quantum wedge, in a small neighborhood the quantum surface is close in total variation to a neighborhood of the origin in a $\gamma$-quantum wedge; 
\item Proposition~\ref{prop: zoom in Dirichlet boundary}(b) implies us that when we zoom in on a fixed (i.e. independent of $\eta'$) boundary point to the right of the origin, the curve $\eta'$ in a small neighborhood of the point is close in total variation to a counterclockwise space-filling $\SLE_{\kappa'}$. This is because a space-filling SLE$_{\kappa'}$ curve can be coupled with a GFF with certain Dirichlet boundary data in such a way that the curve is locally determined by the GFF, using the theory of imaginary geometry~\cite{ig4}.
\end{itemize}
Since the quantum wedge and $\eta'$ are independent, when we zoom in on a quantum-typical boundary point of the quantum wedge, the joint law of the quantum wedge and space-filling SLE in a small neighborhood of this point is close in total variation to the law of a $\gamma$-quantum wedge and an independent counterclockwise space-filling $\SLE_{\kappa'}$. Combining with the first ingredient above, we have the proof of Theorem~\ref{thm: peanosphere gamma wedge} in the case $\gamma \in (0,\sqrt2]$. For the regime $\gamma \in (\sqrt2,2)$, the $\frac{3\gamma}{2}$-quantum wedge is a thin quantum wedge, so it has countably many beads joined at pinch points. Nevertheless, we can carry out the same procedure by zooming in on a boundary point of one of these beads.

The proof of Theorem~\ref{thm: peanosphere disk} is similar but much more involved. Roughly speaking, we can obtain a quantum surface by conditioning a $\gamma$-quantum wedge $(\bbH, h, 0, \infty)$ to have a small ``bottleneck'' which ``pinches off'' a quantum surface with $0$ on its boundary, such that the quantum boundary lengths of this surface to the left and right of $0$ are each close to $1/2$. Of course, this procedure depends on how one defines the bottleneck. The field of the $\gamma$-quantum wedge gives us one natural way to define a bottleneck, and under this definition the quantum surface parametrized by the pinched off region is close to a unit boundary length quantum disk. Alternatively, the correlated 2D Brownian motion of Theorem~\ref{thm: peanosphere gamma wedge} gives a different way of defining a bottleneck on a curve-decorated $\gamma$-quantum wedge, and the resulting pinched-off curve-decorated surface has boundary length process close to a correlated Brownian cone excursion. One can show that these two definitions of bottlenecks are compatible in a certain sense, and by taking limits obtain Theorem~\ref{thm: peanosphere disk}.

Our proof of Theorem~\ref{thm: peanosphere disk} is in some ways similar to the proof of the quantum sphere version of the mating of trees theorem in~\cite[Theorem 1.1]{sphere-constructions} with $\gamma \in (\sqrt2,2)$. There, the authors take a space-filling $\SLE_{\kappa'}$-decorated $\gamma$-quantum cone (i.e., the quantum surface appearing in the whole-plane mating of trees theorem) conditioned to have a bottleneck pinching off a region of quantum area close to 1, and show that the quantum surface parametrized by this region is similar to a quantum sphere decorated by an independent space-filling $\SLE_{\kappa'}$. As in the present paper, the authors of~\cite{sphere-constructions} also define two bottlenecks (using the field of the quantum cone, and using the space-filling $\SLE_{\kappa'}$ exploration of the cone) and show that they are compatible. In their setting, the $\SLE_{\kappa'}$ is self-intersecting and pinches off ``bubbles''. This allows them to define the latter bottleneck by looking at the first bubble containing a target point whose boundary is ``short'', and then conditioning the area of the bubble to be close to 1. In particular, the bottleneck can be identified without reference to the \emph{exact} area of the bubble. This argument yields a mating of trees theorem for the unit area quantum sphere. 

Since we want to obtain a unit boundary length quantum disk, we instead want our $\SLE_{\kappa'}$ exploration bottleneck to pinch off a region with boundary length close to 1. However, the space-filling SLE does not pinch off bubbles so there does not seem to be a reasonable definition for the $\SLE_{\kappa'}$ exploration bottleneck that 1) does not specify the \emph{exact} quantum lengths of the left and right boundaries of the pinched-off region, 2) has a tractable left/right boundary length process in the pinched-off region and 3) is compatible with the quantum wedge bottleneck. To get around this, we forfeit 1) so when we condition on the existence of this bottleneck, we are also conditioning on the lengths of the left and right boundaries of the pinched-off region\footnote{We condition on \emph{two} boundary lengths because the most convenient description of the quantum disk involves two marked points on its boundary (corresponding to $0$ and the location of the ``pinch''), hence there are two natural marked boundary arcs (the two arcs separating the points).}. As a result we encounter significant challenges which are not present in the sphere case. 
\begin{itemize}
\item Because our definition of the $\SLE_{\kappa'}$ exploration bottleneck specifies the exact boundary lengths of the pinched-off region, we need our definition of the quantum wedge bottleneck to specify the two pinched-off boundary lengths being in exponentially short intervals in order to compare the curve-decorated quantum surfaces corresponding to the two types of bottlenecks. 
\item Given our pinched-off quantum surface, the remaining (infinite) quantum surface on the other side of the bottleneck has a law depending on two parameters (i.e., the boundary lengths of two marked arcs), unlike the quantum sphere case where the remaining surface depends only on one parameter.
\end{itemize}

In Section~\ref{section: preliminaries}, we review some preliminary facts about GFF, SLE, quantum wedges and disks, and conformal maps. In Section~\ref{section: wedge mating}, we prove Theorem~\ref{thm: peanosphere gamma wedge}. In Section~\ref{section: cone excursions}, we review the notion of a Brownian excursion in the cone, prove an approximation theorem for cone excursions, and carry out the Brownian motion calculation which leads to Theorem~\ref{thm: intro area disk}.
In Section~\ref{section: pinching off a disk}, we show that under suitable conditioning, we can ``pinch off'' a unit boundary length quantum disk from a $\gamma$-quantum wedge, and in Section~\ref{section: disk peanosphere} we compare bottlenecks to prove Theorem~\ref{thm: peanosphere disk}.
\medskip

\noindent\textbf{Acknowledgments.}  We thank Jason Miller, Minjae Park, Guillaume Remy, Scott Sheffield, and Xin Sun for helpful discussions. We thank an anonymous referee for helpful comments on an earlier version of this paper. We  also  thank the Isaac Newton Institute for Mathematical Sciences, Cambridge University, for its hospitality during the Random Geometry Workshop where part of this work was carried out.
M.A.\ was partially supported by the NSF grant DMS-1712862.
E.G.\ was partially supported by a Herchel Smith fellowship and a Trinity College junior research fellowship. 

\section{Preliminaries}\label{section: preliminaries}

In Section~\ref{subsection: GFF}, we recall properties of the GFF; in particular, the restrictions of a GFF to two open sets are almost independent when the open sets are far apart. In Section~\ref{subsection: space-filling SLE} we give a review of space-filling $\SLE_{\kappa'}$ (introduced in \cite{ig4}), and discuss properties of counterclockwise space-filling $\SLE_{\kappa'}$ starting and ending at the same point. In Section~\ref{section-lqg-intro}, we explain the definition and properties of the quantum area and boundary length measures. In Section~\ref{subsection: quantum wedges and disks}, we provide a brief explanation of quantum wedges and disks (introduced in \cite{wedges}). In particular we introduce the quantum disk with two marked points sampled from its boundary measure, conditioned on the lengths $(a,b)$ of the boundary arcs between these two marked points. In Section~\ref{sec-MOT}, we review the whole-plane version of the mating-of-trees theorem from~\cite{wedges}, discuss its connection to Theorem~\hyperref[thm: peanosphere alpha wedge]{A}, and prove a lemma to the effect that the Brownian motion in the theorem determines the curve-decorated quantum surface in a local manner. In Section~\ref{subsection: continuity of field} we provide a certain decomposition of the $(a,b)$-length quantum disk, and show this decomposition is continuous with respect to $(a,b)$. In Section~\ref{subsection: qualitative bounds on boundary lengths} we prove that if a quantum surface has small field averages, then its boundary lengths are small. Finally, in Section~\ref{subsection: conformal estimates} we prove an estimate on conformal maps which we will use when we perform cutting and gluing procedures on quantum surfaces.

\subsection{The Gaussian free field}\label{subsection: GFF}
Let $\cS = \R \times [0,\pi]$ be the strip, and $\cS_+ = \R_+ \times [0,\pi]$ the half-strip. 
It will often be convenient for us to work in $\cS$ since certain quantum surfaces (such as quantum disks and wedges) have nicer descriptions when parametrized by $\cS$. For notational convenience we will often identify $\cS, \cS_+ \subset \R^2$ with subsets of the complex plane, so for instance $\R_+ + i\pi$ refers to the half-line $\R_+ \times \{\pi\} \subset \cS$, and $\cS_+ - N$ refers to the translated half-strip $[-N, \infty) \times [0,\pi]$.

Let $D \subsetneq \C$ be a simply connected domain. For smooth compactly supported functions $f,g \in C^\infty_0(D)$ (or, more generally, smooth functions with $L^2$ gradients), we define the Dirichlet inner product
\[
(f,g)_\nabla = \frac{1}{2\pi} \int_D \nabla f (z) \cdot \nabla g(z) \ dz.
\]
Let $H^0(D)$ be the Hilbert space closure of $C^\infty_0(D)$ with respect to the Dirichlet inner product. The \emph{zero boundary GFF} $h$  is defined to be the ``standard Gaussian random variable in $H^0(D)$'', in the sense that for any choice of orthonormal basis $(f_n)$ of $H^0(D)$ we have $h \stackrel{d}{=} \sum \alpha_nf_n$ for i.i.d. $\alpha_n \sim N(0,1)$. This summation does not converge in $H^0(D)$ (so $h \not \in H^0(D)$), but a.s. converges in the space of distributions. If $\frk f$ is a harmonic function on $D$, the \emph{Dirichlet GFF with boundary data given by $\frk f$} is defined to be the sum of $\frk f$ and a zero-boundary GFF on $D$.

Next, we introduce the \emph{Neumann GFF} on $D$. A \emph{function modulo additive constant} is an equivalence class identifying functions which differ by an additive constant, i.e. $f \sim f + c$ for $c \in \R$. 
Let $\wt C^\infty(D)$ be the space of smooth functions modulo additive constant which have $L^2$ gradients, and let $H(D)$ be the Hilbert space closure of $\wt C^\infty(D)$ under the Dirichlet inner product.
The \emph{Neumann GFF modulo additive constant} is a ``standard Gaussian random variable in $H(D)$''; as with the Dirichlet case, it is a.s. not an element of $H(D)$, but makes sense as a distribution modulo additive constant. We can likewise define the \emph{whole-plane GFF} (modulo additive constant) by setting $D = \C$ in the above construction. 
The additive constant of a Neumann GFF can be fixed in various ways. For the Neumann GFF on $\cS$, we will typically fix it by requiring that its average\footnote{These averages can be defined by mapping $\cS$ to $\bbH$ by exponentiation, where Neumann GFF half-circle averages (on $\partial B_{e^t} (0) \cap \bbH$) are well defined \cite[Section 6.1]{shef-kpz}. We can also directly define the average of $h$ over $[t,t+i\pi]$ since the inverse Laplacian of the uniform measure on $[t,t+i\pi]$ has finite Dirichlet energy, cf.~\cite[Section 3.1]{shef-kpz}.} on $ [0, i \pi]$ (as defined just below) is zero.

Finally, we can also define the GFF with \emph{mixed} boundary conditions (namely, Neumann on part of $\partial D$, and Dirichlet on the rest). Let $D \subset \bbH$ be a domain, and let $I \subset \R \cap \partial D$ be a finite union of linear boundary intervals of $D$. Writing $D^\dagger = D \cup \ol D \cup I$, we have the orthogonal decomposition $H^0(D^\dagger)= H_e(D^\dagger) \oplus H_o(D^\dagger)$ into spaces of even and odd functions respectively. The mixed boundary GFF in $D$ with Neumann boundary conditions on $I$ and Dirichlet boundary conditions on $\partial D \backslash I$ is then the projection of the Dirichlet GFF on $D^\dagger$ (with reflected boundary conditions) to $H_e(D^\dagger)$.
See \cite[Section 6.2]{shef-kpz} for details. If we impose Dirichlet boundary conditions on a non-trivial segment of $\bdy D$ then the mixed boundary GFF is well-defined not just modulo additive constant.

We note that GFFs are conformally invariant; this follows from the conformal invariance of the Dirichlet inner product. 
Additionally, GFFs satisfy a Markov property, which we state below.
\begin{lemma}\label{lem-markov}
For $U \subset D$, we have a Markov decomposition for various types of GFFs $h$ on $D$
\eqb\label{eq-markov}
h = \mathfrak h + \mathring{h},
\eqe
where the fields $\mathfrak h$ and $\mathring h$ are independent, $\mathfrak h$ is a distribution on $D$ whose restriction to $U$ is a harmonic function, and $\mathring h$ is some kind of GFF on $U$ (and identically zero outside $U$). We state below several versions of this for different choices of $h,D,U$:
\begin{enumerate}[(a)]
\item Let $D\subset \C$ be a domain with harmonically nontrivial boundary, and let $U \subset D$ be an open subset of $D$ with $\mathrm{dist}(U, \partial D) > 0$. Let $h$ be a Dirichlet GFF on $D$. Then we have the decomposition~\eqref{eq-markov}, with $\mathring h$ a zero boundary GFF on $U$. 

\item Let $D \subset \C$ be a simply connected domain with harmonically nontrivial boundary, and let $I \subset \partial D$ be a smooth boundary interval. Let $U \subset D$ be an open subset such that $\partial U \cap \partial D \subset I$ is a union of finitely many boundary intervals. Let $h$ be a mixed boundary GFF on $D$, with Neumann boundary conditions on $I$ and Dirichlet boundary conditions on $\partial D \backslash I$. Then we have the decomposition~\eqref{eq-markov}, where $\mathring h$ is a mixed boundary GFF with Neumann boundary conditions on $\partial U \cap I$ and zero boundary conditions on $\partial U \cap I^c$. Moreover, the harmonic function $\mathfrak h|_{U}$ extends smoothly to $\partial U \cap I$, and has normal derivative zero there. (We allow $\partial U \cap \partial D = \emptyset$; in this case $\mathring h$ would just be a zero boundary GFF.)

\item Let $D = \C$ and let $U \subset D$ be an open set with harmonically nontrivial boundary. Let $h$ be a whole-plane GFF (modulo additive constant). Then we also have the decomposition~\eqref{eq-markov}, where $\mathring h$ is a zero boundary GFF, and we view $\mathfrak h$ as a distribution modulo additive constant.
\end{enumerate}
\end{lemma}
\begin{proof}
(a) is proved in \cite[Theorem 2.17]{shef-gff}, and (c) in \cite[Proposition 2.8]{ig4}. Roughly speaking, (a) holds because one can decompose the space $H(D)$ comprising smooth functions supported in $D$ with $L^2$ derivatives into $(\cdot, \cdot)_\nabla$-orthogonal spaces $H(D) = H(U) \oplus H^\mathrm{harm} (U)$, where the space $H^\mathrm{harm}(U)$ comprises functions which are harmonic in $U$. 

To obtain (b) from (a), we conformally change the domains so $D = \mathbb D \cap \bbH$ and $I = [-1,1]$. Write $J = \partial U \cap I$, and let $\wt U$ to be the union of $U$, $J$, and the reflection of $U$ across $\R$. Recalling the definition of $h$ as the even part of a Dirichlet GFF on $\D$, we apply the Markov property of the Dirichlet GFF on $\D$ to the sub-domain $\wt U$; taking the even part gives (b).
\end{proof}

\begin{remark}
In the above decomposition, the restriction of $\mathfrak h$ to $U$ can be interpreted as the ``harmonic extension'' of $h|_{\partial U}$ to $U$ \cite[above Proposition 3.1]{ig1}. Consequently, instead of saying ``conditioned on $\mathfrak h$, the law of $h|_U$ is given by $\mathfrak h + \mathring h$'', we sometimes instead say ``the law of $h|_U$ given $h|_{D \backslash U}$ is that of a (mixed or Dirichlet) GFF on $U$ with Dirichlet boundary values specified by $h|_{D \backslash U}$''. This reduces notational clutter. 
\end{remark}

Let $\cH_1(\cS)$ denote the subspace of $H(\cS)$ comprising functions which are constant on vertical lines viewed modulo additive constant, and let $\cH_2(\cS)$ denote the subspace of $H(\cS)$ given by functions which have mean zero on vertical lines (\emph{not} viewed modulo additive constant). By~\cite[Lemma 4.3]{wedges} we have the following decomposition of $H(\cS)$ into $(\cdot, \cdot)_\nabla$-orthogonal subspaces:
\begin{equation}\label{eqn: orthogonal decomposition}
H(\cS) = \cH_1(\cS) \oplus \cH_2(\cS).
\end{equation}
We also have the analogous decomposition $H(\cS_+) = \cH_1(\cS_+) \oplus \cH_2 (\cS_+)$. 
In this paper, we will view elements of $\cH_1(\cS)$ (resp. $\cH_1(\cS_+)$) as functions from $\R$ to $\R$ (resp. $\R_+$ to $\R$).
 
\begin{remark} \label{remark: GFF on strip}
As in \cite[Section 4.1.6]{wedges}, we point out that the above decomposition of $H(\cS)$ gives us the following explicit description of a Neumann GFF on $\cS$ normalized to have average 0 on $[0,i\pi]$, in terms of its (independent) projections onto $\cH_1(\cS)$ and $\cH_2(\cS)$:
\begin{itemize}
\item Its projection onto $\cH_1(\cS)$ is a standard two-sided linear Brownian motion with quadratic variation $2dt$, taking the value 0 at time 0;
\item Its projection onto $\cH_2(\cS)$ can be sampled as $\sum_n \alpha_n f_n$ where $(f_n)$ is an orthonormal basis of $\cH_2(\cS)$ and $(\alpha_n)$ is an i.i.d. sequence of standard Gaussians. 
\end{itemize}
We note also that the projection $h_2$ of a Neumann GFF on $\cS$ onto $\cH_2(\cS)$ has a law that is is translation invariant, i.e. $h_2(\cdot) \stackrel{d}{=} h_2(\cdot + C)$ for any $C \in \R$. This follows since an orthonormal basis of $\cH_2(\cS)$ is still an orthonormal basis after translation.
\end{remark} 

The following has essentially the same proof as \cite[Proposition 2.10]{ig4}.

\begin{proposition} \label{prop: adaptation of ig4 proposition 2.10}
Let $h$ be a Gaussian free field on $\cS_+$, having Neumann boundary conditions on $\R_+$ and $\R_+ + i\pi$, and arbitrary Dirichlet boundary conditions on $[0, i\pi]$. Then as $N \to \infty$, the total variation distance between the following two fields goes to zero:
\begin{itemize}
\item $h(\cdot + N)|_{\cS_+}$, viewed as a distribution modulo additive constant;

\item A GFF on $\cS$ with Neumann boundary conditions, restricted to $\cS_+$, and viewed modulo additive constant.
\end{itemize}
The rate of convergence depends on the choice of Dirichlet boundary conditions.
\end{proposition}

\begin{proof}
We map $\cS_+ \to \D \cap \bbH$ and $\cS \to \bbH$ via the map $z \mapsto e^{\pi i - z}$. The mixed-boundary GFF on $\D \cap \bbH$ with Neumann boundary conditions on $[-1,1]$ and Dirichlet boundary conditions on $\partial \D \cap \bbH$ can be written as the even part of the Dirichlet GFF on $\D$ with reflected boundary conditions on $\partial \D$, and similarly, the Neumann GFF on $\bbH$ (modulo additive constant) can be written as the even part of a whole plane GFF modulo additive constant \cite[Section 3.2]{shef-zipper}. Thus, it suffices to prove that the total variation distance between the following two fields goes to zero as $r \to 0$:
\begin{itemize}
\item $h|_{B(0,r)}$ viewed as a distribution modulo additive constant, for $h$ a Dirichlet GFF on $\D$ with arbitrary boundary conditions;
\item $f|_{B(0,r)}$ viewed as a distribution modulo additive constant, for $f$ a whole-plane GFF modulo additive constant restricted to $\D$. 
\end{itemize}
We will prove this by using the Markov property of the GFF together with a Radon-Nikodym derivative bound. By (a), (c) of Lemma~\ref{lem-markov} we can write each of $f$ and $h$ as a sum of a zero boundary GFF on $\D$ and an independent distribution which is harmonic in $\D$. Thus we can couple $f,h$ so that $h = f + \mathfrak g$ where $ \mathfrak g$ is a random distribution which is harmonic in $\D$ and independent of $f$. 

Consider first a fixed choice of $\mathfrak g$; WLOG choose the additive constant of $\mathfrak g$ so that $\mathfrak g(0) = 0$, so $\lim_{r \to 0} \max_{B_r(0)} |\mathfrak g| = 0$. Let $\phi$ be a smooth function compactly supported in $2\D$ and equal to $1$ on $\D$, and set $\phi_r (z) = \phi(r^{-1}z)$. Then writing $\mathfrak g_r(z) = \phi_r(z) \mathfrak g(z)$, we see that $\mathfrak g$ and $\mathfrak g_r$ agree in $B_r(0)$. Moreover, since $\nabla \phi_r = O(r^{-1})$ and $\nabla \mathfrak g|_{B_1(0)} = O(1)$, we have  $\nabla \mathfrak  g_r = \mathfrak  g \nabla \phi_r + \phi_r \nabla  \mathfrak  g = O(r^{-1} \max_{B_r(0)} |\mathfrak g| + 1)$. Since $\frk g_r$ is supported on $B_r(0)$, we get 
\[|(\mathfrak g_r, \mathfrak  g_r)_\nabla| \lesssim r^2 \cdot (r^{-1} \max_{B_r(0)} |\mathfrak g| + 1)^2 = o_r(1). \]
The Radon-Nikodym derivative of the law of $f+\mathfrak g_r$ with respect to the law of $f$ is given by \cite[Proposition 2.9]{ig4}
\[  e^{(f,\mathfrak g_r)_\nabla - \frac12 (\frk g_r , \frk g_r)_\nabla}, \]
so for fixed $\mathfrak g$ the total variation distance between $h|_{B(0,r)} = (f+\mathfrak g_r)|_{B(0,r)}$ and $f|_{B(0,r)}$ goes to zero as $r \to 0$. Since $\mathfrak g$ is independent of $f$ we have shown Proposition~\ref{prop: adaptation of ig4 proposition 2.10}. 
\end{proof}

Now, we provide an analogous proposition for GFFs with piecewise-constant Dirichlet boundary conditions, and for GFFs zoomed in at a boundary point.
\begin{proposition}\label{prop: zoom in Dirichlet boundary}
\begin{enumerate}[(a)]
\item 
For $a,b \in \R$, let $h$ be a $\GFF$ on $\cS_+$ with constant boundary conditions $a$ on $\R_+$ and $b$ on $\R_+ + i\pi$, and arbitrary Dirichlet boundary conditions on $[0,i\pi]$. Let $h^\infty$ be the $\GFF$ on $\cS$ with constant boundary conditions $a$ on $\R$ and $b$ on $\R + i\pi$, restricted to $\cS_+$. Then the law of $h(\cdot + N)|_{\cS_+}$ converges to that of $h^\infty$ in total variation as $N \to \infty$.
\item Let $D \subset \bbH$ be a simply connected domain such that $\D \cap \bbH \subset D$.
For $a \in \R$, let $h$ be a $\GFF$ on $D$ with arbitrary bounded Dirichlet boundary conditions, and constant boundary value $a$ on $[-1,1]$. Let $h^\infty$ be a Dirichlet $\GFF$ on $\bbH$ with constant boundary value $a$. Then as $d \to 0$, the total variation distance between the laws of $h|_{d\D \cap D}$ and $h^\infty|_{d \D \cap \bbH}$ goes to zero.
\end{enumerate}
\end{proposition}
\begin{proof}
\begin{enumerate}[(a)]
\item 
The proof is similar to that of Proposition~\ref{prop: adaptation of ig4 proposition 2.10}. First, by subtracting off the harmonic function $f(z) = a + \frac{\Im(z)}{\pi} (b-a)$ we reduce to the case where $a = b = 0$. Then, we map the half-strip and strip to the half-disk and upper half-plane respectively, and since we want zero boundary conditions on $[-1,1]$ and $\R$ respectively, we can sample these free fields as the odd parts of GFFs in the reflected domains $\D$ and $\C$. The rest of the proof is identical.

\item The proof is almost identical to that of the previous part. By subtracting $a$ from the boundary conditions, we can WLOG assume $a=0$. We can obtain the fields we want as the odd parts of GFFs on the domains $D \cup \overline D \cup [-1,1]$ and $\C$. The rest of the proof is identical.
\end{enumerate}
\end{proof}

\subsection{Space-filling $\SLE$}
\label{subsection: space-filling SLE}

For $\kappa' > 4$, space-filling $\SLE_{\kappa'}$ is a variant of $\SLE_{\kappa'}$ \cite{schramm0} first introduced in \cite[Section 1.2.3]{ig4}. In the regime $\kappa' \geq 8$, ordinary $\SLE_{\kappa'}$ is already space-filling, and coincides with space-filling $\SLE_{\kappa'}$ (this is immediate from the construction in~\cite{ig4}). For $\kappa' \in (4,8)$, however, ordinary $\SLE_{\kappa'}$ is \emph{not} space-filling. It bounces off of the boundary and itself, disconnecting ``bubbles'' from its target point, and subsequently never revisits these bubbles. Roughly speaking, space-filling $\SLE_{\kappa'}$ can be obtained by iteratively filling in the bubbles of ordinary $\SLE_{\kappa'}$ with space-filling $\SLE_{\kappa'}$ type curves. It is possible to give a definition of space-filling SLE$_{\kappa'}$ which is directly based on this rough description; see~\cite[Section 3.6.3]{ghs-mating-survey}. However, here we will instead give the original definition from~\cite{ig4} since it is somewhat simpler to describe and is more directly relevant to our arguments (it follows from results in~\cite{ig4} that the two definitions are equivalent).

Now, we discuss properties and the construction of space-filling $\SLE_{\kappa'}$ in the upper half-plane $\bbH$ from 0 to $\infty$. For $\kappa'\in (4,8)$, the easiest way to construct it rigorously is via imaginary geometry (for $\kappa'\geq 8$, the construction just gives ordinary SLE$_{\kappa'}$). For $\kappa' > 4$, let 
\eqb \label{eqn: ig parameter}
\kappa = \frac{16}{\kappa'} \in (0,4) ,\quad \lambda' = \frac{\pi}{\sqrt{\kappa'}} , \quad  \lambda = \frac{\pi}{\sqrt\kappa},\quad \text{and} \quad
\chi = \frac{2}{\sqrt\kappa} - \frac{\sqrt\kappa}{2} ,
\eqe
as in~\cite{ig1,ig2,ig3,ig4}. Let $h^\IG$ be a GFF in $\bbH$ with boundary conditions given by $-\lambda'$ on $\R_+$ and $\lambda'$ on $\R_-$ (here, IG stands for ``Imaginary Geometry'', and is used to distinguish the field $h^{\IG}$ from the field corresponding to an LQG surface). The space-filling $\SLE_{\kappa'}$ $\eta'$ can be coupled with $h^\IG$ so that $\eta'$ is a.s. determined by $h^\IG$ \cite[Theorem 4.12]{ig4}. 

For $z \in \bbH$ and $\theta \in \R$, one can define the \emph{flow line} $\eta$ of $h^{\IG}$ started from $z$ with angle $\theta$ as in \cite[Section 1.2.3]{ig4}, which has the informal interpretation of being the curve solving the ODE $\frac{d}{dt} \eta(t) = \exp(i h^\IG (\eta(t)) /\chi  + \theta)$ (this does not make literal sense because the distribution $h^\IG$ cannot be evaluated pointwise). This is an $\SLE_{\kappa}$-type curve which is a.s. determined by the field $h^\IG$. 

For any point $z \in \bbH$, let $\eta^L_z$ and $\eta^R_z$ be the flow lines started at $z$ with angles $\pi/2$ and $-\pi/2$ respectively. For $\kappa' \geq 8$, these two flow lines started at $z$ do not meet again, and for $\kappa' \in (4,8)$, they a.s. bounce off of each other without crossing \cite[Theorem 1.7]{ig4}. The space-filling $\SLE_{\kappa'}$  curve $\eta'$ is defined in such a way that for each $z \in \bbH$, the flow lines $\eta^L_z$ and $\eta^R_z$ are almost surely the left and right outer boundaries of the curve $\eta'$ stopped when it first hits $z$. More specifically, $\eta^L_z$ and $\eta^R_z$ divide $\bbH$ into two parts: 
\begin{enumerate}[(i)]
\item those points in complementary components whose boundary consists of a segment of either the left side of $\eta^L_z$ or the right side of $\eta^R_z$ (and possibly also an arc of $\partial \bbH$) and
\item those points in complementary components whose boundary consists of a segment of either the right side of $\eta^L_z$ or the left side of $\eta^R_z$ (and possibly also an arc of $\partial \bbH$).
\end{enumerate}
Then the closure of (i) comprises the points that $\eta'$ hits before hitting $z$, and the closure of (ii) the points that $\eta'$ hits after hitting $z$. In fact, by considering the countable collection of left and right flow lines started from $z \in \Q^2 \cap \bbH$, this property allows us to a.s. define $\eta'$ as a continuous space-filling curve which is a.s.\ determined by $h^\IG$ and which is continuous when parametrized so that it traverses one unit of Lebesgue measure in one unit of time~\cite[Theorem 1.16]{ig4}.

For $\kappa' \geq 8$, the region explored by space-filling $\SLE_{\kappa'}$ between the times when it hits two specified points is almost surely simply connected. For $\kappa' \in (4,8)$, however, the interior and the complement of this region each have countably many disk-homeomorphic components.

Space-filling SLE can be defined on other simply-connected domains by a similar procedure, and is conformally invariant (this follows from the conformal invariance of imaginary geometry constructions). We turn to the construction of whole-plane space-filling SLE, as described in \cite[Footnote 4]{wedges} (immediately before \cite[Theorem 1.9]{wedges}). We consider a whole-plane GFF viewed modulo an integer multiple of $2\pi \chi$, and draw its corresponding flow lines $\eta^L$ and $\eta^R$ from 0. If $\kappa' \geq 8$, these flow lines partition the plane into two regions each homeomorphic to $\bbH$, and we can construct a whole-plane space-filling $\SLE_{\kappa'}$ by concatenating two chordal space-filling $\SLE_{\kappa'}$'s, one in each region --- the first path is taken to run from $\infty$ to $0$ and the second from $0$ to $\infty$. If $\kappa' \in (4,8)$, the flow lines partition the plane into a countable collection of pockets, and whole-plane space-filling $\SLE_{\kappa'}$ is constructed by concatenating independently sampled space-filling $\SLE_{\kappa'}$'s in each of these pockets.

Next, we discuss \emph{counterclockwise space-filling $\SLE_{\kappa'}$} from $x$ to $x$, which appears in Theorems~\ref{thm: peanosphere disk} and~\ref{thm: peanosphere gamma wedge}. Suppose we start with a simply-connected domain $D$ with two marked boundary points $x, y \in \partial D$, so one can define space-filling $\SLE_{\kappa'}$ from $x$ to $y$. Sending $y \to x$ in the counterclockwise direction and taking a limit, we obtain counterclockwise space-filling $\SLE_{\kappa'}$ from $x$ to $x$ (see~\cite[Appendix A.3]{bg-lbm} for more details). Alternatively, if we consider the domain $D = \bbH$ and let $h^\IG$ be a Dirichlet GFF with constant boundary value $-\lambda'$, then the induced space-filling curve is counterclockwise $\SLE_{\kappa'}$ from $\infty$ to $\infty$. It can be seen that each fixed point of $\bdy D$ is a.s.\ hit exactly once by the space-filling SLE curve (although there are exceptional points which are hit twice). Say that a boundary point is ``typical" if it is hit exactly once. Almost surely, the counterclockwise space-filling $\SLE_{\kappa'}$ from $x$ to $x$ hits the typical points of $\bdy D$ in the counterclockwise order from $x$. The time-reversal of counterclockwise $\SLE_{\kappa'}$ hits boundary points in clockwise order, but this time reversal is \emph{not} clockwise $\SLE_{\kappa'}$.

In the next lemma we identify the interface between the past and future of a space-filling counterclockwise $\SLE_{\kappa'}$ when it hits a boundary point.
This will be used to identify the laws of the past and future quantum surfaces in the setting of Theorem~\ref{thm: peanosphere gamma wedge}; see Theorem~\ref{thm: wedge laws}. 
The interface is a process called $\SLE_\kappa(\ul \rho)$. This is a variant of $\SLE_\kappa$ where one keeps track of two extra marked boundary points $x_L, x_R$ to the left and right of 0 called \emph{force points}, which have weights $\ul \rho = (\rho_L, \rho_R)$ (this process is well defined for more than two force points, but we only need two here). We allow $x_L$ and $x_R$ to be $0^-$ and $0^+$; in this case we will neglect to specify $x_L, x_R$ and just write $\SLE_{\kappa'}(\rho_L;\rho_R)$. See \cite{ig1} for a construction of $\SLE_\kappa(\ul\rho)$ and its coupling with an imaginary geometry field $h^{\IG}$. For $\kappa' > 4$, one can also analogously define space-filling $\SLE_{\kappa'}(\ul\rho)$ curves; see \cite{ig4}. 

\begin{lemma}\label{lem: identify-wedges}
For $\kappa' > 4$, let $\kappa = 16/\kappa'$. Let $\eta'$ be a counterclockwise space-filling $\SLE_{\kappa'}$ on $\bbH$ from $\infty$ to $\infty$, with time parametrized so that it traverses one unit of Lebesgue measure in one unit of time and hits the origin at time 0. Then the interface $\eta'((-\infty, 0])\cap\eta'([0,\infty))$ is an $\SLE_{\kappa}(\frac{\kappa}{2}-2;-\frac{\kappa}{2})$ curve from $0$ to $\infty$. Moreover, if we condition on this interface then $\eta'|_{(-\infty,0]}$ and $\eta'|_{[0,\infty)}$ are conditionally independent and their laws are described as follows.
\begin{itemize}
\item The domain $\eta'((-\infty, 0])$ is simply connected and $\eta'|_{(-\infty, 0]}$ is a space-filling $\SLE_{\kappa'}(\frac{\kappa'}{2} - 4; 0)$ curve from $\infty$ to $0$ in  this domain.
\item If $\kappa \in (0,2]$, the domain $\eta'([0,\infty))$ is simply connected and $\eta'|_{[0,\infty)}$ is a chordal space-filling $\SLE_{\kappa'}$ from $0$ to $\infty$ in this domain. 
\item If $\kappa \in (2,4)$, the domain $\eta'|_{[0,\infty)}$ is not simply connected and $\eta'|_{[0,\infty)}$ is the concatenation of conditionally independent chordal space-filling SLE$_{\kappa'}$ curves in the connected components of the interior of $\eta'([0,\infty))$, each going between the first and last points on the boundary of the component which are hit by the interface $\eta'((-\infty, 0])\cap\eta'([0,\infty))$.
\end{itemize} 
\end{lemma}
\begin{proof}
Recall the imaginary geometry parameters \eqref{eqn: ig parameter}. Let $h^{\op{IG}}$ be the imaginary geometry GFF on $\bbH$ with constant boundary value $-\lambda'$, which is coupled with $\eta'$. Let $\eta^L_0$ be the flow line of $h^{IG}$ started at the origin with angle $\pi/2$, i.e. the flow line of the field $h^{IG} + \frac{\pi \chi}{2}$. By the definition of space-filling $\SLE_{\kappa'}$, this flow line $\eta^L_0$ is the interface $\eta'((-\infty, 0]) \cap \eta'([0,\infty))$ of the regions filled in by $\eta'$ before and after hitting 0. 

Since $h^{\op{IG}} +  \frac{\pi \chi}{2}$ has boundary value $-\lambda' +\frac{\pi \chi}{2} = -\lambda + \pi \chi$, we see from~\cite[Theorem 1.1]{ig1} that $\eta^L_0$ is a $\SLE_{\kappa} (\rho_L; \rho_R)$ curve from $0$ to $\infty$, where $\rho_L$ and $\rho_R$ satisfy
\[-\lambda(1+\rho_L) = -\lambda + \pi \chi, \quad \lambda (1 + \rho_R) = -\lambda +\pi \chi. \]
Solving, we have $\rho_L = \frac{\kappa}{2} - 2$ and $\rho_R = -\frac\kappa2$, so the interface $\eta'((-\infty, 0]) \cap \eta'([0,\infty))$ is a $\SLE_{\kappa}(\frac{\kappa}{2}-2;-\frac{\kappa}{2})$ curve from $0$ to $\infty$.

We now comment on the topologies of the regions to the left and right of $\eta^L_0$. For all $\kappa \in (0,4)$, the region $\eta'((-\infty, 0])$ has simply connected interior. For $\kappa \in (0,2]$, the region $\eta'([0,\infty))$ has simply connected interior, but for $\kappa \in (2,4)$, the region $\eta'([0,\infty))$ does not have simply connected interior. Indeed, this follows from the above description of $\eta'((-\infty, 0]) \cap \eta'([0,\infty))$ and the fact that SLE$_\kappa(\rho_L;\rho_R)$ hits $(-\infty,0)$ (resp.\ $(0,\infty)$) if and only if $\rho_L < \kappa/2-2$ (resp.\ $\rho_R < \kappa/2-2$)~\cite[Lemma 15]{dubedat-duality} (see also~\cite[Section 4]{ig1}).

Now, by looking at the boundary values of $h^{\IG}$ on $\R_-$ and on the left of $\eta^L_0$, we can determine (via~\cite[Theorem 1.1]{ig1}) the law of $\eta'|_{(-\infty, 0]}$ in the domain $\eta'((-\infty,0])$. It is a space-filling $\SLE_{\kappa'}(\widetilde \rho_L;\widetilde \rho_R)$ curve from $\infty$ to $0$ with $\widetilde \rho_L, \widetilde \rho_R$ satisfying
\begin{align}
\lambda'(1+\rho_L) = -\lambda' + \pi \chi, \qquad 
-\lambda'(1+\rho_R)  = -\lambda' ,
\end{align}
so $( \widetilde \rho_L, \widetilde \rho_R) = (\frac{\kappa'}{2} - 4, 0)$. Likewise we can solve for the $\rho$-weights of the curve $\eta'|_{[0,\infty)}$ in the domain $\eta'([0,\infty))$, and we find that it is just ordinary space-filling $\SLE_{\kappa'}$ from $0$ to $\infty$ (if $\kappa \in (2,4)$, then $\eta'$ is a concatenation of independent ordinary space-filling $\SLE_{\kappa'}$s in each bead). 
\end{proof}

\subsection{Quantum area and quantum boundary length} \label{section-lqg-intro}
Let $h = \wt h + g$ be a field on $D$, where $\wt h$ is one of the aforementioned types of GFF on $D$ in Section~\ref{subsection: GFF} and $g$ is a random continuous function on $D$. If $\wt h$ is a Neumann GFF, its additive constant must be fixed in some way. Examples include fields of quantum wedges/disks (Section~\ref{subsection: quantum wedges and disks}). Following~\cite[Proposition 1.1]{shef-kpz}, for $\gamma \in (0,2)$ we define the $\gamma$-LQG area measure by
\[\mu_h = \lim_{\eps \to 0} \eps^{\gamma^2/2} e^{\gamma h_\eps(z)} dz,  \]
where $h_\eps(z)$ is the average of $h$ on the circle $\partial B_\eps(z)$, $dz$ is Lebesgue measure on $D$, and the convergence occurs a.s.\ when $\eps \to 0$ is taken along a dyadic sequence.
We refer to~\cite[Section 3.1]{shef-kpz} for more background on the circle average process.

On a linear segment of $\partial D$ where $h$ has Neumann boundary conditions, if $g$ extends continuously to this boundary segment, we can similarly define the quantum boundary length measure $\nu_h$. In particular, following~\cite[Section 6]{shef-kpz}, we define
\[
\nu_h = \lim_{\eps \to 0} \eps^{\gamma^2/4} e^{\gamma h_\eps(x) /2} \,dx, 
\]
where $h_\eps(x)$ is the average of $h$ over the semicircle $\partial B_\eps(x) \cap D$, $dx$ is Lebesgue measure on the linear boundary segment, and the convergence occurs a.s.\ when $\eps \to 0$ is taken along a dyadic sequence. 

The measures $\mu_h$ and $\nu_h$ are a special case of a more general family of random measures constructed from log-correlated Gaussian fields called \emph{Gaussian multiplicative chaos}, which was initiated in~\cite{kahane}. See~\cite{rhodes-vargas-review,berestycki-gmt-elementary,aru-gmc-survey} for expository works on this theory.

Several observations are immediate from the limit definitions of quantum area and boundary length. Firstly, these measures are local functions of the field, i.e. the restriction of $\mu_h$ to an open set $U$ is determined by the restriction of $h$ to $U$, and the restriction of $\nu_h$ to a boundary interval $I$ is determined by the restriction of $h$ to any neighborhood of $I$. Secondly, for a constant $C$ we have $\mu_{h+C} = e^{\gamma C} \mu_h$ and similarly $\nu_{h+C} = e^{\gamma C/2} \nu_h$.

\subsection{Quantum wedges and disks}\label{subsection: quantum wedges and disks}
We recall the definitions of the quantum surfaces we will be working with. These surfaces are most easily described when parametrized by the strip $\cS = \R \times [0,\pi]$; parametrizations by other domains (like the half-plane or disk) can be obtained by applying a conformal map and using the coordinate change formula~\eqref{eqn: quantum surface defn}. For a more comprehensive introduction to these quantum surfaces, see \cite[Section 4]{wedges} and \cite[Section 2]{sphere-constructions}. 

When we work in the strip $\cS$, since the horizontal translation $T_c: z \mapsto z - c$ satisfies $Q \log |T_c'| = 0$, the quantum surface $(\cS, h)$ (possibly with marked points $\pm \infty$) is equivalent to the quantum surface $(\cS, h \circ T_c)$. Thus, we can horizontally translate the field without changing the quantum surface. We will often do so when it is notationally convenient.

Recall the decomposition \eqref{eqn: orthogonal decomposition}. In the subsections below, we define the quantum wedge and quantum disk by their projections onto the subspaces $\cH_1(\cS), \cH_2(\cS)$. Their projections onto $\cH_2(\cS)$ will be given by the projection of a Neumann GFF on $\cS$ to $\cH_2(\cS)$ (this projection is a well defined field, not just up to additive constant). We will also refer to the projection of a field to $\cH_1(\cS)$ as its \emph{field average process}. 
\subsubsection{Thick quantum wedges}\label{subsection: wedge}
\begin{definition}\label{def: thick wedge}
For $\alpha \in (-\infty, Q)$, the $\alpha$-quantum wedge $(\cS, h, +\infty, -\infty)$ is the quantum surface with $h$ sampled in the following way \cite[Remark 4.6 (second field description)]{wedges}: 
\begin{itemize}
\item Let $(X_s)_{s \in \R}$ be the projection of $h$ onto $\cH_1(\cS)$, i.e. $X_s$ is the average of $h$ on $[s,s+i\pi]$. Then $X$ is obtained by first sampling independent Brownian motions $(B_s)_{s\geq 0}$ and $(\widehat{B}_s)_{s \geq 0}$ such that
\begin{itemize}
\item $(B_s)_{s\geq 0}$ has variance 2, initial value $B_0 = 0$, and downward linear drift of $(\alpha - Q)$;

\item $(\widehat B_s)_{s\geq 0}$ has variance 2, initial value $\widehat B_0 = 0$, and upward linear drift of $(Q - \alpha)$; moreover it is conditioned to satisfy $\widehat B_s > 0$ for all $s > 0$.
\end{itemize}
Then $(X_s)_{s \in \R}$ is given by the concatenation
\[X_s = \left\{
	\begin{array}{ll}
		B_s  & \mbox{if } s \geq 0 \\
		\widehat B_{-s} & \mbox{if } s < 0.
	\end{array}
\right. \]
See Figure~\ref{fig: brownian_infinite} for a sketch.

\item Independently of $(X_s)_{ s \in \R}$, the projection of $h$ onto $\cH_2(\cS)$ is given by the projection of a Neumann Gaussian free field on $\cS$ onto $\cH_2(\cS)$. 
\end{itemize}
\end{definition}
\begin{figure}[ht!]
\begin{center}
\includegraphics[scale=0.4]{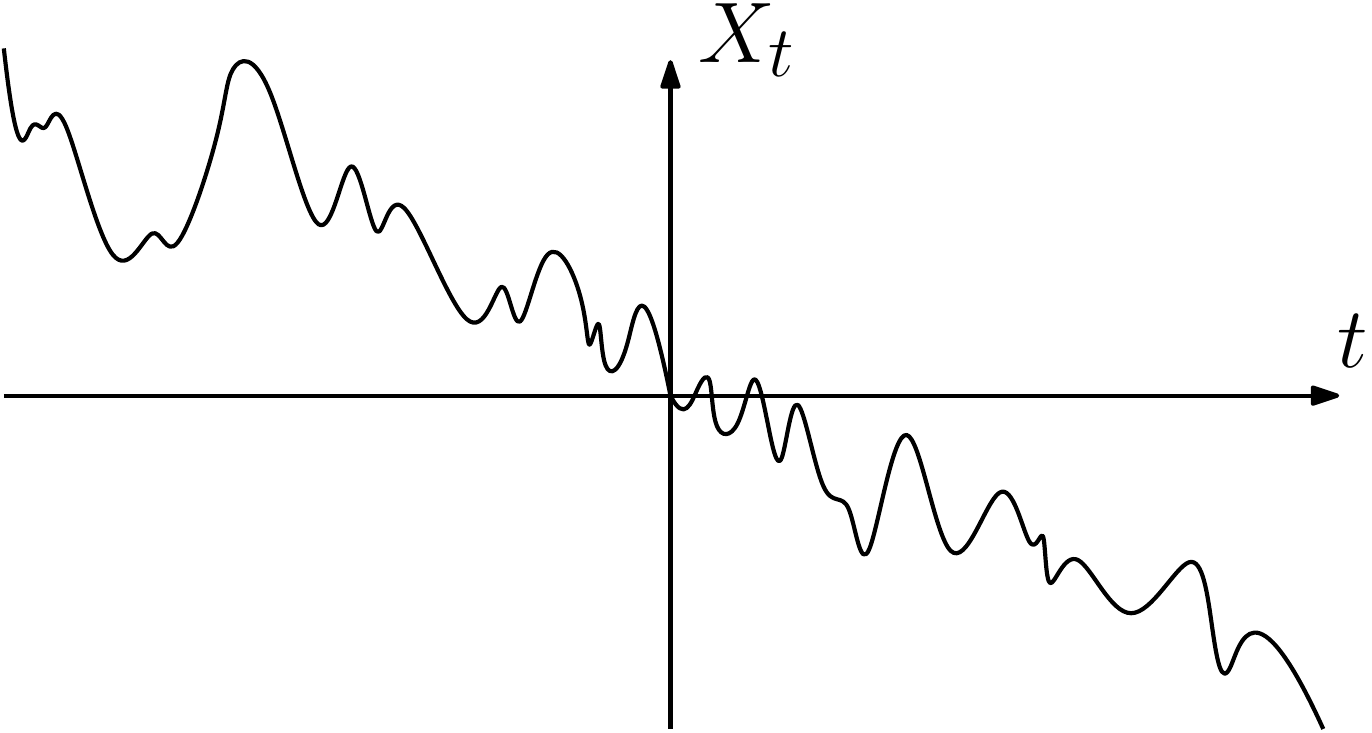}
\end{center}
\caption{\label{fig: brownian_infinite}Sketch of the field average process $X_t$ of an $\alpha$-quantum wedge with $\alpha < Q$; see Definition~\ref{def: thick wedge}. Note that for $t < 0$ we have $X_t > 0$.}
\end{figure}
A more informal way of describing the process $(X_s)_{s \in \R}$ is as a variance 2 Brownian motion with negative drift $(\alpha - Q)$, starting from $+\infty$ at time $-\infty$ and going to $-\infty$ at time $+\infty$, and with time parametrized so that it first hits $0$ at time 0. From this perspective, it is clear that quantum wedges are \emph{scale invariant} --- if $(\cS, h , +\infty,-\infty)$ is an $\alpha$-quantum wedge, then for any deterministic constant $C$ the quantum surface $(\cS, h+C, +\infty,-\infty)$ also has the law of an $\alpha$-quantum wedge. That is, there exists a random $t\in\BB R$ such that $h(\cdot+t) + C$ agrees in law with $h$.

The thick quantum wedge comes with two distinguished points $\pm \infty$. Every neighborhood of $-\infty$ has infinite quantum mass and boundary length, and the complement of any such neighborhood has finite quantum mass and boundary length. In other descriptions of the $\alpha$-quantum wedge, the process $(X_s)_{s \geq 0}$ is instead taken to have \emph{positive} drift rather than negative (the roles of $\pm \infty$ are switched around). Here, as in \cite{sphere-constructions}, we choose our notation so that the distinguished point having infinite neighborhoods is at $-\infty$, since we will usually be exploring the quantum surface from the ``infinite area'' end to the ``finite area'' end, and it seems notationally more natural for this exploration to proceed from left to right. We record below a description of the unexplored field in one such left-to-right exploration.

\begin{lemma}\label{lem-wedge-tip}
Let $(\cS, h^\cS, +\infty, -\infty)$ be a thick $\alpha$-quantum wedge ($\alpha \in (-\infty, Q)$). Fix $r \in \R$ and choose the horizontal centering of $h^\cS$ so that the left-to-right field average process first takes the value $-r$ at $0$. Then we can write
\begin{equation}\label{eqn: wedge future is GFF}
h^\cS|_{\cS_+} = \widetilde h + (\alpha - Q) \Re (\cdot) - r,
\end{equation}
where $\widetilde h$ is a Neumann GFF in $\cS$ normalized to have average 0 on $[0,i\pi]$ and restricted to $\cS_+$. 
\end{lemma}
\begin{proof}
For the case $r = 0$, this follows from Definition~\ref{def: thick wedge} and Remark~\ref{remark: GFF on strip}. The case for general $r$ is an immediate consequence of the scale-invariance of the quantum wedge. 
\end{proof}
 
\begin{remark}\label{rem: right to left}
The quantum wedge field description of Definition~\ref{def: thick wedge} is natural from the perspective of exploring the field from left to right (i.e. from the marked point with infinite neighborhoods to the marked point with finite neighborhoods). Sometimes, however, it will be useful to explore the field from right to left. We provide an equivalent definition\footnote{For $\alpha \in (-\infty, Q)$, this definition can be recovered from Definition~\ref{def: thick wedge} by horizontally shifting so that 0 is the \emph{last} time that $X_s$ hits zero, rather than the first time. For $\alpha = Q$, the field average process of a $Q$-quantum wedge is defined to be a certain reparametrized log Bessel process \cite[Section 4.4]{wedges}, and by \cite[Proposition 3.4]{wedges} and subsequent discussion this process has the description given in Remark~\ref{rem: right to left}.} of the field of an $\alpha$-quantum wedge for $\alpha \in (-\infty, Q]$. 
The $\alpha$-quantum wedge $(\cS, h, +\infty, -\infty)$ is the quantum surface such that the field $h$ has independent projections to $\cH_1(\cS)$ and $\cH_2(\cS)$ which are sampled as follows:
\begin{itemize}
\item Sample independent Brownian motions $(B_s)_{s\geq 0}$ and $(\wh B_{s})_{s \geq 0}$ such that
\begin{itemize}
\item $(\wh B_s)_{s\geq0}$ has variance 2, initial value $\wh B_0 = 0$, and downward linear drift of $(\alpha - Q)$; moreover it is conditioned to satisfy $\wh B_s < 0$ for all $s > 0$;
\item $(  B_s)_{s\geq0}$ has variance 2, initial value $ B_0 = 0$, and upward linear drift of $(Q-\alpha)$. 
\end{itemize}
Then the field average process $(X_s)_{s \in \R}$ is given by the concatenation
\[X_s = \left\{
	\begin{array}{ll}
		\wh B_s  & \mbox{if } s \geq 0 \\
		 B_{-s} & \mbox{if } s < 0.
	\end{array}
\right. \]

\item Independently of $(X_s)_{ s \in \R}$, the projection of $h$ onto $\cH_2(\cS)$ is given by the projection of a Neumann Gaussian free field on $\cS$ onto $\cH_2(\cS)$. 
\end{itemize}
\end{remark}

Roughly speaking, the next lemma can be considered an application of a ``strong Markov property'' of the GFF: we explore the thick quantum wedge $(\cS,h,+\infty, -\infty)$ from right to left until a stopping time, then, conditioned on the explored field, we deduce the conditional law of the unexplored field. Technically, however, $h$ is not a GFF, so we first use Remark~\ref{rem: right to left} to explore a small neighborhood of $+\infty$ (so the unexplored region has exactly the law of a GFF plus drift), then conclude by using the theory of local sets of the GFF (analogous to the strong Markov property of Brownian motion).

\begin{lemma}\label{lem: rest of thick wedge}
Let $(\cS, h, +\infty, -\infty)$ be an $\alpha$-quantum wedge with $\alpha \in (-\infty,Q]$ with field defined as in Remark~\ref{rem: right to left}.  
\begin{enumerate}[(a)]
\item Fix any $l>0$ and horizontally recenter $h$ so that $\nu_h(\R_+) = l$. Then conditioned on $h|_{\cS_+}$, we can sample $h|_{\cS_-}$ from its regular conditional distribution by sampling a GFF on $\cS_-$ with Neumann boundary conditions on $\R_-$ and $\R_- + i\pi$ and Dirichlet boundary conditions on $[0,i\pi]$ specified by $h|_{\cS_+}$, and adding a linear drift of $-(Q-\alpha) \Re  \cdot $. 
\item Writing $\kappa' = 4/\gamma^2$, let $\eta'$ be an independently sampled counterclockwise space-filling $\SLE_{\kappa'}$ on $\cS$ from $-\infty$ to $-\infty$. Fix any $q_1, q_2 > 0$, and let $x_1 \in \R$ and $x_2 \in \R + i\pi$ satisfy $\nu_h(\R_+ + x_1) = q_1, \nu_h(\R_+ +x_2) = q_2$. Let $U$ be the region explored by $\eta'$ between the times it hits $x_1$ and $x_2$, and horizontally recenter $h, \eta'$ so that $\inf_{z\in U} \Re(z) = 0$. Then conditioned on $h|_{\cS_+}$ and $\eta'$, we can sample $h|_{\cS_-}$ from its regular conditional distribution by sampling a GFF on $\cS_-$ with Neumann boundary conditions on $\R_-$ and $\R_- + i\pi$ and Dirichlet boundary conditions on $[0,i\pi]$ specified by $h|_{\cS_+}$, and adding a linear drift of $-(Q-\alpha) \Re  \cdot $. 
\end{enumerate}
\end{lemma}
\begin{proof}
We justify (a); the proof of (b) is similar. 

First, we express $h$ in a more convenient way, as follows. Let $g$ be the field described in Remark~\ref{rem: right to left}. Fix any $C \gg 0$ and define $\wt h = g - C$; we have $\nu_{\wt h} = e^{-\gamma C/2}\nu_{g}$ and so $\lim_{C \to \infty} \P[\nu_{\wt h}(\R_+) \leq l] = \lim_{C \to \infty} \P[e^{-\gamma C/2} \nu_g(\R_+) \leq l ] = 1$. By the scale invariance of the quantum wedge (see the discussion just after Definition~\ref{def: thick wedge}) we see that $(\cS,\wt h,+\infty, -\infty)$ is an $\alpha$-quantum wedge. That is, the quantum surfaces $(\cS,h,+\infty,-\infty)$ and $(\cS,\wt h,+\infty, -\infty)$ have the same law, so we can couple $h, \wt h$ to agree a.s.\ modulo horizontal centering. 

By Lemma~\ref{lem-markov} applied to the restriction of $\wt h$ to $\cS_+$ and Remark~\ref{rem: right to left}, conditioned on $\wt h|_{\cS_+}$, the field $\wt h|_{\cS_-}$ is a GFF with Neumann boundary conditions on $\R_-$ and $\R_- + i\pi$ and Dirichlet boundary conditions on $[0,i\pi]$ specified by $ \wt h|_{\cS_+}$, with $-(Q-\alpha)\Re (\cdot)$ added. On the event $\{\nu_{\wt h} (\R_+) \leq l\}$, set $\tau \leq 0$ to be the point such that $\nu_{\wt h}(\R_+ + \tau) = l$. Note that $h(\cdot) = \wt h(\cdot + \tau)$.

Since $ \lim_{C \to \infty} P[ \nu_{\wt h}(\R_+) \leq l] = 1$, we are done once we verify the following claim: on the event $\{ \nu_{\wt h}(\R_+) \leq l\}$, the conditional law of the field $\wt h|_{\cS_-+\tau}$ given $\wt h|_{\cS_+ + \tau}$ is that of a mixed boundary GFF with Neumann boundary conditions on the horizontal segments and specified Dirichlet boundary conditions on the vertical segment, plus a drift term $-(Q-\alpha) \Re  \cdot $. 

One way to prove this claim is to follow the proof of the strong Markov property of Brownian motion: treat $\tau$ as a right-to-left exploration stopping time, approximate $\tau$ by stopping times taking countably many values (e.g. round $\tau$ to the next multiple of $2^{-n}$), use Lemma~\ref{lem-markov} and take a limit. Alternatively, one can use the machinery of ``local sets of the GFF\footnote{
We use \cite[Lemma 3.9]{ss-contour}, which was stated for Dirichlet GFFs on simply connected domains, whereas we have a mixed-boundary GFF. To apply it, we map $\cS_-$ to $ \D\cap \bbH$ by the exponential map, and note that the mixed-boundary GFF in this domain is the even part of a Dirichlet GFF on $\D$ (see Section~\ref{subsection: GFF}). 
}'': conditioned on $\wt h|_{\cS_+}$, the random set $A = [\tau, 0] \times [0,\pi]$ is a local set of the GFF $\wt h|_{\cS_-}$ \cite[Lemma 3.9 Condition 1]{ss-contour}, and hence the conditional law of $\wt h|_{\cS_- +\tau}$ is that of a mixed-boundary GFF with Dirichlet boundary conditions on $[\tau, \tau + i\pi]$ given by $\wt h|_{\cS_+ + \tau}$ and Neumann boundary conditions elsewhere, plus the drift term \cite[Lemma 3.9 Condition 4]{ss-contour}. Since $h(\cdot) = \wt h(\cdot + \tau)$, we have shown Lemma~\ref{lem: rest of thick wedge} (a).
\end{proof}  

\subsubsection{Thin quantum wedges}\label{subsection: thin wedge}
For $\alpha \in (Q, Q+\frac\gamma2)$, the $\alpha$-quantum wedge is a random ordered sequence of surfaces (``beads''), each bead being a disk-homeomorphic quantum surface decorated by two marked points (i.e., each bead can be parametrized by $(\D,h,-1,1)$ for some field $h$). We give a brief definition below; for further discussion see \cite[Definition 4.15]{wedges}.

To sample an $\alpha$-thin quantum wedge, let $\delta = 2+\frac{2(Q-\alpha)}{\gamma}<2$, and let $Y: [0,\infty) \to [0,\infty)$ be a Bessel process of dimension $\delta$. The process $Y$ decomposes as a countable ordered collection of Bessel excursions. From each such excursion $e$ we create a disk-homeomorphic quantum surface $B_e = (\cS, h, -\infty, +\infty)$ (an $\alpha$-quantum bead), as follows:
\begin{itemize}
\item The excursion $e = (e_t)_{t \in I}$ is defined on some finite interval $I$. Take the process $\frac2\gamma \log e_t$ and apply any time-reparametrization $s = s(t)$ such that the reparametrized process $(X_s)_{s \in \R}$ has quadratic variation $2 ds$. Then the projection of $h$ to $\cH_1(\cS)$ is given by $(X_s)_{s \in \R}$;

\item The projection of $h$ to $\cH_2(\cS)$ is given by the projection of a Neumann GFF on $\cS$ to $\cH_2(\cS)$. This Neumann GFF is independent of both the Bessel process $Y$ and the other beads.
\end{itemize}
Then the $\alpha$-thin quantum wedge is given by the ordered collection $(B_e)$.

The Bessel excursion measure is infinite, and consequently so is the measure on $\alpha$-quantum beads. However, for any $l>0$ the measure of beads $(\cS, h, -\infty, +\infty)$ satisfying $\nu_h(\R)>l$ is finite. We give a partial description of the law of a bead conditioned on $\{ \nu_h(\R) > l\}$ here, which is proved in the same way as Lemma~\ref{lem: rest of thick wedge}.
\begin{lemma}\label{lem: alpha bead}
For $\alpha \in (Q, Q + \frac\gamma2)$ and $l>0$, let $(\cS, h, -\infty, +\infty)$ be an $\alpha$-quantum bead conditioned on $\{\nu_h(\R)>l\}$ and suppose that we have horizontally translated the field so that $\nu_h((-\infty, 0])=l$. Then conditioned on the field $h|_{\cS_-}$, we can sample $h|_{\cS_+}$ by sampling a GFF on $\cS_+$ with Neumann boundary conditions on $\R_+$ and $\R_+ + i\pi$, and Dirichlet boundary conditions on $[0,i\pi]$ specified by $h|_{\cS_-}$, and adding a downward linear drift of $(Q - \alpha) \Re \cdot$.
\end{lemma}

\subsubsection{Quantum disks}\label{subsection: disk}
A quantum disk is a kind of quantum surface $(\cS, \psi, +\infty, -\infty)$ decorated by two boundary points $\pm \infty$. 
There is a natural infinite measure $\mathsf{M}$ on quantum disks $(\cS, \psi, +\infty, -\infty)$, which again has independent projections to $\cH_1(\cS)$ and $\cH_2(\cS)$, the latter having the law of a Neumann GFF projected to $\cH_2(\cS)$. The field average process $(X_t)_{t \in \R}$ is obtained by ``sampling'' an excursion $e$ from the infinite excursion measure of a Bessel process of dimension $3-\frac{4}{\gamma^2}$, then setting $(X_t)_{t \in \R}$ to be the process $2\gamma^{-1} \log e$ reparametrized to have quadratic variation $2dt$ \cite[Section 4.5]{wedges}.

Informally, given the quantum surface $(\cS,\psi)$, the two marked points are chosen uniformly and independently from the quantum boundary length measure. More precisely, the law of the field $\psi$ (defined modulo horizontal translations of $\cS$) is invariant under the operation of independently sampling two boundary points $x,y$ from $\nu_\psi$, then replacing $\psi$ with the field $\psi \circ \varphi + Q \log |\varphi'|$ where $\varphi: \cS \to \cS$ is a conformal map sending $\pm \infty$ to $x,y$.

Although $\mathsf{M}$ is infinite, for reasonable notions of ``large'' we have $\mathsf{M}(\{\text{quantum disk is ``large''} \}) < \infty$. In particular, for any $l > 0$ the measure $\mathsf{M}$ assigns a finite mass to the set of quantum disks with boundary length at least $l$, so it makes sense to discuss the law of a quantum disk conditioned on boundary length being at least $l$. We can further define the regular conditional law of $\mathsf{M}$ given the probability zero event $\{\nu_\psi(\partial \cS) = l\}$.

We now introduce the $(a,b)$-length quantum disk. It comes with two marked points dividing the boundary into two segments of quantum lengths $a,b$, but given the quantum surface and one of the marked points, the other may be deterministically recovered. 
\begin{definition}
An $(a,b)$-length quantum disk $(D, \psi, x,y)$ is a quantum surface decorated by two marked boundary points, which is sampled as follows. First sample a quantum disk $(D,\psi)$ conditioned to have boundary length $a+b$, then sample $x \in \partial D$ from the boundary length measure, and finally define $y$ to be the point on $\partial D$ such that the counterclockwise arc from $x$ to $y$ has quantum length $a$.
\end{definition}
\begin{remark}
Since the marked points $\pm \infty$ of the quantum disk $(\cS, \psi, +\infty, -\infty)$ are conditionally independent uniform samples from the $\gamma$-LQG boundary length measure if we condition on $(\cS , \psi)$~\cite[Proposition A.8]{wedges}, one can equivalently define the $(a,b)$-length quantum disk by conditioning $\psi$ on the event that $\nu_\psi(\R) = a, \nu_\psi(\R + i\pi) = b$.
\end{remark}

Another way to measure the ``size'' of a quantum disk is to look at the maximum value attained by its field average process. In this case also, $\mathsf{M}$ assigns finite mass to quantum disks which are large.
\begin{proposition}\label{prop: large quantum disk}
Let $(\cS, \psi, +\infty, -\infty)$ be a quantum disk. Writing $X_s$ for the average of $\psi$ on $[s,s+i\pi]$, the event
\begin{equation}\label{eqn: large quantum disk}
E'_\beta = \{\sup_t X_t \geq -\beta \} 
\end{equation}
satisfies $\mathsf{M}(E'_\beta) < \infty$. 

Let $\psi$ be sampled from the probability measure obtained by conditioning $\mathsf{M}$ on $E_\beta'$. For notational convenience we horizontally translate the field $\psi$ so that $\inf \{ t \in \R \: : \: X_t = -\beta \} = 0$. We can then explicitly describe the conditional law of $\psi$:

\begin{itemize}
\item First sample independent Brownian motions $(B_s)_{s\geq 0}$ and $(\widehat{B}_s)_{s \geq 0}$ such that
\begin{itemize}
\item $(B_s)_{s\geq 0}$ has variance 2, initial value $B_0 = -\beta$, and downward linear drift of $(\gamma - Q)$;

\item $(\widehat B_s)_{s\geq 0}$ has variance 2, initial value $\widehat B_0 = -\beta$, and downward linear drift of $(\gamma - Q)$; moreover it is conditioned to satisfy $\widehat B_s < -\beta$ for all $s > 0$.
\end{itemize}
Then the projection $(X_s)_{s \in \R}$ of $\psi$ to $\cH_1(\cS)$ is given by the concatenation
\[X_s = \left\{
	\begin{array}{ll}
		B_s  & \mbox{if } s \geq 0 \\
		\widehat B_{-s} & \mbox{if } s < 0.
	\end{array}
\right. \]

\item Independently of $(X_s)_{ s \in \R}$, the projection of $\psi$ onto $\cH_2(\cS)$ is given by the projection of a Gaussian free field onto $\cH_2(\cS)$. 

\end{itemize}
\end{proposition}
\begin{proof}
Let $(X_s)$ be the projection of a quantum disk field conditioned on $E'_\beta$ onto $\cH_1(\cS)$. Choose any $r > \beta$, and write $\tau_{-r} = \inf \{ x \: : \: \text{average of }X \text{ on }[x,x+i\pi]\}$. From \cite[Proposition 3.4]{wedges} and \cite[Lemma 3.6]{wedges}, the law of $X(\cdot + \tau_{-r})|_{\R_+}$ conditioned on $E'_\beta$ is Brownian motion with variance 2 started at $-r$ with upward linear drift of $(Q - \gamma)$ until it hits $-\beta$, and subsequently downward linear drift of $(\gamma - Q)$. Taking $r \to \infty$ yields this description. \end{proof}

\subsection{The whole-plane mating of trees theorem}\label{sec-MOT}
In this section, we review the main mating of trees result of \cite{wedges}, explain why Theorem~\hyperref[thm: peanosphere alpha wedge]{A} is an immediate consequence, and explain that the curve-field pair can be locally recovered from the boundary length process. 

\begin{theorem}[{\cite[Theorem 1.9]{wedges}}]\label{thm-MOT}
Let $(\BB C , \wh h , 0 ,\infty)$ be a $\gamma$-quantum cone, let $\wh\eta'$ be an independent whole-plane space-filling SLE$_{\kappa'}$ from $\infty$ to $\infty$ parametrized by $\gamma$-LQG mass, and let $(\wh L ,\wh R) : \BB R\rta\BB R^2$ be its associated left/right boundary length process. Precisely, for $t > 0$, the pair $(\wh L_t, \wh R_t)$ are defined as in Figure~\ref{fig: alpha_wedge}, and analogously for $t < 0$. Then $(\wh L, \wh R)$ evolves as two-sided Brownian motion with covariances given by~\eqref{eqn: covariance}. 
Moreover, for each $t \in \R$ the joint law of $(\wh h, \wh \eta')$ as a curve-decorated quantum surface is invariant under shifting by $t$ units of time, i.e.  $(\wh h, \wh \eta') \stackrel{d}= (\wh h(\cdot +\wh \eta'(t)), \wh \eta'(\cdot + t))$. 
Finally, the quantum surfaces $(\wh\eta'((-\infty,0]), \wh h|_{\eta'((-\infty,0])} , 0 , \infty)$ and $(\wh \eta'([0,\infty)), \wh h|_{\eta'([0,\infty))} , 0 , \infty)$ are independent $\frac{3\gamma}2$-quantum wedges.
\end{theorem}

\begin{theorem}[{\cite[Theorem 1.11]{wedges}}]\label{thm-measurability}
With the notation and setup of Theorem~\ref{thm-MOT}, the pair $(\wh L, \wh R)$ a.s. determines the quantum surface $(\C, \wh h, \wh \eta', 0, \infty)$ (i.e. determines $(\wh h, \wh \eta')$ up to conformal automorphisms of $\C$ fixing $0$ and $\infty$). 
\end{theorem}
We note again that the above results were shown in \cite{wedges}, except for the explicit form of the covariances \eqref{eqn: covariance} for $\gamma \in (0,\sqrt 2)$ which was later established in \cite{kappa8-cov}. Theorem~\hyperref[thm: peanosphere alpha wedge]{A} follows from the above by restricting the curve-decorated quantum cone to the region explored by $\wh \eta'([0,\infty))$.

The  following  lemma statement is implicit in the proof of the measurability statement of Theorem~\ref{thm-measurability}, as given in~\cite[Section 9]{wedges}. 
However, for the sake of clarity we will deduce the lemma from the two theorems above. 

\begin{lemma}\label{lem-measurability}
Consider the setup of Theorem~\hyperref[thm: peanosphere alpha wedge]{A}, where we have a $\frac{3\gamma}{2}$-quantum wedge with field $h$ decorated by a space-filling $\SLE_{\kappa'}$ $\eta'$, and we write $(L_t, R_t)_{t\in [0,\infty)}$ for the boundary length process (in the case when $\gamma \in (\sqrt 2 , 2)$, so the wedge is thin, $h$ is an ordered sequence of random distributions, one for each bead of the surface). 
	For $a , b \in [0,\infty] $ with $a<b$, the restricted left/right boundary length process $(L-L_a , R-R_a)|_{[a,b]}$ a.s. determines the curve-decorated quantum surface $(\eta'([a,b]) , h|_{\eta'([a,b])} , \eta'|_{[a,b]} )$.
\end{lemma}

We emphasize that $(L-L_a , R-R_a)|_{[a,b]}$ only determines $(\eta'([a,b]) , h|_{\eta'([a,b])} , \eta'|_{[a,b]} )$ as a curve-decorated quantum surface, i.e., modulo coordinate changes of the form~\eqref{eqn: quantum surface defn}. The Brownian motion increment $(L-L_a , R-R_a)|_{[a,b]}$ \emph{does not} a.s.\ determine $\eta'([a,b])$.

\begin{proof}[Proof of Lemma~\ref{lem-measurability}]
Recall the relationship between Theorem~\hyperref[thm: peanosphere alpha wedge]{A} and Theorem~\ref{thm-MOT} discussed above. Let $(\BB C , \wh h , 0 ,\infty)$ be a $\gamma$-quantum cone, $\wh\eta'$ an independent whole-plane space-filling SLE$_{\kappa'}$ from $\infty$ to $\infty$ parametrized by $\gamma$-LQG mass, and $(\wh L ,\wh R) : \BB R\rta\BB R^2$ the associated left/right boundary length process. It suffices to show that if $0\leq a < b \leq \infty$ then  $(\wh L-\wh L_a , \wh R-\wh R_a)|_{[a,b]}$ a.s.\ determines $(\wh\eta'([a,b]) , h|_{\wh\eta'([a,b])} , \wh\eta'|_{[a,b]})$. 

For this purpose, to lighten notation we define $\wh Z := (\wh L , \wh R)$ and for $a,b\in\BB R \cup \{-\infty,\infty\}$ with $a <b$ we define the curve-decorated quantum surface
\eqbn
\mcl W_{a,b}^* :=  (\wh\eta'([a,b]) , h|_{\wh\eta'([a,b])} , \wh\eta'|_{[a,b]}) .
\eqen
\sloppy
By Theorem~\ref{thm-MOT}, for $a\in\BB R$ the quantum surfaces $(\wh\eta'((-\infty,a]) , \wh h|_{\eta'((-\infty,a])} , \wh\eta'(a) , \infty  )$ and $(\wh\eta'|_{[a,\infty)} , \wh h|_{\eta'([a,\infty)} , \wh\eta'(a) , \infty )$ are independent. 
By the construction of whole-plane space-filling SLE described in Section~\ref{subsection: space-filling SLE} and the conformal invariance of SLE, it follows that the curve-decorated quantum surfaces $\mcl W_{-\infty,a}^*$ and $\mcl W_{a,\infty}^*$ are independent.  

\fussy
It is clear that $\mcl W_{-\infty,a}^*$ (resp.\ $\mcl W_{a,\infty}^*$) a.s.\ determines $(\wh Z-\wh Z_a)|_{(-\infty , a]}$ (resp.\ $(\wh Z-\wh Z_a)|_{[a,\infty)}$).  Since $\wh Z$ a.s.\ determines both of the above two curve-decorated quantum surfaces (Theorem~\ref{thm-measurability}), it follows that  $(\wh Z-\wh Z_a)|_{(-\infty , a]}$ (resp.\ $(\wh Z-\wh Z_a)|_{[a,\infty)}$) a.s.\ determines  $\mcl W^*_{-\infty,a}$ (resp.\ $\mcl W^*_{a,\infty}$). 

For $a,b\in\BB R$ with $a<b$, the curve-decorated quantum surface $\mcl W^*_{a,b}$ is a.s.\ determined by each of the pairs $\mcl W^*_{-\infty,b}$ and $\mcl W^*_{a,\infty}$. Consequently, the previous paragraph implies that $\mcl W^*_{a,b}$ is a.s.\ determined by each of $(\wh Z-\wh Z_b)|_{(-\infty , b]}$  and $(\wh Z-\wh Z_a)|_{[a,\infty)}$.
The intersection of the $\sigma$-algebras generated by these two restricted Brownian motions is the $\sigma$-algebra generated by $(\wh Z-\wh Z_a)|_{[a,b]}$. Consequently, $\mcl W^*_{a,b}$ is a.s.\ determined by $(\wh Z-\wh Z_a)|_{[a,b]}$.  
\end{proof}

\subsection{Regularity of the $(a,b)$-length quantum disk in $a,b$} 
\label{subsection: continuity of field}
In this section, we modify the procedure of Proposition \ref{prop: large quantum disk} to give an alternate description of the field of a quantum disk conditioned on $E'_\beta$, and show that when we condition on the side lengths of the field being $(a,b)$, this description of the field is in some sense continuous in $(a,b)$. We also show that $\P[E'_\beta \mid (\nu_\psi(\R), \nu_\psi(\R + i\pi)) = (y_1,y_2)] = 1-o_\beta(1)$ uniformly for $y_1,y_2 \in [\frac12, 1]$. Combining these we deduce that if $\psi$ is a quantum disk field conditioned on $(\nu_\psi(\R) , \nu_\psi(\R + i\pi))=(a,b)$ (with $a+b> \frac12$) and $\sigma \in \R$ satisfies $\nu_\psi(\R_+ + \sigma) + \nu_\psi(\R_+ + i\pi + \sigma) = \frac12$, then for any $N$ the law of the field $\psi(\cdot + \sigma)|_{\cS_+ -N}$ is continuous w.r.t. total variation distance as we vary $(a,b)$.

Proposition \ref{prop: large quantum disk} describes the law of the field $\psi$ of a quantum disk $(\cS, \psi, +\infty, -\infty)$ conditioned on $E'_{\beta}$. Write $(X_s)_{s\in \R}$ for the field average process. Recall (Remark~\ref{remark: GFF on strip}) that if we let $\{f_j\}_{j \in \N}$ be an orthonormal basis for $\cH_2(\cS)$, then we can sample the projection of $\psi$ to $\cH_2(\cS)$ as $\sum_{j\in \N} \alpha_j f_j$, where $\{\alpha_j\}_{j \in \N}$ are i.i.d. standard Gaussians. We can take $f_1, f_2 \in \cH_2(\cS)$ to be smooth functions supported on $[- 3, 0 ] \times [0,\frac\pi2]$ and $[- 3, 0 ] \times [\frac\pi2,\pi]$ respectively, such that 
\begin{itemize}
\item $\| f_1\|_\nabla = 1$, and $f_1$ is nonnegative on $[-3,0]$ and strictly positive on $[-2,-1]$.

\item $\|f_2\|_\nabla = 1$, and $f_2$ is nonnegative on $[-3 +i\pi,i\pi]$ and strictly positive on $[-2+i\pi,-1+i\pi]$.
\end{itemize}
We may then sample the projection of $\psi$ onto $\cH_2(\cS)$ as $f + \alpha_1 f_1 + \alpha_2 f_2$, where $f$ is a random distribution on $\cS$, $\alpha_1,\alpha_2$ are standard Gaussians, and $f, \alpha_1, \alpha_2$ are mutually independent. Now, the field $\psi$ conditioned on $E'_\beta$ is given by 
\begin{equation}\label{eqn: field decomposition}
\psi = X_{\Re \cdot} + f +\alpha_1 f_1 + \alpha_2 f_2.
\end{equation}
Note that in future uses of this decomposition, we will horizontally recenter the field, so that $\tau_{-\beta} := \inf \{ u \: : \: \text{average of }\psi \text{ on } [u,u+i\pi] \text{ is } -\beta\}$ is not necessarily zero. After horizontally translating, the functions $f_1,f_2$ will be compactly supported on $[\tau_{-\beta} - 3, \tau_{-\beta}] \times [0,\pi]$ instead. 

\eqref{eqn: field decomposition} allows us to tweak the field by varying $\alpha_1, \alpha_2$, while keeping $(X_t)_{t \in \R}$ and $f$ fixed. Note that for fixed $(X,f)$, the side lengths $\nu_\psi(\R)$ and $\nu_\psi(\R_+)$ are strictly increasing in $\alpha_1, \alpha_2$ respectively. For any particular choice of $X,f$, the quantum length pair $(\nu_\psi(\R), \nu_\psi(\R + i\pi))$ has a conditional density $d^\beta_{\mathrm{disk}}(\cdot, \cdot \mid X,f)$ with respect to Lebesgue measure in $\R_+^2$. 

\begin{proposition} \label{prop: continuity of field}
With this decomposition of $\psi$ conditioned on $E'_\beta$, let $\cL^{\beta, a,b}_{\mathrm{disk}}$ be the conditional law of $(X,f)$ given $E'_\beta \cap \{\nu_\psi(\R) = a, \nu_\psi(\R) = b\}$. Then for fixed $\beta$, $\cL^{\beta, a,b}_{\mathrm{disk}}$ is continuous in $(a,b)$ w.r.t.\ total variation distance.
\end{proposition}
\begin{proof}
Let $\psi$ be the field of a quantum disk conditioned on $E'_\beta$, and let $d^\beta_\mathrm{disk}(\cdot, \cdot)$ be the density of $(\nu_\psi(\R), \nu_\psi(\R + i\pi))$ with respect to Lebesgue measure. Then we have the Radon-Nikodym derivative
\[\frac{d \cL^{\beta, a',b'}_{\mathrm{disk}}}{d \cL^{\beta, a,b}_{\mathrm{disk}}} (X,f) = \frac{d^\beta_\mathrm{disk}(a',b' \mid X,f)}{d^\beta_\mathrm{disk}(a, b \mid X,f)} \cdot \frac{d^\beta_\mathrm{disk}(a,b)}{d^\beta_\mathrm{disk}(a',b')}.\]
Both $d^\beta_\mathrm{disk}(\cdot, \cdot)$ and $d^\beta_\mathrm{disk}(\cdot, \cdot \mid X,f)$ are continuous functions, so for any fixed $(X,f)$, we have $\frac{d \cL^{\beta,a',b'}_\mathrm{disk}}{d \cL^{\beta,a,b}_\mathrm{disk}} (X,f) \to 1$ as $(a',b') \to (a,b)$.  
\end{proof}

Next, write $\P^{y_1, y_2}_\mathrm{disk}$ for the law of a quantum disk field $\psi$ conditioned on $\{(\nu_\psi(\R), \nu_\psi(\R + i\pi)) = (y_1, y_2)\}$. As $\beta \to \infty$, the event $E'_\beta$ is of uniformly high probability w.r.t. $\P^{y_1,y_2}_\mathrm{disk}$ for all $y_1, y_2 \in [\frac12,1]$:
\begin{proposition}\label{prop: E' uniformly likely for all y1,y2}
With $E'_\beta$ as in \eqref{eqn: large quantum disk}, we have
\[\P^{y_1,y_2}_\mathrm{disk} [E'_\beta] \geq 1 - o_\beta(1) \quad \text{uniformly over all } y_1, y_2 \in \left[ \frac12, 1 \right],\] 
\end{proposition}
\begin{proof}
Let $d_\mathrm{disk}(\cdot, \cdot)$ be the probability density of the side lengths $(\nu_\psi(\R), \nu_\psi(\R + i\pi))$ of a quantum disk conditioned on $\{ \nu_\psi(\R), \nu_\psi(\R + i\pi) \in \left[ \frac12, 1 \right]\}$. As above, let $d^\beta_\mathrm{disk}(\cdot, \cdot)$ be the probability density of $(\nu_\psi(\R), \nu_\psi(\R + i\pi))$ conditioned on $E'_\beta$.

For any $\delta > 0$, for each point $(y_1, y_2) \in [1/2, 1]^2$ we can choose some sufficiently large $\beta$ so that $\P^{y_1,y_2}_\mathrm{disk}[E'_\beta] > 1-\delta$, and choose a ball $B\ni(y_1,y_2)$ so that $d_\mathrm{disk}, d^\beta_\mathrm{disk}$ are close to constant in $B$ (so for all $(y_1',y_2' )\in B$, we have $\P^{y_1',y_2'}_\mathrm{disk}[E'_\beta] > 1-2\delta$). Using the compactness of the square $\left[\frac12,1\right]^2$, we can cover the square by some finite collection $B^1, \dots, B^N$ (with corresponding values $\beta^1, \dots, \beta^N$), and conclude that for $\beta = \max_j \beta^j$ we have $\P^{y_1,y_2}_\mathrm{disk}[E'_\beta] > 1- 2\delta$ for all $y_1, y_2 \in [\frac12,1]$. 
\end{proof}

\begin{corollary}\label{cor: Psi continuous}
Fix $N>0$ and let $\sigma \in \R$ be the unique number such that $\nu_\psi(\R_+ + \sigma) + \nu_\psi(\R_+ + i\pi + \sigma) = \frac12$. Then 
\eqb \label{eqn: E' and tau likely}
\P^{y_1,y_2}_\mathrm{disk}[E'_\beta \cap \{\tau_{-\beta} < \sigma - N\} ] > 1-o_\beta(1) \quad \text{ uniformly for }y_1, y_2 \in [\frac12, 1].
\eqe
Consequently, for $y_1, y_2 \in [\frac12,1]$, the law of $\psi(\cdot + \sigma)|_{\cS_+ - N}$ sampled from $\P^{y_1,y_2}_\mathrm{disk}$ is continuous in $(y_1,y_2)$ w.r.t.\ the total variation distance.
\end{corollary}
\begin{proof}
For fixed $N$ and for each $y_1, y_2 \in [\frac12, 1]$, we have $\P^{y_1,y_2}_\mathrm{disk}[\tau_{-\beta} < \sigma - N \mid E'_\beta] = 1-o_\beta(1)$, so by Proposition~\ref{prop: continuity of field} and the compactness of $[\frac12,1]^2$, we have $\P^{y_1,y_2}_\mathrm{disk}[\tau_{-\beta} < \sigma - N \mid E'_\beta] > 1-o_\beta(1)$ uniformly for all $y_1,y_2 \in [\frac12,1]$. Combining this with Proposition~\ref{prop: E' uniformly likely for all y1,y2} yields~\eqref{eqn: E' and tau likely}.

On the event $E'_\beta \cap \{(\nu_\psi(\R), \nu_\psi(\R + i\pi)) = (y_1,y_2) \}\cap \{\tau_{-\beta} < \sigma - N\}$, the field $\psi(\cdot + \sigma)|_{\cS_+ -N }$ is a function of $(X,f)$, so by Proposition~\ref{prop: continuity of field} we obtain the second assertion of Corollary~\ref{cor: Psi continuous}.
\end{proof}

\subsection{Bounds on quantum lengths}\label{subsection: qualitative bounds on boundary lengths}
In our subsequent arguments, we will want to say that if the field averages of various quantum surfaces are small, then their quantum boundary lengths are small with high probability. The results of this section will be used in the proofs of Lemmas~\ref{lem: conditioned on start, field looks like disk} and~\ref{lem: choose q1, q2}.

Let $\wt h$ be a Neumann GFF on $\cS$ restricted to $\cS_+$ and normalized so its average on $[0,i\pi]$ is 0. For $r > 0$, define 
 \eqb \label{eq-defh}
h = \wt h + (\gamma - Q) \Re (\cdot)  - r, 
\eqe
and let $E'_{r, \beta}$ be the event that the field average process of $h$ attains the value $-\beta$. 
\begin{lemma}[Variant of {\cite[(A.10)]{wedges}}]\label{wedges-lem}
Fix $p \in (0,\frac4{\gamma^2})$ and $R= [0,S] \times [0,\pi]$ for some $S>0$. Almost surely, there is a random constant $C = C(p, \wt h|_R)$ such that, uniformly over $r > \beta > 0$, we have
\[\E\left[ \left( \frac{\nu_h(\R_+) + \nu_h(\R_+ + i\pi)}{e^{-\gamma \beta/2}}\right)^p \Big| E'_{r,\beta}, \wt h|_R\right]  < C. \]

For the quantum disk $(\cS, \psi, +\infty, -\infty)$ conditioned on $E'_{\beta}$ (defined in \eqref{eqn: large quantum disk}), there is a constant $C = C(p)$ such that uniformly over $\beta$, 
\[\E\left[ \left( \frac{\nu_\psi(\R) + \nu_\psi(\R +  i\pi)}{e^{-\gamma \beta/2}}\right)^p\Big| E'_{\beta}\right]  < C. \]
\end{lemma}

We sketch the proof for the GFF case (the quantum disk case is almost identical); see Lemmas A.4, A.5 and A.6 of \cite{wedges} for details. Let $(X_t)_{t \geq 0}$ and $\wh h$ be the projections of $h$ to $\cH_1(\cS_+)$ and $\cH_2(\cS_+)$. We establish bounds for each projection, and combine them to conclude. 

Firstly, conditioned on $\wt h|_R$ the process $(X_t)_{t \geq S}$ is Brownian motion with initial value $X_S = - r + \nu_{\wt h}([S,S+i\pi])$ and with constant negative drift $-(Q-\gamma)$ and variance 2. Further conditioning on $E'_{r,\beta}$ amounts to further conditioning on $\sup_t X_t \geq -\beta$; under this conditioning $(X_t)_{t \geq S}$ evolves first as variance 2 Brownian motion with upward drift of $(Q-\gamma)t$ until it hits $-\beta$, then as  variance 2 Brownian motion with downward drift $-(Q-\gamma)t$ \cite[Lemma 3.6]{wedges}. Thus,  for a constant $C = C(p, \wt h|_R)$ we have uniformly over all $r>\beta > 0$ that
\eqb\label{eq-field-average-moment}
\E\left[ \left( \int_S^\infty e^{\frac\gamma2 X_t} \ dt \right)^p \Big| E'_{r,\beta} , \wt h|_R \right] \leq C e^{-\frac{\gamma \beta p}{2}}.
\eqe

With the field average process bound~\eqref{eq-field-average-moment}, since $E'_{r,\beta}$ is independent of $\wh h$, we only need the following to conclude the proof:
\eqb\label{eq-lateral-moment}
\E \left[\nu_{\wh h} ([u,u+1] \times \{0,\pi\})^p \:\big|\: \wt h|_R \right] <C \quad \text{ for all } u \geq S.
\eqe
For $u \leq S+1$, this follows from the a.s. inequality
\[\E\left[ \nu_{\wh h}([S,S+2] \times \{0,\pi\})^p \:\Big|\: \wt h|_R\right] < \infty, \]
which holds since the expectation of the above quantity over $\wt h|_R$ is a $p$-th GMC moment which is finite by \cite[Theorem 2.11]{rhodes-vargas-review}.

We turn to the case\footnote{The argument in the $u \leq S+1$ case gives finiteness of the expectation in~\eqref{eq-lateral-moment} for each choice of $u$ and $\wt h|_R$, but does not guarantee a uniform bound across $u$ for each $\wt h|_R$, so we need an alternative argument.} $u \geq S+1$. 
Let $h'$ be the projection to $\cH_2(\cS_+ + S)$ of a GFF on $\cS_+ +S$ with zero boundary conditions on $[S, S+i\pi]$ and Neumann boundary conditions elsewhere. We claim that
\eqb\label{eq-lateral-moment-zero}
\E \left[\nu_{h'} ([u,u+1] \times \{0,\pi\})^p \right] <C \quad \text{ for all } u \geq S+1.
\eqe
Indeed, the Markov property of the GFF allows us to sample $h'$ via $\mathring h = h' + \mathfrak h$; here $\mathring h$ is a Neumann GFF on $\cS_+ + S$ projected onto $\cH_2(\cS_+ +S)$, and $\mathfrak h$ is a random harmonic function on $\cS_+ + S$ with zero normal derivatives on $\R_+ + S$ and $\R_+ + S + i\pi$, such that $\mathfrak h$ and $h'$ are independent. Take any $p' \in (p, 4/\gamma^2)$. The translation invariance of the Neumann GFF on $\cS$ implies that $\E \left[(\nu_{\mathring h} ([u,u+1] \times \{0,\pi\})^{p'} \right]$ is a finite constant independent of $u$. Also, the Borell-TIS inequality gives a lognormal tail bound on $\sup_{\cS_+ + S+1} |\mathfrak h|$. An application of H\"older's inequality then yields~\eqref{eq-lateral-moment-zero}.

To conclude the proof of~\eqref{eq-lateral-moment} for $u \geq S+1$, the Markov property of the GFF tells us that conditioned on $\wt h|_R$ we can write $\wh h|_{\cS_+ +S}$ as a field with the law of $h'$, plus a harmonic function depending only on $\wt h_R$ with zero normal derivative on $\R_+ + S$ and $\R_+ + S + i\pi$. The maximum principle says that this harmonic function is uniformly bounded on $\cS_+ +S+1$, so~\eqref{eq-lateral-moment} for $u \geq S+1$ follows from~\eqref{eq-lateral-moment-zero}. 

Combining~\eqref{eq-field-average-moment} and~\eqref{eq-lateral-moment} gives Lemma~\ref{wedges-lem} in the GFF case.

\begin{corollary}\label{cor-large-boundary}
Let $(\cS, \psi, +\infty, -\infty)$ be a quantum disk conditioned on $E'_\beta$ (defined in \eqref{eqn: large quantum disk}). Then for any $p < \frac{4}{\gamma^2}$ there is a constant $C = C(p)$ such that for all $\beta  >0$, 
\[\P\left[ \nu_\psi(\R_+) + \nu_\psi(\R_+ + i\pi) > \frac12  \mid E'_{\beta} \right] \leq C e^{-\frac{\gamma \beta p}2}. \] 
\end{corollary}
\begin{proof}
This follows from Lemma~\ref{wedges-lem} and Markov's inequality. 
\end{proof}

\begin{corollary}\label{cor-medium-boundary}
Let $(\cS, \psi, +\infty, -\infty)$ be a quantum disk conditioned on $E'_{\beta}$ (defined in \eqref{eqn: large quantum disk}). For $r > \beta$, let $\tau^\psi_{-r} := \inf \{ t \in \R\: : \: \text{average of } \psi \text{ on } [t, t+i\pi] \text{ is }-r \}$.
Then 
\[
\lim_{r \to \infty} \P\left[\nu_\psi((-\infty, \tau_{-r}^\psi] \times \{0,\pi\}) \geq e^{-\gamma r / 4} \right] = 0.
\]
Similarly, with the setup of Lemma~\ref{wedges-lem}, 
let $\tau^h_{-r/2} = \inf \{ t > 0 \: :\: \text{average of } h \text{ on } [t,t+i\pi] \text{ is } -r/2\}$. Almost surely the field $\wt h|_R$ is such that when we condition on $\wt h|_R$, we have the a.s. limit
\[\lim_{r \to \infty} \P[ \nu_h( [0,\tau_{-r/2}] \times \{0,\pi\}) > e^{-\gamma r / 8} \mid E'_{r,\beta}, \wt h|_R ] = 0. \]
\end{corollary}
\begin{proof}
The first inequality for the case $\beta = r$ follows immediately from Lemma~\ref{wedges-lem} and Markov's inequality. But by the Markov property of Brownian motion the law of $\psi(\cdot + \tau^\psi_{-r})|_{\cS_+}$ conditioned on $E'_{\beta}$ does not depend on the value of $\beta < r$. The second inequality follows by the same argument. 
\end{proof}

\subsection{Distortion estimate for conformal maps on the strip}\label{subsection: conformal estimates}
To prove Theorem \ref{thm: peanosphere disk}, we will perform some cutting and gluing operations on quantum surfaces. The purpose of this section is to bound the effect of these operations.

Let $\bbH$ and $\widehat \C = \C \cup \{\infty\}$ be the half plane and the Riemann sphere respectively. We say that $K \subset \bbH$ (resp. $K \subset \widehat \C$) is a \emph{hull} if $K$ is bounded and has simply connected complement w.r.t. $\bbH$ (resp. $\widehat \C$). We may identify the strip $\cS$ with $\bbH$ via the map $z \mapsto e^z$, and say a set $K \subset \cS$ is a hull if $\exp(K) \subset \bbH$ is a hull in $\bbH$. Let $\cQ$ be the infinite cylinder; concretely, define $\cQ = \R \times [-\pi,\pi]$ with the lines $\R-i\pi$ and $\R + i\pi $ identified. Define $K \subset \cQ$ to be a hull if $\exp(K) \subset \widehat \C$ is a hull in $\widehat \C$.  

\begin{lemma}[{Variant of \cite[Lemma 2.4]{sphere-constructions}}]\label{lem: distortion estimate}
There exist universal constants $C_1, C_2 > 0$ such that the following holds. Suppose that $K_1 \subset \cS_-$ and $K_2 \subset \cS$ are hulls, and $\varphi: \cS \backslash K_1 \to \cS \backslash K_2$ is a conformal map with $|\varphi(w) - w| \to 0$ as $w \to +\infty$. Then
\begin{equation}\label{eq-distortion}
|\varphi(w) - w| \leq C_2 \exp (-\Re (w))  \text{ for all } w \in \cS_+ + C_1,
\end{equation}
and
\begin{align*}
|\varphi'(w)|^{-1}, |\varphi'(w)|, |\varphi''(w)| \leq C_2 \text{ for all }w \in \cS_+ + C_1.
\end{align*}
\end{lemma}
\begin{proof}
Exponentiation sends $\cS_+$ to $\bbH \backslash \D$ and $\cS$ to $\bbH$. Consider the map $G: \bbH \backslash \D \to \bbH$ defined by $\exp \circ \varphi \circ \log$; this admits a power series expansion $G(z) = z + \sum_{n=1}^\infty a_nz^{-n}$ (the first two coefficients are $a_{-1} = 1$ and $a_0 = 0$ since $\lim_{w \to +\infty} |\varphi(w) - w| = 0$). By the Schwarz reflection principle, $G$ extends to a map $\C \backslash \D \to \C$ with the same power series expansion. Consequently, the area theorem of complex analysis tells us that $\sum_{n=1}^\infty n |a_n|^2 \leq 1$.
Cauchy-Schwarz then yields for $|z| > 2$ that 
\[  \left|\sum_{n=1}^\infty a_nz^{-n} \right| \leq \left(\sum |a_n|^2\right)^{1/2} \left(\sum_{n = 1}^\infty |z|^{-2n}\right)^{1/2} \lesssim |z|^{-1}. \]
Thus, for $w \in \cS_+ + \log 2$ we have, using $| \log (1 + u)| = O(|u|)$ for small $|u|$,
\[|\varphi(w) - w| = |\log G(e^w) - \log e^w| \lesssim \left| \sum_{n=1}^\infty a_n e^{-nw} \right| \lesssim e^{-\Re w},\]
so we have shown~\eqref{eq-distortion}. To obtain the bounds on $1/|\varphi'|$, $|\varphi'|$ and $|\varphi''|$, we combine~\eqref{eq-distortion} with Cauchy's integral formula. 
\end{proof}

\section{$\gamma$-quantum wedge as a mating of trees} \label{section: wedge mating}

The goal of this section is to prove Theorem \ref{thm: peanosphere gamma wedge}. Roughly speaking, starting with the mating-of-trees result for the $\frac{3\gamma}{2}$-quantum wedge with field $h$ decorated by an independent space-filling $\SLE$ $\eta'$, we pick a boundary point at quantum distance one from the origin and zoom in on it. The field $h$ near this point is close to that of a $\gamma$-quantum wedge in total variation, and the curve $\eta'$ near this point is close to an independent counterclockwise space-filling $\SLE$ in total variation. Since we already know the boundary length process of $\eta'$ in the $\frac{3\gamma}{2}$-quantum wedge by Theorem~\hyperref[thm: peanosphere alpha wedge]{A}, we can deduce that of an independent space-filling $\SLE$ on a $\gamma$-quantum wedge. 

Our first task is to formalize what it means to ``locally look like a $\gamma$-quantum wedge''.

\begin{definition}
For $\delta,\eps > 0$ and a field $h$ defined on a neighborhood of 0 in $\bbH$, we say that \emph{the $\eps$-neighborhoods of $h$ and a $\gamma$-quantum wedge are $\delta$-close in total variation} if there exists a coupling of $h$ with a quantum wedge $(\bbH,\wt h, 0, \infty)$ such that with probability $1-\delta$ the following two fields agree:
\begin{itemize}
\item The field $h|_{B_d(0) \cap \bbH}$, where $ d>0$ satisfies $\nu_h([- d, d]) = \eps$;
\item The field $\wt h|_{B_d(0) \cap \bbH}$, where we have fixed the embedding of $(\bbH, \wt h, 0, \infty)$ so that $\wt h$ satisfies $\nu_{\wt h}([- d,  d]) = \eps$. 
\end{itemize}
\end{definition}

The following is a rephrasing\footnote{There, the parameter $C$ that they send to $\infty$ corresponds to our $\eps$ that goes to $0$, and they use the area version of ``canonical description'' while we use the boundary length variant (see \cite[paragraph after (1.8)]{shef-zipper}).} of \cite[Proposition 5.5]{shef-zipper}, and roughly says that when we condition on the quantum boundary length of a quantum surface and zoom in on a boundary point chosen by quantum length, the field looks like that of a $\gamma$-quantum wedge.
\begin{lemma}{\cite[Proposition 5.5]{shef-zipper}}\label{lem-zipper}
Let $D \subset \bbH$ be a bounded domain for which $I = \partial D \cap \R$ is a segment of positive length. Let $\wt h$ be a mixed-boundary GFF on $D$ with Dirichlet boundary conditions on $\partial D \cap \bbH$ and Neumann boundary conditions on $I$. For fixed $L_1,L_2 > 0$, condition on $\nu_h(I) = L_1 + L_2$, and let $y \in I$ be the point splitting $I$ into segments of $\nu_h$-lengths $L_1$ and $L_2$ respectively. Then for $\eps  >0$, the $\eps$-neighborhoods of $\wt h(\cdot + y)$ and a $\gamma$-quantum wedge are within $o(\eps)$ in total variation.
\end{lemma}

As an easy consequence of Lemma~\ref{lem-zipper}, we check that when one zooms in on a boundary point at unit quantum distance from the origin of a $\frac{3\gamma}{2}$-quantum wedge, the surface is locally close to a $\gamma$-quantum wedge. This is easier in the regime $\gamma \in (0,\sqrt2]$ because the $\frac{3\gamma}{2}$-quantum wedge is thick. When $\gamma \in (\sqrt2,2)$, the $\frac{3\gamma}{2}$-quantum wedge is thin and comprises countably many beads, and we will need to zoom in on a boundary point of one of these beads. 
 
We emphasize that in the next two proofs, for the case $\gamma \in (0,\sqrt2]$ we will be working with a $\frac32\gamma$-quantum wedge parametrized by $\cS$ so that neighborhoods of $-\infty$ (resp. $+\infty$) have finite (resp. infinite) quantum area. This is the opposite convention from Section~\ref{subsection: wedge}, so when we invoke Lemma~\ref{lem: rest of thick wedge} we will need to rotate its statement by a half-turn.

\begin{lemma}\label{lem: zoom in on alpha wedge}
Let $\alpha = \frac32 \gamma$.
\begin{enumerate}[(a)]
\item
Let $\gamma \in (0,\sqrt2]$, and fix $\eps > 0$. Let $(\cS, h, -\infty, +\infty)$ be an $\alpha$-quantum wedge. Let $y \in \R $ be the point satisfying $\nu_h((-\infty, y]) = 1$, and let $d > 0$ satisfy $\nu_h([y-d,y+d]) = \eps$. Then the $\eps$-neighborhoods of $h(\cdot + y)$ and a $\gamma$-quantum wedge are within $o(\eps)$ in total variation.
\item Let $\gamma \in (\sqrt2,2)$, and fix $\eps > 0$. Consider an $\alpha$-quantum wedge with field $h$. Let $y >0 $ be the point on the right boundary of this thin wedge satisfying $\nu_h([0,y]) = 1$, and parametrize the bead containing $y$ by $(\cS, h, -\infty, +\infty)$. As $\eps \to 0$, the probability of $\{\nu_h(\R) > \eps\}$ goes to 1. On this event define $d > 0$ via $\nu_h([y-d,y+d]) = \eps$. Then the $\eps$-neighborhoods of $h(\cdot + y)$ and a $\gamma$-quantum wedge are within $o(\eps)$ in total variation.

\end{enumerate}
\end{lemma}
\begin{proof}[Proof of Lemma~\ref{lem: zoom in on alpha wedge}]

\noindent\textit{Proof of (a).}
Since $+\infty$ is the marked point of the quantum wedge with infinite neighborhoods, all boundary segments bounded away from $+\infty$ have finite quantum length. Thus we can horizontally recenter the field so that $\nu_h(\R_-) = \frac12$. By Lemma~\ref{lem: rest of thick wedge}  (rotated by a half-turn) we know that, conditioned on $h|_{\cS_-}$, the law of $h|_{\cS_+}$ is a GFF with Neumann boundary conditions on $\R_+$ and $\R_+ + i\pi$ and Dirichlet boundary conditions on $[0,i\pi]$, plus an upward linear drift of $(Q-\alpha) \Re \cdot$. Choose $R>0$ so large that $y \in [0,R]$ with high probability, and let $D = [0,R] \times [0,\frac\pi2]$. Further condition on the realizations of $h|_{\cS \backslash D}$ and $L := \nu_h([0,R])$. By Lemma~\ref{lem-markov} (b), the conditional law of $h|_D$ is given by a GFF with Dirichlet boundary conditions on $\partial D \cap \cS$ (specified by $h|_{\cS \backslash D}$), Neumann boundary conditions on $[0,R]$, and conditioned on $\nu_h([0,R]) = L$. 

Recalling the definition of $y$, for any choice of $L > \frac12$ we see that conditioning on $\{\nu_h([0,R]) = L\}$ is the same as conditioning on both $\{\nu_h([0,y]) = \frac12\}$ and $\{\nu_h([y,R])=L - \frac12\}$. Thus by Lemma~\ref{lem-zipper}, for sufficiently small $\eps > 0$, under this conditioning the $\eps$-neighborhoods of $h(\cdot + y)$ and a $\gamma$-quantum wedge are close in total variation. Here, we need to choose $\eps$ small in terms of $L$ and $h|_{\cS \backslash D}$. Nevertheless we can choose $\eps$ so small that, with high probability w.r.t. the realizations of $L$ and $h|_{\cS \backslash D}$, conditioned on $L$ and $h|_{\cS \backslash D}$ the $\eps$-neighborhoods of $h(\cdot + y)$ and a $\gamma$-quantum wedge are close in total variation. Finally, as $R \to \infty$ the probability of $\{L > \frac12\}$ tends to 1, so (a) holds.

\medskip

\noindent \textit{Proof of (b).} In the regime $\gamma \in (\sqrt2, 2)$, our proof strategy is basically the same, but with additional details. Let $1-l$ be the sum of the right-boundary lengths of all the beads that come before the bead $(\cS, h, -\infty, +\infty)$ containing $y$, so that $y \in \R$ satisfies $\nu_h((-\infty, y]) = l$. Condition on $l$, so $(\cS, h, -\infty, +\infty)$ has the law of a thin wedge bead conditioned on $\nu_h(\R) > l$, and $y$ is the point satisfying $\nu_h((-\infty, y]) = l$. Now we can horizontally recenter the field so $\nu_h(\R_-) = \frac{l}{2}$, use Lemma~\ref{lem: alpha bead} and follow the proof of (a).
\end{proof}

Now, we are ready to prove Theorem \ref{thm: peanosphere gamma wedge}. 

\begin{figure}[ht!]
\begin{center}
\includegraphics[scale=0.65]{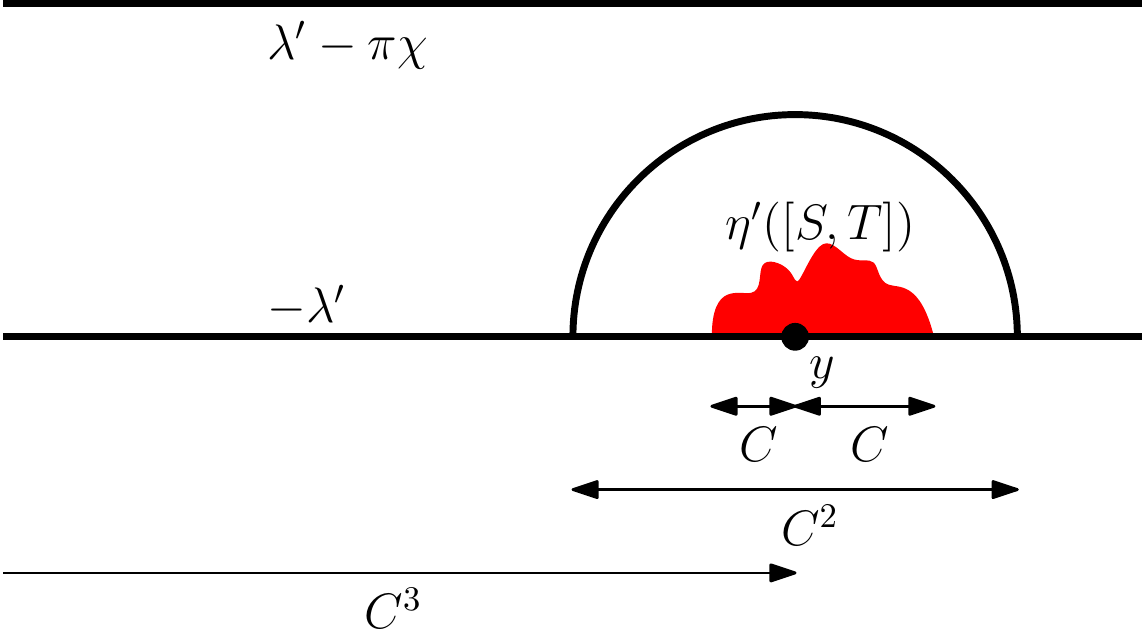}
\end{center}
\caption{\label{fig: gamma_wedge} Consider the case $\gamma \in (0,\sqrt2]$. On an $\alpha$-quantum wedge $(\cS, \widehat h, -\infty, +\infty)$, look at the point $y\in \R$ such that $\nu_{\widehat h}(\R_- + y) = C^3$, and draw a half-disk $D$ centered at $y$ such that $\nu_{\widehat h}(\partial D \cap \R) = C^2$. Consider an independent space-filling curve $\eta'$ from 0 to $\infty$ (coupled with a GFF $h^{\IG}$ with boundary values $-\lambda'$ on $\R$ and $\lambda'-\pi \chi$ on $\R+i\pi$). The $C^2$-neighborhoods of $\wh h(\cdot + y)$ and a $\gamma$-quantum wedge are close in total variation, and the curve $\eta'$ restricted to an interval of time where it's close to $y$ is close in total variation to a counterclockwise space-filling $\SLE$.}
\end{figure}

\begin{proof}[Proof of Theorem \ref{thm: peanosphere gamma wedge}]
Here we flesh out the proof for the regime $\gamma \in (0,\sqrt2]$. The case $\gamma \in (\sqrt2, 2)$ is proved in the same way with minor modifications. 

Pick some large $C$. Take a $\frac{3\gamma}2$-quantum wedge $(\cS, \widehat h, -\infty, +\infty)$ parametrized so $-\infty$ (resp. $+\infty$) has neighborhoods of finite (resp. infinite) quantum area. Mark the point $y \in \R$ such that $\nu_{\widehat h}((-\infty,y]) = C^3$ (see Figure \ref{fig: gamma_wedge}). Let $D$ be the half-disk centered at $y$ such that $\nu_{\widehat h}(\partial D \cap \R) = C^2$. By Lemma \ref{lem: zoom in on alpha wedge} and the scale invariance of the quantum wedge, the $C^2$-neighborhoods of $\wh h(\cdot + y)$ and a $\gamma$-quantum wedge have total variation distance going to zero as $C \to \infty$; moreover, the Euclidean diameter of $D$ converges in probability to $0$ as $C \to \infty$. Now, with $\lambda', \chi$ as in~\eqref{eqn: ig parameter}, we sample an independent Dirichlet GFF $h^{\op{IG}}$ on $\bbH$ with boundary value $-\lambda'$ on $\R$ and $\lambda' - \pi \chi$ on $\R + i\pi$, and consider its associated space-filling $\SLE_{\kappa'}$ $ \eta'$ from $-\infty$ to $+\infty$ as in Section~\ref{subsection: space-filling SLE}. We parametrize $\eta'$ by $\gamma$-quantum mass with respect to $\widehat h$. 
Let $S$ (resp. $T$) be the first time that $\eta'$ hits the point $C$ units of quantum length to the left (resp. right) of $y$; for $C$ large, we have with high probability that $\eta'[S,T] \subset D$. By Proposition \ref{prop: zoom in Dirichlet boundary}(b), we know that as $C \to \infty$ the field $h^{\op{IG}}|_D$ converges in total variation distance to the restriction to $D$ of a $\GFF$ in $\bbH$ with constant Dirichlet boundary conditions $-\lambda'$. Consequently, by \cite[Lemma 2.4]{gms-harmonic} we know that the path segment $\eta'|_{[S,T]}$ is $o_C(1)$-close in total variation to the path segment we would get by replacing $\eta'$ with a counterclockwise space-filling $\SLE_{\kappa'}$ on $\bbH$ independent of $\wh h$.

Thus, to understand the boundary length process in the setting of Theorem \ref{thm: peanosphere gamma wedge}, it suffices to describe the boundary length process of $\eta'|_{[S,T]}$ on $\wh h$ (with time reparametrized so the curve hits $y$ at time zero), then send $C \to \infty$. 

Let $(\widehat L_t,\widehat   R_t)_{t \geq 0}$ be the boundary length process of $\eta'$ on the $\frac{3\gamma}{2}$-quantum wedge; by Theorem~\hyperref[thm: peanosphere alpha wedge]{A} this is a two-dimensional Brownian motion with initial value $\widehat L_0 = \widehat R_0 = 0$ and having covariances given by \eqref{eqn: covariance}. By definition, $S$ (resp. $T$) is the first time that $\widehat R_t = -C^3 + C$ (resp. $\widehat R_t = -C^3 - C$). Let $\tau \in (S,T)$ be the first time that $\widehat R_t = -C^3$ (i.e. the time that $\eta'$ hits $y$). Let $(L_t, R_t)_{[S-\tau, T - \tau]}$ be the time-translation of the process $(\widehat L_t, \widehat R_t)_{[S,T]}$, with additive constant normalized so that $L_0 = R_0 = 0$. Then $(L_t, R_t)_{[0,T-\tau]}$ is Brownian motion with covariances \eqref{eqn: covariance} started at the origin and stopped when $R_t$ hits $-C$, and $(L_t,R_t)_{[S-\tau, 0]}$ is the time-reversal of Brownian motion $(\widetilde L_t, \widetilde R_t)_{[0,\tau - S]}$ with covariances \eqref{eqn: covariance} started at the origin, conditioned on $\widetilde R_t \geq 0$ for all time, and stopped at the last time $\widetilde R_t$ takes the value $C$. By the strong Markov property of Brownian motion, $(L_t,R_t)_{[S-\tau, 0]}$ and $(L_t, R_t)_{[0,T-\tau]}$ are independent. Thus, this boundary length process $(L_t, R_t)_{[S-\tau, T - \tau]}$ converges in total variation on compact time intervals to the process described in Theorem~\ref{thm: peanosphere gamma wedge}.

Finally, we show that if $\cW$ is a $\gamma$-quantum wedge and $\eta'$ an independent counterclockwise space-filling $\SLE_{\kappa'}$ from $\infty$ to $\infty$ on $\cW$, then the boundary length process of $\eta'$ on $\cW$ a.s. determines the curve-decorated quantum surface $\cW^* := (\cW, \eta')$. 
Write $\wt \cW^*$ for the space-filling $\SLE_{\kappa'}$-decorated $\frac{3\gamma}2$-quantum wedge of Theorem~\hyperref[thm: peanosphere alpha wedge]{A}. Lemma~\ref{lem-measurability} tells us that a.s., for any finite interval of time $I$, the boundary length process of $\wt \cW^*$ restricted to $I$ locally determines the curve-decorated quantum surface $(\eta'(I), h, \eta'|_{I})$. In the above proof we obtained local approximations (in total variation) of $\cW^*$ by conditioning $\wt \cW^*$ on events of positive probability. Thus, given the boundary length process on $\cW^*$ in any compact interval $I$ we can a.s. recover the curve-decorated surface explored during that interval. Letting $I$ increase to all of $\BB R$ concludes the proof of Theorem~\ref{thm: peanosphere gamma wedge}.

\end{proof}

\begin{remark}\label{rem-recover}
The above argument proves that $(L_t, R_t)$ determines the curve-decorated $\gamma$-quantum wedge $(\bbH, h, \eta', 0, \infty)$, and, more strongly,  a.s. for every interval $I = [a,b]$, the process $(L_t -L_a , R_t -R_b )_{t \in I}$ determines the curve-decorated quantum surface $(\eta'(I), h|_{\eta'(I) , \eta'|_I})$ explored by $\eta'$ during the interval $I$. 
\end{remark}

Now, we identify the curve-decorated quantum surfaces parametrized by the regions explored by $\eta'$ before and after hitting 0. We recall that to decorate a thin quantum wedge by an independent space-filling $\SLE_{\kappa'}$ curve, we independently sample for each bead a chordal space-filling $\SLE_{\kappa'}$ from one marked point to the other, then concatenate these curves according to the ordering of the beads. 

\begin{theorem}\label{thm: wedge laws}
For $\gamma \in (0,2)$, let $(\bbH, h, 0, \infty)$ be a $\gamma$-quantum wedge, and $\eta'$ an independent space-filling counterclockwise $\SLE_{\kappa'}$ curve from $\infty$ to $\infty$ parametrized so that $\eta'(0) = 0$. Let $U $ be the interior of $\eta' ((-\infty, 0])$ and let $V$ be the interior of $\eta'([0,\infty))$. Then $(U, h, 0, \infty, \eta'_{(-\infty,0]})$ is a $Q$-quantum wedge decorated by an independent space-filling  $\SLE_{\kappa'}(\frac{\kappa'}{2}-4;0)$ curve from $\infty$ to $0$, and $(V, h, 0, \infty, \eta'|_{[0,\infty)})$ is a $\frac{3\gamma}{2}$-quantum wedge decorated by an independent space-filling  $\SLE_{\kappa'}$ curve from $0$ to $\infty$.
\end{theorem}
\begin{proof} 
Following~\cite[Table 1.1]{wedges}, we define the \emph{weight} of an $\alpha$-quantum wedge by
\eqbn
W = \gamma\left( \frac{\gamma}{2} + Q - \alpha \right) .
\eqen 
Note that a $\gamma$-quantum wedge has weight $W=2$.
By Lemma~\ref{lem: identify-wedges}, we know that the interface $\eta'((-\infty,0])\cap \eta'([0,\infty))$ is an $\SLE_{\kappa}(\rho_L; \rho_R)$ curve from 0 to $\infty$ in $\BB H$ with force points immediately to the left and right of 0, where $\rho_L = \frac{\gamma^2}{2}-2$ and $\rho_R = -\frac{\gamma^2}{2}$. Let $  W_L = \rho_L + 2$, and $W_R = \rho_R + 2$. 
By \cite[Theorem 1.2]{wedges}, the flow line described above splits the $\gamma$-quantum wedge into two independent wedges of weights $W_L$ and $W_R$. As such, we see that $(U, h, 0, \infty)$ is a wedge of weight $ W_L = \frac{\gamma^2}{2}$. Converting, we conclude that this is an $\alpha$-quantum wedge for $\alpha = Q$. Likewise, $(V,h,0,\infty)$ is a quantum wedge of weight $2-\frac{\gamma^2}{2}$, so it is an $\alpha$-quantum wedge for $\alpha = \frac{3\gamma}{2}$.
\end{proof}

\begin{remark}\label{remark: topology of wedges}
For $\gamma \in (0,\sqrt2]$, both quantum wedges satisfy $\alpha \leq Q$, so they are thick wedges. Thus $U,V$ are simply connected.
For $\gamma \in (\sqrt2,2)$, the wedge $(U, h, 0, \infty)$ is thick, but since $\frac{3\gamma}{2} > Q$, the wedge $(V,h,0,\infty)$ is thin. Thus $U$ is simply connected but $V$ has countably many simply connected components or ``beads''. The beads come with a natural ordering (the order in which the boundaries are drawn by $\eta'$), and so we can define the left (resp. right) boundary of $V$ by concatenating the left (resp. right) boundaries of the beads according to their natural ordering. See Figure~\ref{fig: alpha_wedge}, right.
\end{remark}

\section{Brownian excursions in the cone}\label{section: cone excursions}
The goal of this section is to define and discuss the properties of the correlated Brownian excursion mentioned in Theorem \ref{thm: peanosphere disk}, to prove an approximation result for correlated Brownian excursions in $\R_+ \times \R_+$, and to identify the law of the quantum area of a unit boundary length quantum disk assuming Theorem~\ref{thm: peanosphere disk}. In Section~\ref{subsection: uncorrelated brownian excursions}, we recall the properties of uncorrelated Brownian excursions. In Section~\ref{subsection: correlated brownian excursions}, we define correlated Brownian excursions, and discuss an approximation scheme for the correlated excursion of Theorem \ref{thm: peanosphere disk}; this approximation result will be used in Section~\ref{subsection: decomposing wedge} to show that the boundary length process of a space-filling SLE in a bottlenecked region of a quantum wedge converges to the correlated excursion of Theorem \ref{thm: peanosphere disk}. In Section~\ref{subsection-shimura}, we discuss Brownian motion started at the vertex of a cone and conditioned to stay in the cone for some time; this was studied in \cite{shimura1985}. Finally in Section~\ref{subsection: area of disk} we build on Section~\ref{subsection-shimura} to identify the unit boundary length quantum disk area law (Theorem~\ref{thm: area of disk}) conditional on Theorem~\ref{thm: peanosphere disk}. 

\subsection{Uncorrelated Brownian excursions} \label{subsection: uncorrelated brownian excursions}
In this section, we summarize the properties of uncorrelated Brownian excursions. Most properties in this section are taken verbatim from \cite[Sections 2 and 3]{lawler-werner-soup}, and we will justify the remainder via results in \cite{lawler-werner-soup}.

We start with some notation. Let $\cK$ be the space of all parametrized continuous planar curves $\eta$ defined on a time-interval $[0,t_\eta]$, endowed with the metric
\begin{equation} \label{eqn: metric on space of curves}
d_\cK (\eta, \eta^1) = \inf_\theta \left\{ \sup_{s \in [0,t_\eta]} |s - \theta(s)| + |\eta (s) - \eta^1(\theta(s))|\right\}, 
\end{equation}
where the infimum is taken over all increasing homeomorphisms $\theta: [0,t_\eta] \to [0,t_{\eta^1}]$. Note that this metric does not identify curves which are the same under time-reparametrization. Sometimes, we will deal with curves $\eta:[s,t] \to \C$ with $s \neq 0$; by translating the domain of $\eta$ to $[0, t-s]$, we will view $\eta$ as an element of $\cK$. For a simply-connected domain $D$, we write $\cK_D$ for the curves $\eta$ such that $\eta ((0, t_\eta)) \subset D$. Note that we allow the endpoints of $\eta$ to lie in $\partial D$. The space of probability measures on $\cK$ is a metric space, under the Prohorov metric.

Given a conformal map $\varphi: D \to \widetilde D$ and a curve $\eta$, if the quantity
\[s_t = \int_0^t |\varphi'(\eta(s))|^2 \ ds \]
is finite for all $t \in [0,t_\eta]$, we can define the curve $\varphi_* \eta$ in $\widetilde D$ via $\varphi_*  \eta (s_t) := \varphi(\eta(t))$. We choose this time-parametrization for the pushforward for the following reason: if $\eta$ is Brownian motion parametrized to have quadratic variation $2dt$, then $\varphi_* \eta$ is also Brownian motion parametrized to have quadratic variation $2dt$. If $\mu$ is any measure on $\cK_D$ supported on the set of curves $\eta$ such that $\varphi_* \eta$ is well defined and in $\cK_{\widetilde D}$, we write $\varphi_* \mu$ for the induced measure 
\[\varphi_* \mu (V) = \mu \{ \eta \: : \: \varphi_* \eta \in V\}. \]

Now, we are ready to discuss the various kinds of normalized bridge and excursion measures we need in this paper. Consider first the case where $D$ is a simply connected domain in $\C$ whose boundary is a finite union of analytic curves; as we will explain later the results can be extended to all simply connected $D$ via conformal invariance. For $z \in D$ and any harmonically nontrivial boundary interval $I \subset \partial D$, let $\mu^\#_D(z,I)$ be the probability measure on $\cK_D$ given by standard Brownian motion started at $z$ and conditioned to exit $D$ at some boundary point in $I$. Let $H_D(z, dw)$ be the \textit{Poisson kernel}; that is, for any $I \subset \partial D$ we have 
\[\P(\text{Brownian motion started at }z \text{ exits }D \text{ in } I) = \int_I H_D(z, dw). \]

For $z \in D$ and $w \in \partial D$, the \textit{normalized\footnote{In \cite{lawler-werner-soup}, the measures $\mu_D(z,w)$ are not necessarily probability measures. Using their notation, we write $\mu^\#_D(z,w)$ to refer to the normalized measure having total mass one.} interior-to-boundary measure} $\mu^\#_D(z, w)$ is a probability measure on $\cK_D$ supported on paths starting at $z$ and ending at $w$. 
For any harmonically nontrivial boundary interval $I \subset \partial D$, we can decompose $\mu^\#_D(z,I)$ in terms of the normalized interior-to-boundary Brownian measure (this is a rephrasing of \cite[Section 3.1.2]{lawler-werner-soup}):
\eqb\label{eqn: decomposition of BM to excursions}
\mu^\#_D(z, I) = \int_I \mu_D^\#(z,w) \frac{H_D(z,dw)}{H_D(z,I)}. 
\eqe
The probability measure $\mu^\#_D(z,w)$ is conformally invariant, i.e. for conformal $\varphi: D \to \widetilde D$ we have
\[\mu^\#_{\widetilde D}(\varphi(z), \varphi(w)) = \varphi_* \mu^\#_D(z,w). \]

For $z,w \in \partial D$, the \textit{normalized boundary-to-boundary excursion measure} $\mu^\#_D(z,w)$ is a probability measure on $\cK_D$ supported on paths starting at $z$ and ending at $w$. By normalizing the measures in \cite[(6)]{lawler-werner-soup} to be probability measures, we have the following: for any boundary points $z,w \in \partial D$ such that $\partial D$ is locally analytic at $z$, this normalized excursion measure satisfies 
\begin{equation}\label{eqn: approximating brownian excursion}
\lim_{\eps \to 0} \mu^\#_D (z + \eps \mathbf{n}_z, w) = \mu^\#_D (z,w), \text{ with convergence under the Prohorov metric.}
\end{equation}
Here, $\mathbf{n}_z$ is the inward pointing normal vector at $z$, and $\mu^\#_D (z + \eps \mathbf{n}_z, w)$ is the normalized interior-to-boundary Brownian measure. As with the interior-to-boundary case, this probability measure is conformally invariant; for $\varphi: D \to \widetilde D$ a conformal map we have 
\[\mu^\#_{\widetilde D}(\varphi(z), \varphi(w)) = \varphi_* \mu^\#_D(z,w). \]
These measures make sense for $D$ with boundary a finite union of analytic curves. By conformal invariance, we can make sense of these measures for \emph{all} simply connected $D$, by conformally mapping to a domain for which we have defined these measures. 

The normalized boundary-to-boundary Brownian excursion measure is reversible: for $z,w \in \partial D$, one can sample from $\mu^\#_D(z,w)$ by taking the time-reversal of a path sampled from $\mu^\#_D(w,z)$. This follows immediately from \cite[(6)]{lawler-werner-soup} and the reversibility of Brownian motion.

One can often prove statements of the form $\lim_{n\to \infty} \mu^\#_D(x_n, y_n) = \mu^\#_D(x,y)$ for sequences of points satisfying $x_n \to x$ and $y_n \to y$ (for instance \eqref{eqn: approximating brownian excursion}). See \cite[Sections 3.2.2, 3.3.2]{lawler-werner-soup}.

\subsection{Correlated Brownian excursions and approximations} \label{subsection: correlated brownian excursions}

Recalling the unknown constant $\BB a$ in the mating-of-trees Brownian motion covariance \eqref{eqn: covariance}, let
\begin{equation}\label{eqn: Lambda}
\theta := \frac{\pi \gamma^2}{4}, \quad \Lambda := \frac1{\BB a}\begin{pmatrix}
\frac{1}{\sin \theta}& \frac{1}{\tan \theta},\\
0 & 1
\end{pmatrix},
\end{equation}
and define
\begin{equation} \label{eqn: sheared coordinates}
\cC_\theta := \Lambda \R_+^2 = \{z \in \C \: : \: \arg(z) \in [0,\theta]\}, \quad x := \Lambda \binom10 = \frac1{\BB a}\binom{\frac{1}{\sin \theta}}{0}, \quad y := \Lambda \binom01 = \frac1{\BB a}\binom{\frac{1}{\tan \theta}}{1}.
\end{equation}
It is easy to check that the shear transformation $\Lambda^{-1}$ sends standard Brownian motion into Brownian motion with covariances \eqref{eqn: covariance}. Hence, the following definition is natural.
 
\begin{definition}\label{def: cone BM}
Fix $\gamma \in (0,2)$ and $z,w \in \partial \R_+^2$. We define the \emph{sheared normalized boundary-to-boundary excursion measure} from $z$ to $w$ with with covariances given by \eqref{eqn: covariance} as follows: To get a sample $(L_t, R_t)$ from it, first sample $\eta \sim  \mu_{\cC_\theta}^\# (\Lambda z,\Lambda w)$ (with random duration $t_\eta$), and then set $\binom{L_t}{R_t}:= \Lambda^{-1} \eta_t$ for all $t \in [0,t_\eta]$.
We similarly define the \emph{sheared normalized interior-to-boundary measure}.
\end{definition}
See Figure~\ref{fig: boundary_lengths_disk_bottom} for a simulation of a sheared excursion from $(0,0)$ to $(0,1)$. This definition extends by translation to cones $\R_+^2 + (a,b)$ with arbitrary vertex $(a,b)$.
For the rest of this section, we work to prove Proposition \ref{prop: approx BM}, which says that we can approximate the sheared normalized boundary-to-boundary measure by suitable sheared normalized interior-to-boundary measures. In the proof of Theorem~\ref{thm: peanosphere disk} we approximate a quantum disk by conditioning on the existence of a bottleneck in a quantum wedge; Proposition~\ref{prop: approx BM} tells us that the boundary length process of space-filling SLE in the bottlenecked region converges to the desired correlated Brownian excursion.

\begin{proposition}\label{prop: approx BM}
Let $(L_t, R_t)$ be two-dimensional Brownian motion with covariances given by \eqref{eqn: covariance}. For $\delta, c> 0$, start at $(L_0,R_0) = (0,c)$, and run until the stopping time $T$ when $(L_t, R_t)$ first exits the cone $(\R_+ - \delta) \times \R_+ $. Condition on $\{L_T \in [\delta, 2\delta], R_T = 0\}$. Then as $\delta \to 0$ and $c \to 1$, in the Prohorov topology, the law of $(L_t, R_t)$ converges to the sheared normalized boundary-to-boundary excursion in $\R_+^2$ defined in Definition \ref{def: cone BM}, starting at $(0,1)$ and ending at $(0,0)$. 
\end{proposition}

In order to prove Proposition \ref{prop: approx BM}, we first prove the statement for unsheared excursion measures. The argument is messy but not difficult, and follows \cite[Section 3.3.2]{lawler-werner-soup}.

\begin{lemma} \label{lem: sheared cone approximation}
Recall the definitions in \eqref{eqn: Lambda} and \eqref{eqn: sheared coordinates}. For $\delta > 0$, let $I_\delta$ be the boundary interval $[2\delta x, 3 \delta x] \subset \partial \cC_\theta$. Under the Prohorov metric, we have
\begin{equation}
\mu^\#_{\cC_\theta}(cy + \delta x, I_\delta) \to \mu^\#_{\cC_\theta}(y,0) \quad \text{ as } (\delta,c) \to (0,1) .
\end{equation}
\end{lemma}

\begin{proof}
Consider any large $R > 0$, and define the truncated domain $D_R = \cC_\theta \cap \Lambda B_R(0)$. Clearly, for the following family of normalized interior-to-boundary measures we have
\begin{equation}
\lim_{R \to \infty} \mu^\#_{\cC_\theta}(cy+sx,tx) (\{\eta : \eta[0,t_\eta] \subset D_R \}) = 1 \quad \text{ uniformly over } c,s,t \in [0,2).
\end{equation}
Put somewhat less precisely, if we start Brownian motion close to the origin and run it until it exits $\cC_\theta$ near the origin, then uniformly over the choice of starting and ending points, with high probability the Brownian motion does not wander too far from the origin. As a result, the Prohorov distance between $\mu^\#_{D_R}(cy + \delta x, I_\delta)$ and $\mu^\#_{\cC_\theta}(cy + \delta x, I_\delta)$ goes to zero as $R \to \infty$, uniformly for $c, \delta \in [0,2]$.  

Pick any $w \in I_\delta$ and let $\mathbf{n}$ be the unit inward normal vector to the ray $\{ty \: : \: t \in \R_+\}$. Let $\phi^{\delta,c,w}: D_R \to D_R$ be the unique conformal map which sends $w \mapsto 0$ and $cy+\delta x \mapsto y + \delta {\BB a}^{-1} \mathbf{n}$. Using \cite[Lemmas 2 and 3]{lawler-werner-soup}, uniformly in the choice of $(\delta, c)$ close to $(0,1)$ and $w \in I_\delta$, for $\eta\sim \mu_{D_R}^\#(cy + \delta x, w)$ we have $\E[ d_\cK(\eta, \phi^{\delta, c, w}_* \eta)] = o(1)$. By conformal invariance, the law of $\phi^{\delta, c, w}_* \eta$ is simply $\mu_{D_R}^\#(y + \delta {\BB a}^{-1} \mathbf{n}, 0)$, so the Prohorov distance between $\mu_{D_R}^\#(cy + \delta x, w)$ and $\mu_{D_R}^\#(y + \delta {\BB a}^{-1} \mathbf{n}, 0)$ goes to zero as $(\delta,c) \to (0,1)$, uniformly in $w$. But by \eqref{eqn: approximating brownian excursion}, we see that the Prohorov distance between $\mu_{D_R}^\#(y + \delta {\BB a}^{-1} \mathbf{n}, 0)$ and $\mu_{D_R}^\#(y, 0)$ goes to 0 as $(\delta,c) \to (0,1)$. Finally, since $\mu^\#_{D_R}(cy + \delta x, I_\delta)  = \int_{I_\delta}  \mu^\#_{D_R}(cy + \delta x, w) \ \frac{H_{D_R}(cy + \delta x, dw)}{H_{D_R}(cy+\delta x, I_\delta)} $, we conclude that $\mu^\#_{D_R}(cy + \delta x, I_\delta)\to\mu_{D_R}^\#(y, 0)$ as $(\delta,c) \to (0,1)$. Taking $R \to \infty$, we conclude that $\mu^\#_{\cC_\theta}(cy + \delta x, I_\delta)  \to \mu_{\cC_\theta}^\#(y,0)$ as desired. 
\end{proof}

\begin{proof}[Proof of Proposition \ref{prop: approx BM}]
By Lemma~\ref{lem: sheared cone approximation}, we know that as $(\delta, c) \to (0,1)$, the Prohorov distance between the law of $(L_t, R_t)$ and the sheared normalized boundary-to-boundary cone excursion in $(\R_+ - \delta) \times \R_+ $ from $(-\delta,1)$ to $(-\delta,0)$ goes to zero. But the Prohorov distance between this latter law and the sheared normalized boundary-to-boundary cone excursion in $\R_+\times \R_+ $ from $(0,1)$ to $(0,0)$ is at most $\delta$, since we can couple the laws via $(\eta - \delta, \eta)$ (so $\eta$ is an excursion in $\R_+ \times \R_+$, and $\eta - \delta$ an excursion in $(\R_+ - \delta) \times \R_+$) to have distance $d_\mathcal{K} (\eta - \delta, \eta) = \delta$. Thus Proposition~\ref{prop: approx BM} holds.
\end{proof}

\subsection{Shimura's Brownian excursion in the cone}\label{subsection-shimura}
In this section we summarize the main result we need from \cite[Theorem 2]{shimura1985}, which makes sense of ``Brownian motion started at the vertex of a cone and conditioned to stay in the cone for some time'' via a limiting procedure. We prove a decomposition analogous to~\eqref{eqn: decomposition of BM to excursions} for this cone excursion, which will be used in Section~\ref{subsection: area of disk} to analyze the law of the cone excursion duration.

Let $\cC_\theta$ be the $\theta$-angle cone defined in~\eqref{eqn: sheared coordinates}. This section deals with multiple probability measures on Brownian-type paths in $\cC_\theta$, e.g. $\mu^\#_{\cC_\theta}(z,w)$ and $\P^\eps_z$ (defined below). To avoid notational clutter, we will always refer to such paths as $(Z_t)_{[0,\tau]}$, where $\tau$ is the random exit time of $\cC_\theta$.  
For $z \in \cC_\theta$ and $\eps > 0$, define $\P^\eps_z$ to be the probability measure of Brownian motion $(Z_t)_{[0,\tau]}$ started at $Z_0 = z$ and conditioned on $\{\tau > \eps\}$. Finally, we set
\eqb\label{eq-lambda}
\lambda := \frac{\pi}{\theta}  = \frac{4}{\gamma^2};
\eqe
this is called $\mu$ in~\cite{shimura1985} but we rename it to avoid ambiguity.

\begin{proposition}{\cite[Theorem 2]{shimura1985}}\label{prop-shimura}
The measures $\P_z^\eps$ converge weakly as $z \to 0$ to a limit measure $\P^\eps$, which can be interpreted as the law of Brownian motion in $\cC_\theta$ started at $0$ and conditioned to stay in $\cC_\theta$ for time $\eps$. Moreover, the $\P^1$-law of the Brownian motion evaluated at time 1 is
\eqb \label{eqn-Z_1-density}
\BB P^1\left[ Z_1 \in \,dz \right] =    c_1 |z|^{\mu} \sin\left( \lambda \op{arg}(z) \right)   \exp\left(  - \frac{|z|^2}{2 }  \right)  \BB 1_{(z\in \cC_\theta)} \,dz
 .
\eqe 
\end{proposition}
Strictly speaking, \cite[Theorem 2]{shimura1985} only proves weak convergence of the curves of $\P_z^1$ \emph{restricted} to time interval $[0,1]$ to some limit law (i.e. the curves considered are $(Z_t)_{[0,1]}$ rather than $(Z_t)_{[0,\tau]}$); since $(Z_t)_{[0,\tau]}$ evolves as standard Brownian motion after time 1, our extension is valid. Brownian rescaling allows us to consider general $\eps > 0$.

To prove a decomposition similar to~\eqref{eqn: decomposition of BM to excursions} for $\P^\eps$, we will need the following continuity lemma.
\begin{lemma}\label{lem-cone-cont}
Let $\mu^\#_{\cC_\theta,\eps}(0,u)$ be the measure $\mu^\#_{\cC_\theta}(0,u)$ conditioned on the event $\{\tau > \eps\}$, i.e. conditioned on the excursion duration being greater than $\eps$. Then the measures $\mu^\#_{\cC_\theta}(0,u),\mu^\#_{\cC_\theta,\eps}(0,u)$ vary continuously in the weak topology for $u \in \partial \cC_\theta \backslash \{ 0\}$. 
\end{lemma}
\begin{proof}
We can show the continuity of $\mu^\#_{\cC_\theta}(0,u)$ by the same methods used to prove Lemma~\ref{lem: sheared cone approximation} (or, similarly, \cite[Lemma 4]{lawler-werner-soup}). 
The continuity of $\mu^\#_{\cC_\theta,\eps}(0,u)$ follows from that of $\mu^\#_{\cC_\theta}(0,u)$ together with the fact that the $\mu^\#_{\cC_\theta}(0,u)$-law of $\tau$ has a density against Lebesgue measure. 
\end{proof}
\begin{proposition}\label{prop-shimura-decomp}
Let $\mu^\#_{\cC_\theta,\eps}(z,w)$ be as in Lemma~\eqref{lem-cone-cont}. The measure $\P^\eps$ of Proposition~\ref{prop-shimura} can be decomposed as
\begin{equation}\label{eqn: decompose P}
\P^\eps = \int_{\cC_\theta} \mu^\#_{\cC_\theta,\eps}(0,u) \P^\eps (Z_\tau \in du).
\end{equation}
\end{proposition}
\begin{proof}
Conditioning~\eqref{eqn: decomposition of BM to excursions} on the event $\{\tau > \eps\}$ and using Bayes' theorem, we get the following decomposition:
\begin{equation}\label{eqn: decompose Pz}
\P^\eps_z = \int_{\partial \cC_\theta}  \mu^\#_{\cC_\theta, \eps} (z,u)  \P^\eps_z( Z_\tau \in du).
\end{equation}
The boundary measure $\P^\eps_z( Z_\tau \in du)$ converges weakly to $\P^\eps(Z_\tau \in du)$ as $z \to 0$ (Proposition~\ref{prop-shimura}). By \eqref{eqn: approximating brownian excursion} (e.g. after mapping to the domain $\bbH$ with smooth boundary), we also have the weak convergence $\mu^\#_{\cC_\theta}(z,u) \to \mu^\#_{\cC_\theta}(0,u)$ as $z \to 0$, and hence $\mu^\#_{\cC_\theta, \eps}(z,u) \to \mu^\#_{\cC_\theta, \eps}(0,u)$ as $z \to 0$. Finally, $\mu^\#_{\cC_\theta,\eps}(0,u)$ is continuous in $u$ by Lemma~\ref{lem-cone-cont}. Thus we can take $z \to 0$ in \eqref{eqn: decompose Pz} to obtain \eqref{eqn: decompose P}.
\end{proof}

\subsection{Conditional law of the area of a quantum disk given its boundary length}\label{subsection: area of disk}
In this section, we compute the law of the exit time of a sheared normalized Brownian excursion with covariances~\eqref{eqn: covariance} in the cone $\R_+^2$ from $(0,0)$ to $(1,0)$, and hence deduce the law of the area of a unit boundary length quantum disk modulo the unknown constant ${\BB a}$. Once we know Theorem~\ref{thm: peanosphere disk}, because Brownian excursions are reversible, the following is a restatement of Theorem~\ref{thm: intro area disk} in terms of Brownian excursions. 

\begin{theorem}\label{thm: area of disk}
The duration of the sheared normalized Brownian excursion with covariances~\eqref{eqn: covariance} in the cone $\R_+^2$ from $(0,0)$ to $(1,0)$ is distributed according to the law 
\eqb \label{eqn-area-law}
  \frac{1}{c t^{1+4/\gamma^2}}  \exp\left(-\frac{1}{2 ( \BB a \sin(\pi\gamma^2/4) )^2  t } \right)  \, dt
  \quad \text{for} \quad c =  2^{4/\gamma^2} \Gamma(4/\gamma^2) (\BB a \sin(\pi\gamma^2/4) )^{ 8/\gamma^2 } .
\eqe    
\end{theorem}

Recall the definitions of the various constants in~\eqref{eqn: Lambda} and~\eqref{eqn: sheared coordinates} (which depend on $\gamma$), and the notations of Section~\ref{subsection-shimura}. This section deals with multiple probability measures on Brownian-type paths in $\cC_\theta$, e.g. $\mu^\#_{\cC_\theta,\eps}(0,w)$ and $\P^\eps$. To avoid notational clutter, we will always refer to such a path as $(Z_t)_{[0,\tau]}$, where $\tau$ is the random exit time of $\cC_\theta$. 

By Definition~\ref{def: cone BM}, it suffices to understand the law of $\tau$ under $\mu^\#_{\cC_\theta}(0,u)$ for a boundary point $u \in \partial \cC_\theta$, then substitute the choice $u = \Lambda(1,0) = ( ( \BB a \sin \theta)^{-1} ,0)$. To that end, we will first compute $\mu^\#_{\cC_\theta, \eps}(0,u) \{\tau > t\}$ for $t>\eps > 0$, then send $\eps \to 0$ to obtain $\mu^\#_{\cC_\theta}(0,u)\{ \tau > t\}$, and finally differentiate in $t$ to finish the proof of Theorem~\ref{thm: area of disk}.

Consider Proposition~\ref{prop-shimura-decomp}. By Bayes' theorem and the continuity of $\mu^\#_{\cC_\theta,\eps}(0,u)$ w.r.t. $u$ in the weak topology (Lemma~\ref{lem-cone-cont}), we obtain
\begin{equation}\label{eqn-bayes}
\mu^\#_{\cC_\theta, \eps}(0,u) \{\tau > t\} = \P^\eps \left[ \tau > t \: | \: Z_\tau = u\right] = \frac{\P^\eps \left[Z_\tau \in du  \: | \: \tau > t\right]\P^\eps[\tau > t]}{\P^\eps[Z_\tau \in du]}.
\end{equation}
We will compute asymptotic formulas for each of the three terms on the right side of~\eqref{eqn-bayes}. 

In the calculations that follow, we write $\lambda = \frac{4}{\gamma^2}$ (as in~\eqref{eq-lambda}), and write $c_1,c_2,\dots$ for explicit constants which depend only on $\gamma$. We will not keep track of the values of these constants during the calculation (although it is possible to do so). Rather, we will compute the multiplicative constant in~\eqref{eqn-area-law} at the very end of the argument using that $\mu^\#_{\cC_\theta} \{\tau  > t\} \rta 1$ as $t\rta 0$. 

We first study the first factor in the numerator in~\eqref{eqn-bayes}. To deal with this, we will compute both the $\P^\eps$-conditional law of $Z_t$ given $\{\tau > t\}$ and the $\P^\eps$-conditional law of $Z_\tau$ given $Z_t$ and $\{\tau > t\}$. We note that~\cite[Equation (8)]{iyengar-hitting-lines} gives a series expansion for $\P^\eps[Z_\tau \in \, du ,\, \tau > t   ]$. We prefer to instead give a direct calculation. 

\begin{lemma} \label{lem-time-t-law}
For $t\geq \ep$, the $\P^\eps$-conditional law of $Z_t$ given $\{\tau > t\}$ is
\eqb \label{eqn-time-t-law}
\P^\eps\left[ Z_t \in \,dz \,|\, \tau > t \right] =    
c_1 \frac{ |z|^{\lambda} \sin\left( \lambda \op{arg}(z) \right) }{  t^{1 + \lambda/2 }}  \exp\left(  - \frac{  |z|^2}{2 t }  \right) \BB 1_{(z\in \cC_\theta)} dz ,   
\eqe  
where $c_1 := 2^{- \lambda/2 } \Gamma\left( \lambda/2 \right)^{-1}$.
\end{lemma}
\begin{proof} 
By definition, the $\P^\eps$-conditional law of $Z_t$ given $\{ \tau > t\}$ is the same as the $\P^t$ law of $Z_t$, which by Brownian scaling is the same as the $\P^1$-law of $ t^{1/2} Z_1$.
Thus~\eqref{eqn-time-t-law} follows from Proposition~\ref{prop-shimura} by making a change of variables.
\end{proof}

\begin{lemma} \label{lem-cone-hm}
For $t \geq \eps$, the $\P^\eps$-regular conditional law of $Z_\tau$ given $Z|_{[0,t]}$ on the event $\{\tau  >t\}  \cap \{ Z_t = z\}$ is given by
\eqb \label{eqn-cone-hm}
   \P^\eps \left[Z_\tau \in du \,|\, Z_{[0,t]}, \left\{\tau > t, Z_t = z \right\}\right]  =  \frac{  |z|^{\lambda} |u|^{1 - \lambda} }{\theta |z^{\lambda} - u^{\lambda} |^2}  \sin\left(\lambda \op{arg}(z) \right) \,du .    
\eqe 
\end{lemma}
\begin{proof}
If we condition on $Z|_{[0,t]}$, then on the event $\{\tau > t\} \cap \{Z_t = z\}$, the regular conditional law of $Z_\tau$ is that of the first exit point from $\mcl C_\theta$ of a standard planar Brownian motion started from $Z_t$. By Brownian scaling, this is the same as the harmonic measure on $\mcl C_\theta$ seen from $z$.  
The map $z\mapsto z^{\lambda}$ takes $\mcl C_\theta$ to the upper half-plane $\BB H  $. By the conformal invariance of Brownian motion and the well-known formula for the Poisson kernel on $\BB H$, the density of harmonic measure as seen from $z^{\lambda}$ in $\BB H$ (which is a measure on $\BB R$) is 
\eqbn
\frac{\im (z^{\lambda})}{\pi |z^{\lambda} - x|^2} \,dx  .
\eqen
By substituting $x = u^{\lambda}$ and $dx = \lambda |u|^{\lambda-1} \, du$ and noting that $\im(z^{\lambda}) = |z|^{\lambda} \sin\left( \lambda \op{arg}(z) \right)$, we get~\eqref{eqn-cone-hm}. 
\end{proof}

By Lemmas~\ref{lem-time-t-law} and~\ref{lem-cone-hm}, the joint law of $Z_\tau$ and $Z_t$ given $\{\tau > t\}$ is 
\eqb \label{eqn-joint-law}
\P^\eps\left[ Z_\tau \in \,du ,\,  Z_t \in \, dz  \,|\, \tau > t \right]
=  \frac{ c_2 |u|^{1 - \lambda} }{   t^{1 + \lambda/2}  } 
\frac{ |z|^{2\lambda} \sin\left( \lambda \op{arg}(z) \right)^2 }{  |z^{\lambda} - u^{\lambda} |^2}  
 \exp\left( - \frac{  |z|^2}{2  t }  \right) \, dz  \,du ,
\eqe 
for a constant $c_2 > 0$ depending only on $\gamma$. 
By integrating out $z$ in~\eqref{eqn-joint-law}, we get an exact formula for the first factor in the numerator in~\eqref{eqn-bayes}. 

\begin{lemma} \label{lem-tau-marginal}
For $t > 0$, the $\P^\eps$-conditional law of $Z_\tau$ given $\{\tau  >t\}$ is given by
\eqb \label{eqn-tau-marginal}
\P^\eps\left[ Z_\tau \in \, du \,|\, \tau  > t \right] 
=  \frac{c_3 |u|^{1-\lambda}}{t^{1+\lambda/2}} \left( t e^{-\frac{|u|^2}{2t} } + \frac{ 2^\lambda t^{1+\lambda} }{ |u|^{2\lambda}}   \ul \Gamma\left( 1 + \lambda , \frac{|u|^2}{2t} \right)  \right)\,du
\eqe
where $c_3 > 0$ is a constant depending only on $\gamma$ and 
\eqb \label{eqn-truncated-gamma}
 \ul \Gamma\left( a,x \right)  = \int_0^x y^{a-1} e^{-y} \,dy 
\eqe
is the truncated $\Gamma$-function. 
\end{lemma}
\begin{proof}
From~\eqref{eqn-joint-law}, we find that  
\allb \label{eqn-Z_tau-law}
&\P^\eps\left[ Z_\tau \in \,du    \,|\, \tau > t \right]
 =   \frac{ c_2 |u|^{1 - \lambda} }{   t^{1 + \lambda/2}  }  f_t(u)     \,du  \notag \\
&\qquad  \text{where} \quad 
f_t(u) := 
\int_{\mcl C_\theta} \frac{ |z|^{2\lambda} \sin\left( \lambda \op{arg}(z) \right)^2 }{  |z^{\lambda} - u^{\lambda} |^2}  
 \exp\left( - \frac{  |z|^2}{2  t }  \right) \, dz  .
\alle
To evaluate the integral $f_t(u)$, we first switch to polar coordinates to get
\eqbn
f_t(u)  =  \int_0^\infty   r^{2\lambda+1} \exp\left(   - \frac{  r^2}{2  t }  \right)  \left(  \int_0^\theta \frac{\sin(\lambda \phi)^2}{|r^\lambda e^{ i \lambda \phi} - u^\lambda|^2} \, d\phi \right) \, dr .
\eqen
Making the change of variables $\phi  =\wt\phi/\lambda$, $d\phi = \lambda^{-1} d\wt\phi$ and applying Lemma~\ref{lem-residue-int} below with $s = r^\lambda / |u|^\lambda$ shows that the inner integral equals 
\eqb \label{eqn-trig-int}
\frac{1}{\lambda} \int_0^\pi \frac{\sin(\wt\phi)^2}{|r^\lambda e^{ i \wt\phi} - u^\lambda|^2} \, d\wt\phi 
= 
\begin{cases}
\frac{\pi }{2 \lambda r^{2\lambda}  } ,\quad &r \geq |u| \\
\frac{\pi}{2 \lambda |u|^{2\lambda} },\quad &r < |u|   
\end{cases}   .
\eqe 
Hence, 
\alb
f_t(u) 
&=   \frac{\pi}{2\lambda} \left(  
\int_{|u|}^\infty   r    \exp\left(   - \frac{  r^2}{2  t }  \right)   \, dr + 
\frac{1}{|u|^{2\lambda}} \int_0^{|u|}  r^{2\lambda+1}    \exp\left(   - \frac{  r^2}{2  t }  \right)   \, dr
\right) \notag\\
&= \frac{\pi}{2\lambda} \left( t e^{-\frac{|u|^2}{2t} } + \frac{ 2^\lambda t^{1+\lambda} }{ |u|^{2\lambda}}   \ul \Gamma\left( 1 + \lambda , \frac{|u|^2}{2t} \right)  \right) .
\ale 
Note that to evaluate the second integral in the first line, we made the substitution $r = \sqrt{2t s}$. 
Combining this with~\eqref{eqn-Z_tau-law} now gives~\eqref{eqn-tau-marginal}.
\end{proof}

The following lemma was used in the proof of Lemma~\ref{lem-tau-marginal} to evaluate the integral~\eqref{eqn-trig-int}.

\begin{lemma} \label{lem-residue-int}
For $s > 0$,  
\eqb \label{eqn-residue-int}
\int_0^\pi \frac{(\sin \phi)^2}{|s e^{ i\phi} -1|^2} \,d\phi = 
\begin{cases}
\frac{\pi}{2},\quad &s \leq 1 \\
\frac{\pi}{2s^2},\quad &s > 1
\end{cases}
\eqe
\end{lemma}
\begin{proof}
We will argue using the formula for the Poisson kernel for the disk. The integral can also be evaluated using the residue theorem. 
First consider the case when $s > 1$. 
Since $|s e^{ i\phi} -1|^2 = |s e^{- i \phi}-1|^2$,  
\eqb \label{eqn-symmetry-int}
\int_0^\pi \frac{( \sin \phi )^2}{|s e^{i \phi} - 1|^2} \,d\phi = \frac12 \int_{-\pi}^\pi \frac{( \sin \phi )^2}{|s e^{i \phi} - 1|^2} \,d\phi .
\eqe 
Let $\frk f(z) := \re\left( \frac{1-z^2}{2}\right)$, so that $\frk f$ is harmonic and $\frk f(e^{ i \phi}) = (\sin \phi)^2$. 
By the Poisson kernel formula for the disk, 
\eqb \label{eqn-poisson-kernel}
 \int_{-\pi}^\pi \frac{( \sin \phi )^2}{|s e^{ i \phi} - 1|^2} \,ds 
 = \frac{1}{s^2(1-1/s^2)} \int_{-\pi}^\pi \frk f( e^{i\phi} ) \frac{ 1-1/s^2 }{ |e^{i\phi} - 1/s|^2}  \, d\phi
 = \frac{2\pi \frk f(1/s)}{s^2 (1-1/s^2) } 
 = \frac{\pi}{s^2} .
\eqe 
Combining this with~\eqref{eqn-symmetry-int} gives~\eqref{eqn-residue-int} in the case when $s > 1$. For $s  \in (0,1)$, we have $|s e^{i\phi} - 1| = |e^{i \phi} - s| = s |e^{i \phi}/s-1|$, so the value of the integral in~\eqref{eqn-residue-int} for $s$ is $1/s^2$ times the value of the integral with $1/s > 1$ in place of $s$. Thus~\eqref{eqn-residue-int} for $s \in (0,1)$ follows from~\eqref{eqn-residue-int} for $s>1$. The case $s=1$ follows since the integral depends continuously on $s$. 
\end{proof}


Using Lemma~\ref{lem-tau-marginal}, we can get an asymptotic formula for the denominator in~\eqref{eqn-bayes}. 

\begin{lemma} \label{lem-Z_tau-asymp}
For each fixed $u\in \bdy\mcl C_\theta$, it holds as $\ep\rta 0$ that
\eqb \label{eqn-Z_tau-asymp}
\P^\eps\left[ Z_\tau \in \,du    \right] 
 = (c_4+o_\ep(1))      |u|^{1 - 3\lambda}       \ep^{\lambda/2}  du  ,  
\eqe 
for a constant $c_4 > 0$ depending only on $\gamma$. 
\end{lemma}
\begin{proof}
By definition the $\P^\eps$-law of $Z_\tau$ is the same as the $\P^\eps$-conditional law of $Z_\tau$ given $\{\tau > \ep\}$. 
Since $\ul\Gamma(1+\lambda , |u|^2/(2\ep) ) \rta \Gamma(1+\lambda)$ and $e^{-|u|^2/(2\ep)}$ decays faster than any power of $\ep$ as $\ep\rta 0$, we obtain~\eqref{eqn-Z_tau-asymp} for an appropriate choice of $c_4$ by setting $t=\ep$ in Lemma~\ref{lem-tau-marginal}.  
\end{proof}

Next we will deal with the factor $\P^\eps[\tau > t]$ appearing in~\eqref{eqn-bayes}, again using formulas from~\cite{shimura1985}. 
 
\begin{lemma} \label{lem-cone-prob}
For each fixed $t \geq \ep$, 
\eqb \label{eqn-cone-prob}
\P^\eps\left[ \tau > t \right] = \left(  c_5 + o_\ep(1) \right) \ep^{ \lambda/2 } t^{-\lambda/2} ,\quad \text{as $\ep\rta 0$},
\eqe  
for a constant $c_5>0$ depending only on $\gamma$. 
\end{lemma}
\begin{proof}
Let $\BB Q_z$ for $z\in\BB C$ be the law of a standard planar Brownian motion $Z^{\op{BM}}$ started from $z$ and let $\tau^{\op{BM}}$ be the exit time of such a Brownian motion from $\mcl C_\theta$. Let $\E^\eps$ be the expectation under $\P^\eps$. By the Markov property of Brownian motion, for $t\geq \ep$, 
\eqb
\BB P^\eps\left[ \tau > t \right] = \BB E^\eps\left[ \BB Q_{Z_\eps} \left[ \tau^{\op{BM}} > t-\ep \right] \right] .
\eqe
By~\cite[Equation (4.2)]{shimura1985}, for each $\delta \in (0,1)$ there exists $r_\delta  = r_\delta( t,\gamma) > 0$ such that 
\eqbn
\left| \BB Q_z\left[ \tau^{\op{BM}}  > t-\ep \right]  - c_5' (t-\ep)^{-\lambda/2} |z|^\lambda \sin(\lambda\op{arg} z) \right| \leq \delta |z|^\lambda    ,\quad \forall z \in \mcl C_\theta \quad \text{with} \quad |z| \leq r_\delta ,
\eqen
where $c_5' > 0$ is a constant depending only on $\gamma$. 
Setting $z = Z_\eps$ and taking expectations of both sides w.r.t. $\E^\eps$, we get that
\allb \label{eqn-tau_ep-asymp}
 \left| \BB P^\eps\left[ \tau > t \right]  
-   (t-\ep)^{-\lambda/2}  \BB E^\eps\left[   |Z_\eps|^\lambda \sin\left(\lambda\op{arg} Z_\eps\right)    \right]   \right|
 \leq \delta \BB E^\eps\left[  | Z_\eps|^\lambda  \right]  + \BB E^\eps\left[ (1 + c_5' (t-\ep)^{-\lambda/2} | Z_\eps|^\lambda ) \BB 1_{( |Z_\eps| > r_\delta )} \right]  .
\alle
By Brownian scaling the $\P^\eps$ law of $\ep^{-1/2} Z_\ep$ is equal to the $\P^1$ law of $Z_1$, so
\allb \label{eqn-tau_ep-to-tau_1}
&\BB E^\eps\left[  |Z_\eps |^\lambda  \sin\left( \lambda\op{arg} Z_\eps \right)  \right] = \ep^{\lambda/2}  \BB E^1\left[  | Z_1|^\lambda  \sin\left( \lambda\op{arg} Z_1 \right)  \right] ,\quad 
\BB E^\eps\left[  |Z_\eps |^\lambda    \right] = \ep^{\lambda/2}  \BB E^1\left[  | Z_1|^\lambda   \right]
 \quad \text{and} \notag \\
& \BB E^\eps\left[ (1 + c_5' (t-\ep)^{-\lambda/2} |Z_\eps|^\lambda ) \BB 1_{( |Z_\eps| > r_\delta )} \right] 
\leq O_\ep(1) \BB E^1\left[ (1 +  | Z_1|^\lambda ) \BB 1_{( |Z_1| > \ep^{-1/2} r_\delta )} \right] .
\alle
By the explicit formula for the $\P^1$-density of $Z_1$ given in~\eqref{eqn-Z_1-density}, we find that $ \BB E^1\left[  | Z_1|^\lambda  \sin\left( \lambda\op{arg} Z_1 \right)  \right]$ and $\BB E^1\left[  |Z_1|^\lambda   \right]$ are finite constants depending only on $\gamma$ and that $\BB E^1\left[ (1 +  | Z_1|^\lambda ) \BB 1_{( | Z_1| > \ep^{-1/2} r_\delta )} \right]$ decays faster than any positive power of $\ep$ as $\ep\rta 0$. 
Since $\delta \in (0,1)$ is arbitrary, combining this with~\eqref{eqn-tau_ep-to-tau_1} and plugging the result into~\eqref{eqn-tau_ep-asymp} yields~\eqref{eqn-cone-prob} for an appropriate choice of $c_5$.
\end{proof} 

 
\begin{proof}[Proof of Theorem~\ref{thm: area of disk}]
By plugging the formulas of Lemmas~\ref{lem-tau-marginal},~\ref{lem-Z_tau-asymp}, and~\ref{lem-cone-prob} into~\eqref{eqn-bayes}, we get
\eqbn
\mu^\#_{\cC_\theta, \eps}(0,u) \{\tau > t\} 
 =(c_6 +o_\ep(1))      \frac{ |u|^{2\lambda}}{t^{1+\lambda}} \left( t e^{-\frac{|u|^2}{2t} } + \frac{ 2^\lambda t^{1+\lambda} }{ |u|^{2\lambda}}   \ul \Gamma\left( 1 + \lambda , \frac{|u|^2}{2t} \right)  \right)     , 
\eqen
for a constant $c_6 > 0$ depending only on $\gamma$.
Setting $u =  ( ( \BB a \sin\theta)^{-1} ,0)$ and sending $\ep \rta 0$ (recall the discussion after~\eqref{eqn-area-law}) now shows that 
\allb \label{eqn-cone-law}
\mu^\#_{\cC_\theta}(0,u)\{\tau > t\}
= \frac{ c_6  }{ ( \BB a \sin\theta)^{2\lambda}  t^{1+\lambda}} \left( t e^{-\frac{1}{2 ( \BB a \sin\theta)^2  t} } +   2^\lambda  ( \BB a \sin\theta)^{2\lambda}  t^{1+\lambda}    \ul \Gamma\left( 1 + \lambda , \frac{1}{2 ( \BB a \sin\theta)^2  t} \right)  \right)  .
\alle
As $t\rta 0$, the left side of~\eqref{eqn-cone-law} converges to 1 and the right side converges to $c_6 2^\lambda \Gamma(1+\lambda)$, so $c_6 = 2^{-\lambda} \Gamma(1+\lambda)^{-1}$. 
Differentiating~\eqref{eqn-cone-law} with respect to $t$ and recalling that $\lambda = 4/\gamma^2$ and $\theta = \pi\gamma^2/4$ gives~\eqref{eqn-area-law}. 
\end{proof}

\section{Constructing a unit boundary length quantum disk from a quantum wedge}\label{section: pinching off a disk}

The goal of this section is to show that under suitable conditioning, we can ``pinch off'' a neighborhood near the marked point of a $\gamma$-quantum wedge to obtain a unit boundary length quantum disk. 

We embed a $\gamma$-quantum wedge in $\cS$ so that $-\infty$ (resp. $+\infty$) is the marked point with neighborhoods of infinite (resp. finite) mass, and explore the wedge from left to right. Roughly speaking, if we stop our exploration when the field becomes small, then condition on the quantum lengths of the unexplored boundary rays in $\R, \R+i\pi$ being close to $\frac12$, the remaining unexplored region resembles a $(\frac12, \frac12)$-length quantum disk. 

We restate this more precisely. Let $(\cS, h^{\cS}, -\infty, +\infty)$ be a $\gamma$-quantum wedge, and horizontally translate the field so that the field average process first attains the value $-r \ll 0$ on the vertical segment $[0,i\pi]$. Then $h := h^{\cS}|_{\cS_+}$ is a distribution on the positive strip $\cS_+$. Since the vertical field average at 0 is $-r \ll 0$ and the field average process has downward drift, the lengths $\nu_h(\R_+)$ and $\nu_h(\R_+ + i\pi)$ are typically on the order of $e^{-\gamma r/2} \ll 1$. If we condition on the rare event that $\nu_h(\R_+), \nu_h(\R_+ + i\pi) \approx \frac12$, then the field $h$ is ``close'' to that of a $(\frac12,\frac12)$-length quantum disk. 

The main technical difficulties in this section arise because we condition on $\nu_h(\R+), \nu_h(\R_+ + i\pi)$ lying in intervals whose lengths are exponentially short in $r$.
This is necessary to prove one of the equivalence-of-bottlenecks results (Proposition~\ref{prop:F given E}); c.f.\ Section~\ref{sec-outline}.
We now explain why this is necessary. With our setup, the segment $[0,i\pi]$ will correspond to to the ``pinch point'' defined via the quantum wedge field, and $h$ will be the field of the pinched-off region $\cS_+$. In Section \ref{subsection: decomposing wedge} we define a bottleneck via the space-filling $\SLE_{\kappa'}$ curve, in a way that specifies the exact boundary lengths to the right of the bottleneck. We need the quantum wedge pinch point $[0,i\pi]$ to be close to the space-filling $\SLE_{\kappa'}$ bottleneck in Euclidean distance to prove Proposition~\ref{prop:F given E}, so we need our quantum wedge bottleneck \eqref{eqn: E GFF} to specify \emph{both} lengths $\nu_h(\R+), \nu_h(\R_+ + i\pi)$ of the pinched-off region to an \emph{exponentially} close degree of precision.

Let $R$ be the rectangle $[0,S]\times [0,\pi]$ for some $S>0$; we keep $R$ fixed while sending $r \to \infty$. In Section~\ref{subsection: head of field conditioning on E} we prove Proposition~\ref{prop: head of GFF}, which describes the law of the triple $(h|_R, \nu_h(\R_+), \nu_h(\R_++i\pi))$ when we condition on $\{\nu_h(\R_+), \nu_h(\R_+ + i\pi) \approx \frac12\}$. In Section~\ref{subsection: pinched disk}, we prove Proposition~\ref{prop: conditioned on past, future is disk}, which roughly speaking says that if we condition on $\{\nu_h(\R_+), \nu_h(\R_+ + i\pi) \approx \frac12\}$ and on ``typical realizations'' of the triple $(h|_R, \nu_h(\R_+), \nu_h(\R_+ + i\pi))$, then the field $h$ is close to a quantum disk. Roughly speaking, Section~\ref{subsection: head of field conditioning on E} studies the conditional law of the field close to $[0,i\pi]$, and Section~\ref{subsection: pinched disk} studies the conditional law of the field far from $[0,i\pi]$ when we additionally condition on the field close to $[0,i\pi]$.

\subsection{Law of ${(h|_R, \nu_h(\R_+), \nu_h(\R_+ + i\pi))}$ given $E_{r,K,q_1,q_2}$} \label{subsection: head of field conditioning on E}
One of the main goals of this section is to prove that, if we condition on $\nu_h(\R_+)$ and $\nu_h(\R_+ + i\pi)$ being in intervals close to $\frac12$, then the conditional law of $h$ near $[0,i\pi]$ is close to an unconditioned GFF with an upward linear drift added to the field average process, and moreover $\nu_h(\R_+)$ and $\nu_h(\R_+ + i\pi)$ are close to being independent reals drawn uniformly from these intervals. 

Informally speaking, although the field average process of the unconditioned field has a downward drift of $(\gamma - Q)$, conditioning on $\nu_h(\R_+)$ and $\nu_h(\R_+ + i\pi)$ being large causes the field average process to grow.

We emphasize that in this section, $h := h^\cS|_{\cS_+}$ is a field on $\mcl S_+$ rather than on all of $\mcl S$. We will work with the equivalent definition~\eqref{eqn: unconditioned h} of $h$ (see Lemma~\ref{lem-wedge-tip} for a proof of equivalence).

\begin{proposition} \label{prop: head of GFF}
Let $\widetilde h$ be a Neumann GFF on $\cS$ restricted to $\cS_+$ and normalized so its average on $[0,i\pi]$ is $0$. For $r > 0$, let
\begin{equation} \label{eqn: unconditioned h}
h = h_r  =  \widetilde h  + (\gamma - Q) \Re (\cdot) - r. 
\end{equation}
For $K > 1$ and $q_1, q_2 > 0$, define the event
\begin{equation}\label{eqn: E GFF}
E_{r,K,q_1,q_2} = \left\{ \nu_ h (\R_+) \in \left[q_1,q_1 + e^{\gamma (-r+K)/2}\right], \nu_ h (\R_+ + i\pi) \in \left[q_2,q_2 + e^{\gamma (-r+K)/2}\right] \right\}. 
\end{equation}
Fix any $S > 0$, and let $R$ be the rectangle $[0,S] \times [0,\pi]$.

Then there exist functions $q_1 = q_1(r), q_2 = q_2(r)$ satisfying $\lim_{r \to \infty} q_1(r) = \lim_{r \to \infty} q_2(r) = \frac12$, such that as $r \to \infty$, the total variation distance between the following laws of triples of random variables goes to zero:
\begin{itemize}
\item The (normalized field, length, length) triple given by $(h|_R + r, \nu_h(\R_+), \nu_h(\R_++i\pi))$ conditioned on $E_{r,K,q_1(r) , q_2(r)}$ (the field is normalized in the sense that the average of $h|_R + r$ on $[0,i\pi]$ is zero);

\item The mutually independent (field, length, length) triple given by $(\phi, V_1, V_2)$, where $V_j$ is sampled from $\mathrm{Unif}([q_j(r) , q_j(r) + e^{\gamma(-r+K)/2}])$ for $j = 1,2$, and $\phi$ is a field on $R$ defined via 
\begin{equation}\label{eqn: initial field}
\phi =  \left(  \widehat h + (Q - \gamma) \Re (\cdot) \right)\Big|_R 
\end{equation}
with $\widehat h$ a Neumann GFF on $\cS$ normalized to have mean zero on $[0,i\pi]$.
\end{itemize}
Moreover, this holds when we instead condition on $E_{r,K,q_1(r),q_2(r)} \cap E'_{r,\beta}$ and send first $r \to \infty$, then $\beta \to \infty$, where 
\begin{equation}\label{eqn: E'}
E'_{r,\beta} = \{ \exists u \geq 0 \text{ such that average of }h \text{ on } [u,u+i\pi] \text{ is at least } -\beta\}.
\end{equation}
\end{proposition}

Since $\lim_{r \to \infty} q_1(r) = \lim_{r \to \infty} q_2(r) = \frac12$, the event $E_{r,K,q_1,q_2}$ can be very roughly described as $\{\nu_h(\R_+) , \nu_h(\R_+ + i\pi) \approx \frac12 \}$. For large $r\gg 0$, Proposition~\ref{prop: head of GFF} describes the law of $(h|_R, \nu_h(\R_+), \nu_h(\R_+ + i\pi))$ conditioned on $E_{r,K,q_1,q_2}$.

For notational convenience, we will often just say $q_1,q_2$ rather than $q_1(r), q_2(r)$. 

Here is a rough explanation for why we expect Proposition \ref{prop: head of GFF} to hold. \cite[Lemma A.4]{wedges} states that uniformly in $r$, the probability $\P[E'_{r,\beta} \mid \nu_h(\R_+) + \nu_h(\R_+ + i\pi) > 1]$ is close to 1 for sufficiently large $\beta$. Conditioning on $E'_{r,\beta}$ and writing $\tau_{-\beta}$ for the leftmost point at which the average of $h$ on $[\tau_{-\beta}, \tau_{-\beta} + i\pi]$ equals $-\beta$, the lengths $\nu_h(\R_+)$ and $\nu_h(\R_+ + i\pi)$ are very close to being a function of $h|_{\cS_+ + \tau_{-\beta}}$, since the field is very small to the left of $\tau_{-\beta}$. Finally, because it takes a long time for the field average to grow from $-r$ to $-\beta$, by Proposition~\ref{prop: adaptation of ig4 proposition 2.10} we see that $h|_R$ is almost independent from $h|_{\cS_+ + \tau_{-\beta}}$. Thus, the event that $h$ has large boundary length is almost independent from $h|_	R$. It is therefore not a stretch to expect that, conditional on $E'_{r,\beta}$, the event $E_{r,K,q_1,q_2}$ is almost independent from $h|_R$.

Although the unconditioned field average process of $h$ has downward drift, when we condition on $E'_{r,\beta}$, the field average process will initially have upward linear drift until it hits $-\beta$, and subsequently have downward linear drift. This explains why we expect the field average process of $h|_R$ to have upward drift (see \eqref{eqn: initial field}). Furthermore, if we further condition on $h|_R$ then define the conditional density function $d: \R_+^2 \to \R$ for the conditional law of the side lengths $(\nu_h(\R_+), \nu_h(\R_+ + i\pi))$, then by the continuity of $d$ we expect it to be almost constant in a small neighborhood of $(\frac12, \frac12)$. In particular, $d$ should be almost constant in the square $[q_1, q_1 + \exp(\gamma (-r+K)/2)] \times [q_2, q_2 + \exp(\gamma (-r+K)/2)] $, so when we condition on $h|_R$ and on $E_{r,K,q_1,q_2}$, the side lengths should be close to being drawn from $\mathrm{Unif}[q_1, q_1 + \exp(\gamma (-r+K)/2)] \otimes \mathrm{Unif}[q_2, q_2 + \exp(\gamma (-r+K)/2)]$. While this brief explanation of Proposition \ref{prop: head of GFF} is quite imprecise, it is morally correct, and made formal in this section. 

We expect this proposition to hold even if we make for all $r$ the choice $q_1 = q_2 = \frac12$, but for technical reasons it is easier to avoid proving it for this specific choice of $q_1,q_2$. The reason for this is that at one point in the proof we will convert a statement about the average over a range of pairs $(q_1,q_2)$ to a statement about one particular pair (see Lemma~\ref{lem: choose q1, q2}).

We introduce some notation. Recall the events $E'_{r,\beta}$ and $E'_\beta$ defined in \eqref{eqn: E'} and \eqref{eqn: large quantum disk} respectively ($E'_\beta$ is the event that the field average process of the quantum disk field attains the value $-\beta$). For a quantum disk $(\cS, \psi, +\infty, - \infty)$ conditioned on $E'_\beta$, the random pair $(\nu_\psi(\R), \nu_\psi(\R + i\pi))$ has a probability density function $d^\beta_{\mathrm{disk}}(y_1,y_2)$ with respect to Lebesgue measure $d y_1 dy_2$ (this can be easily seen by considering the coefficients of two suitable functions in the orthonormal basis expansion of the mean-zero part of the field, as in Section~\ref{subsection: continuity of field}). Likewise, if we let $h$ be as in \eqref{eqn: unconditioned h} and condition on $E_{r,\beta}'$, the random pair $(\nu_h(\R_+), \nu_h(\R_+ + i\pi))$ has a probability density function $d^{\beta,r}_{\GFF} ( y_1, y_2)$ with respect to Lebesgue measure. 

By the Markov property of the GFF, we can sample $h$ by first sampling its restriction to the rectangle $R = [0,S] \times [0,\pi]$, then sampling $h|_{\cS_+ \backslash R }$ as a GFF with Neumann boundary conditions on $(\R_+ + S) \cup (\R_+ + S + i\pi)$ and Dirichlet boundary conditions on $[S,S+i\pi]$ specified by $h|_R$. Thus, if $\varphi$ is a distribution on $R$ with mean zero on $[0,i\pi]$, we can make sense of the regular conditional law of $h$ conditioned on the probability zero event $\{ h|_R = \varphi - r\}$ by sampling $h|_{\cS_+ \backslash R}$ as above. For each such $\varphi$, define the density $d^{\beta,r}_{\GFF} ( y_1, y_2\: \big| \: \varphi )$ for the random variable $(\nu_h(\R_+), \nu_h(\R_++ i\pi))$ conditioned on $E'_{r,\beta} \cap \{h|_R = \varphi - r\}$. More generally, for an event $A$, let $d^{\beta,r}_{\GFF}(y_1, y_2 \mid A)$ denote the conditional density of $(\nu_h(\R_+), \nu_h(\R_+ + i\pi))$ when we condition on $E'_{r,\beta}\cap A$. Likewise, for an event $A$, let $d^\beta_{\mathrm{disk}} (y_1, y_2 \mid A)$ denote the conditional density of $(\psi(\R), \psi(\R + i\pi))$ when we condition on $E'_\beta \cap A$. 

We now explain the broad outline for this section. In Lemma~\ref{lem: coupling of disk and GFF}, we show that when we appropriately truncate the field $\psi$ of a quantum disk conditioned on $E'_\beta$, the resulting field has exactly the law of $h$ conditioned on $E'_{r,\beta}$, so it is natural to study $h$ via comparisons with $\psi$. This follows almost immediately from the explicit field descriptions Proposition \ref{prop: large quantum disk} and Remark~\ref{remark: GFF on strip}. We also describe the large $r$ law of $h|_R +r$ conditioned on $E'_{r,\beta}$. In Lemma~\ref{lem: approx coupling of disk and GFF given varphi}, for any fixed distribution $\varphi$ on $R$ we produce a coupling between $h$ conditioned on $E'_{r,\beta} \cap \{ h|_R = \varphi - r\}$ and $\psi$ conditioned on $E'_\beta$. Using these couplings, in Lemma~\ref{lem: conditioned on start, field looks like disk} we obtain lower bounds $d^{\beta, r}_{\GFF} (y_1,y_2) \geq (1-o_r(1))d^{\beta}_{\mathrm{disk}} (y_1,y_2)$ for $y_1, y_2 \in [1/2,1]$, and $d^{\beta, r}_{\GFF} (y_1,y_2 \mid \varphi) \geq (1-o_r(1))d^{\beta}_{\mathrm{disk}} (y_1,y_2)$ for $y_1, y_2 \in [1/2,1]$ for fixed $\varphi$ and $\beta$. Using the first lower bound, in Lemma \ref{lem: choose q1, q2} we obtain a crude but matching upper bound on $d^{\beta, r}_{\GFF}(y_1,y_2)$. Combining the lower and upper bounds with Bayes' theorem tells us that the law of $h|_R$ given $E_{r,K,q_1,q_2}$ is close to its law given $E'_{r,\beta}$; by Lemma~\ref{lem: coupling of disk and GFF} this is the law of the field of Proposition \ref{prop: head of GFF}. Similar density arguments yield the law of the triple $(h|_R, \nu_h(\R_+), \nu_h(\R_+ + i\pi))$.

In this section, we have two parameters $\beta, r$ that we send to infinity, and one parameter $\eps$ that we send to $0^+$. We will always send $r\to \infty$, $\beta \to \infty$, $\eps \to 0$ in that order. We write $o_r(1)$ to mean an error term that, for each fixed $\eps,\beta$, goes to zero as $r \to \infty$; the error $o_r(1)$ need not be uniform in $\beta, \eps$. We write $o_\beta(1)$ to mean an error term that is, for any fixed $\eps$, close to 0 for all sufficiently large $\beta$ and all sufficiently large $r > r_0(\beta)$. In particular, we always have $o_\beta(1)+o_r(1) = o_\beta(1)$.

In the following lemmas, we will work with multiple fields. We will sometimes add superscripts to functions and stopping times to denote which object (GFF or disk) we are discussing. For instance, $\tau^\psi_s$ is defined by $\inf \{ x : \text{average of }\psi \text{ on } [x, x+i\pi] \text{ is }s\}$, and $\tau^h_s$ the analogous value for the field $h$. 

\begin{lemma}[Coupling of $\psi$ given $E'_\beta$ and $h$ given $E'_{r,\beta}$]\label{lem: coupling of disk and GFF}
Let $(\cS, \psi, +\infty, -\infty)$ be a quantum disk conditioned on $E'_{\beta}$. Then for any $r > \beta$, the field $\psi(\cdot + \tau^\psi_{-r})|_{\cS_+}$ has the same law as that of $h = h_r$ conditioned on $E'_{r,\beta}$ (note that $h$ is a field on $\mcl S_+$). Moreover, we have 
\eqb\label{eqn-initial-field}
\lim_{r \to \infty} d_{TV}\left(h|_R \text{ conditioned on }E'_{r,\beta}, \phi - r \right) = 0,
\eqe
where $\phi$ is the field on $R$ defined in \eqref{eqn: initial field}.
\end{lemma}
\begin{proof}
Proposition \ref{prop: large quantum disk} tells us that $\psi(\cdot + \tau^\psi_{-r})|_{\cS_+}$ conditioned on $E'_\beta$ has independent projections to $\cH_1(\cS_+)$ and $\cH_2(\cS_+)$, with these projections being:
\begin{itemize}
\item Brownian motion with variance $2$, started at $-r$ and having an upward linear drift $(Q- \gamma)$ until it hits $-\beta$, then having a downward drift of $(\gamma - Q)$;

\item The projection of a Neumann GFF on $\cS$ to $\cH_2(\cS)$, restricted to $\cS_+$. 
\end{itemize}
By Remark~\ref{remark: GFF on strip} and \cite[Lemma 3.6]{wedges}, this is precisely the same as the  law of $h$ conditioned on $E_{r,\beta}'$, so we may couple $h$ and $\psi$ so that $h = \psi(\cdot + \tau^\psi_{-r})|_{\cS_+}$ almost surely. 
Finally, \eqref{eqn-initial-field} holds since the probability that the field average process of $h$ within $R$ hits $-\beta$ goes to zero as $r \to \infty$. 
\end{proof}

The above coupling is exact, but if we further specify the restriction of $h$ to the rectangle $R = [0,S] \times [0,\pi]$, we can still obtain an approximate coupling.

\begin{lemma}[Approximate coupling of $\psi$ given $E'_\beta$ and $h$ given $E'_{r,\beta} \cap \{ h|_R = \varphi - r\}$]\label{lem: approx coupling of disk and GFF given varphi}
Let $\varphi$ be the restriction of a Neumann GFF on $\cS$ to $R$ normalized to have zero mean on $[0,i\pi]$. Then $\varphi$-a.s. the following is true when we fix $\varphi$.
We can couple the quantum disk field $\psi$ conditioned on $E'_\beta$ with the GFF $h$ conditioned on $E'_{r,\beta} \cap \{ h|_R = \varphi - r\}$ so that with probability $1-e_1$ we have $h(\cdot + \tau^h_{-r/2})|_{\cS_+} = \psi(\cdot + \tau^\psi_{-r/2} )|_{\cS_+}$. Here, $e_1 = e_1(\varphi, \beta, r)$ satisfies $\lim_{r \to \infty} e_1 = 0$ for each fixed $\varphi, \beta$.
\end{lemma}
\begin{proof}
We note that the projections of each of these fields to $\cH_1(\cS_+)$ and $\cH_2(\cS_+)$ are independent, so it suffices to couple their projections separately. As in the previous lemma, we know that these projections to $\cH_1(\cS_+)$ have exactly the same law: Brownian motion with variance 2 started at $-\frac{r}2$, with upward drift of $(Q - \gamma)$ until it hits $-\beta$, then downward drift of $(\gamma  -Q)$ subsequently. 

Next, observe that $\tau^h_{-r/2} \to \infty$ in probability as $r \to \infty$, and that $\tau^h_{-r/2}$ is independent of the projection of $h$ to $\cH_2(\cS_+)$. Thus by Proposition~\ref{prop: adaptation of ig4 proposition 2.10}, we see that as $r \to \infty$ the total variation distance between the laws of the following two fields goes to zero as $r \to \infty$:
\begin{itemize}
\item The projection of $h(\cdot + \tau_{-r/2}^h)|_{\cS_+}$ to $\cH_2(\cS_+)$;
\item The projection of a Neumann GFF on $\cS$ to $\cH_2(\cS)$, restricted to $\cS_+$. 
\end{itemize}
By the definition of the quantum disk, this latter law is the law of the projection of $\psi(\cdot + \tau^\psi_{-r/2})|_{\cS_+}$ to $\cH_2(\cS_+)$. We conclude that we can couple $h$ conditioned on $E'_{r,\beta} \cap \{ h|_R = \varphi - r\}$ and $\psi$ conditioned on $E'_\beta$ so that with probability $1-o_r(1)$ we have $h(\cdot + \tau_{-r/2}^h)|_{\cS_+} = \psi(\cdot + \tau^\psi_{-r/2})|_{\cS_+}$.
\end{proof}

The above two couplings allow us to lower bound the densities $d^{\beta,r}_{\GFF}(\cdot, \cdot \mid \varphi)$ and $d^{\beta, r}_{\GFF} (\cdot, \cdot)$.

\begin{lemma}[Lower bounds on $d^{\beta, r}_{\GFF}(\cdot, \cdot \: \big| \: \varphi)$ and $d^{\beta, r}_{\GFF} (\cdot, \cdot)$]\label{lem: conditioned on start, field looks like disk}
Let $\varphi$ be the restriction of a Neumann GFF on $\cS$ to $R$ normalized to have zero mean on $[0,i\pi]$. Then $\varphi$-a.s. we have 
\begin{equation}\label{eqn: lower bound conditioned GFF density}
d^{\beta,r}_{\GFF} ( y_1, y_2\: \big| \: \varphi ) \geq (1-e_1)d^\beta_{\mathrm{disk}}(y_1, y_2) \qquad \text{ uniformly over }y_1, y_2 \in \left [\frac12, 1 \right ],
\end{equation}
where the error $e_1 = e_1(\varphi, \beta,r)$ satisfies for each fixed $\varphi, \beta$ the limit $ \lim_{r \to \infty} e_1 = 0$.

Similarly, we have for fixed $\beta$ that 
\begin{equation}\label{eqn: lower bound density GFF}
d^{\beta,r}_{\GFF}(y_1, y_2) \geq (1-o_r(1)) d^\beta_{\mathrm{disk}} (y_1,y_2) \qquad \text{ uniformly over }y_1, y_2 \in \left [\frac12, 1 \right ].
\end{equation}
\end{lemma}
\begin{proof}
We will prove \eqref{eqn: lower bound conditioned GFF density} using Lemma~\ref{lem: approx coupling of disk and GFF given varphi}; the proof of~\eqref{eqn: lower bound density GFF} using Lemma~\ref{lem: coupling of disk and GFF} is the same. By Lemma~\ref{lem: approx coupling of disk and GFF given varphi} we can couple the field $\psi$ of a quantum disk conditioned on $E'_\beta$ with $h$ conditioned on $E'_{r,\beta} \cap \{ h|_R = \varphi - r\}$ so that $\psi( \cdot - \tau^\psi_{-r/2})|_{\cS_+} = h(\cdot - \tau^h_{-r/2})|_{\cS_+}$ with probability $1-o_r(1)$. 
As in \eqref{eqn: field decomposition}, write
\begin{equation} \label{eqn: decomposition g}
h = X^h_{\Re \cdot} + f^h +\alpha_1^h f_1^h + \alpha_2^hf_2^h ,
\end{equation}
where $f_1^h$ and $f_2^h$ are compactly supported on $[\tau^h_{-\beta} - 3, \tau^h_{-\beta}] \times [0,\pi]$, and likewise decompose the field $\psi$ in the same way:
\begin{equation}
\psi = X^\psi_{\Re \cdot} + f^\psi +\alpha_1^\psi f_1^\psi + \alpha_2^\psi f_2^\psi.
\end{equation}
Note that $(X^h, f^h, \alpha_1^h, \alpha_2^h)$ is a function of $h$, and likewise the components of the $\psi$-decomposition are a function of $\psi$. 

Define $d^{\beta, r}_{\GFF}(\cdot, \cdot \mid \varphi, X^h, f^h)$ to be the probability density of $(\nu_h(\R_+), \nu_h(\R_+ + i\pi))$ conditioned on $E'_{r,\beta} \cap \{h|_R = \varphi-r\}$ and on the realizations of $X^h, f^h$ (i.e. it is the density w.r.t. the remaining randomness of $\alpha_1^h, \alpha_2^h$). This is a random function (depending on the random $X^h, f^h$) which is almost surely continuous because $\alpha_1^h, \alpha_2^h$ have continuous densities. Similarly define $d^\beta_{\mathrm{disk}}(\cdot, \cdot \mid X^\psi, f^\psi)$ to be the density of $(\nu_\psi(\R), \nu_\psi(\R + i\pi))$ conditioned on $E'_\beta$ and on the realizations of $X^\psi, f^\psi$. Let $\delta,\delta'> 0$ be values we choose later.
\medskip

\noindent
\textbf{Step 1:} In the probability space of the coupled random variables $(h, \psi)$ (with marginals given by $h$ conditioned on $E'_{r,\beta} \cap \{ h|_R = \varphi - r\}$ and $\psi$ conditioned on $E'_\beta$), let $A_{r,\delta'}$ be the event that the following both hold: 
\begin{align}
h(\cdot + \tau_{-r/2}^h)|_{\cS_+}&=\psi(\cdot + \tau_{-r/2}^\psi)|_{\cS_+}, &\text{(coupling holds)} \label{eqn: coupling with head}\\
d^{\beta, r}_{\GFF} (y_1,y_2 \: \big| \: \varphi, X^h, f^h) &\geq  d^\beta_\mathrm{disk} (y_1,y_2 \: \big| \: X^\psi, f^\psi ) - \delta' \quad &\text{ for all } y_1,y_2 \in \left[\frac12,1\right] \label{eqn: coupling lower bound}.
\end{align}
We show that for $r$ large we have $\P[A_{r,\delta'}] \geq 1-\delta'$. 

To see this, note that when the coupling holds we have $(X^h(\cdot+ \tau_{-r/2}^h)|_{\R_+}, f^h(\cdot + \tau_{-r/2}^h)|_{\cS_+})=(X^\psi(\cdot + \tau^\psi_{-r/2})|_{\R_+}, f^\psi(\cdot+ \tau^\psi_{-r/2})|_{\cS_+})$, and also $f_j^h(\cdot + \tau_{-r/2}^h) = f_j^\psi(\cdot + \tau^\psi_{-r/2})$ for $j=1,2$. Consequently,
\begin{align*}
 &d^{\beta, r}_{\GFF} (y_1,y_2 \: \big| \: \varphi, X^h, f^h) \\
 =&  d^\beta_\mathrm{disk} \left(y_1 + \nu_\psi((-\infty, \tau^\psi_{-r/2}] ) - \nu_h([0,\tau_{-r/2}^h]),y_2 + \nu_\psi((-\infty,\tau^\psi_{-r/2}]+i\pi) - \nu_h([0,\tau_{-r/2}^h] + i\pi)\: \Big| \: X^\psi, f^\psi \right) .
\end{align*}
By Corollary~\ref{cor-medium-boundary}, we know that with probability $1-o_r(1)$ each of $\nu_h([0,\tau_{-r/2}^h])$, $\nu_h([0,\tau_{-r/2}^h] + i\pi)$, $\nu_\psi((-\infty, \tau^\psi_{-r/2}] )$, $\nu_\psi((-\infty, \tau^\psi_{-r/2}] +i\pi)  $ is of order $o_r(1)$, so $\P[A_{r,\delta'}] \geq 1-\delta'$ follows from the uniform continuity of $d^\beta_\mathrm{disk}(\cdot, \cdot \mid X^\psi, f^\psi)$ on $\left[\frac12, 1 \right]^2$.
\medskip

\noindent
\textbf{Step 2:} We use a compactness argument. Let $\cL^{\beta, y_1, y_2}_\mathrm{disk}$ be the law of $(X^\psi, f^\psi)$ conditioned on $E'_\beta \cap \{(\nu_\psi(\R), \nu_\psi(\R + i\pi)) = (y_1, y_2)\}$. Let $B^1, \cdots, B^N$ be a finite collection of open balls covering the square $\left[\frac12, 1\right]^2$ such that for any ball $B^j$ and pair of points $(y_1,y_2),(y_1',y_2') \in B^j$, the total variation distance between $\cL^{\beta, y_1, y_2}_\mathrm{disk}$ and $\cL^{\beta, y_1', y_2'}_\mathrm{disk}$ is at most $\delta$; the existence of these balls follows from the compactness of the square and Proposition \ref{prop: continuity of field}.

Observe that for each $j$ the law of the coupled random variables $(X^h, X^\psi, f^h, f^\psi)$ conditioned on $(\nu_\psi(\R), \nu_\psi(\R + i\pi)) \in B^j$ is absolutely continuous with respect to their unconditioned law. Thus, if we take $\delta'$ sufficiently small in Step 1, then for all $j=1,\dots, N$ and for all sufficiently large $r$ we have 
\begin{equation}\label{eqn: balls absolutely continuous}
\P[A_{r, \delta'}\mid (\nu_\psi(\R), \nu_\psi(\R + i\pi)) \in B^j] \geq 1 - \delta. 
\end{equation}
Notice that the law of $(X^\psi, f^\psi)$ conditioned on $E'_\beta \cap \{(\nu_\psi(\R), \nu_\psi(\R+i\pi)) \in B^j\}$ is a weighted average of the laws $\cL^{\beta,y_1',y_2'}_\mathrm{disk}$ for $(y_1',y_2') \in B^j$. Thus for any $(y_1,y_2) \in B^j$, the total variation distance between $(X^\psi, f^\psi) \sim \cL^{\beta,y_1,y_2}_\mathrm{disk}$ and $(X^\psi, f^\psi)$ conditioned on $\{(\nu_\psi(\R), \nu_\psi(\R + i\pi)) \in B^j\}$ is at most $\delta$. By taking expectations of \eqref{eqn: coupling lower bound} over $A_{r,\delta'}$ we obtain
\begin{align*}
d^{\beta, r}_{\GFF}(y_1,y_2 \mid \varphi) &= \E\left[d^{\beta,r}_{\GFF}(y_1,y_2 \mid \varphi, X^h, f^h)\right] \\
&\geq \E\left[\mathbbm{1}_{A_{r,\delta'}} d^{\beta}_{\mathrm{disk}}(y_1,y_2 \mid X^\psi, f^\psi)\right] -\delta' \\
&= \E\left[\mathbbm{1}_{A_{r,\delta'}} \mid (\nu_h(\R), \nu_h(\R + i\pi)) = (y_1,y_2) \right]d^\beta_\mathrm{disk}(y_1,y_2) -\delta' \\
&\geq \E\left[\mathbbm{1}_{A_{r,\delta'}} \mid (\nu_h(\R), \nu_h(\R + i\pi)) \in B^j \right]d^\beta_\mathrm{disk}(y_1,y_2) -\delta - \delta' \\
&\geq (1-\delta) d^\beta_\mathrm{disk} (y_1,y_2) - 2\delta,
\end{align*}
and taking $\delta$ small relative to $\inf_{[1/2,1]^2} d^\beta_\mathrm{disk}$ yields \eqref{eqn: lower bound conditioned GFF density}. 

To prove \eqref{eqn: lower bound density GFF}, we instead use the coupling of $h$ and $\psi$ provided by Lemma~\ref{lem: coupling of disk and GFF}. Steps 1 and 2 follow exactly as before.
\end{proof}

Lemma~\ref{lem: conditioned on start, field looks like disk} gives us a a lower bound on the density of GFF boundary lengths conditioned on $E'_{r,\beta}$. A slight modification of~\cite[Lemma A.4]{wedges}, combined with Markov's inequality, yields a (cruder) matching upper bound and an assertion that $\P[E'_{r,\beta} \mid E_{r,K,q_1,q_2}] \approx 1$. 

\begin{lemma}[Crude upper bound on $d^{\beta,r}_{\GFF}$ near $(q_1,q_2)$, and ${\mathbb{P} [E'_{r,\beta}|E_{r,K,q_1,q_2}] \approx 1}$]\label{lem: choose q1, q2}
Fix $\eps>0$ small and $N,K > 0$ large. There exists $e_2 = e_2(\beta, K,N,\eps)$ such that for $r$ sufficiently large in terms of $\beta,K,N,\eps$, we can choose $q_1, q_2 \in [\frac12, \frac12 + \eps]$ so that 
\begin{align}
\P[E_{r,\beta}' \mid E_{r,K,q_1,q_2}] &\geq 1 - e_2, \label{eqn: E' given E}\\
\P[E_{r,K,q_1,q_2} \mid E_{r,\beta}'] &\leq (1+e_2 ) e^{\gamma (-r+K)}d_{\mathrm{disk}}^\beta(q_1,q_2) , \label{eqn: most phi are improbable} \\
\P[\{\tau^h_{-\beta} < \sigma - N \} \cap E'_{r,\beta} \mid  E_{r,K,q_1,q_2}] &\geq 1- e_2, \label{eqn: beta far to the left}
\end{align}
where for fixed $K, N, \eps$ we have $e_2 \to 0$ as $\beta \to \infty$, and $\sigma$ is the unique real such that $h(\R_+ + \sigma) + h(\R_+ + i\pi + \sigma) = \frac12$.
\end{lemma}
Note that the RHS of \eqref{eqn: most phi are improbable} is roughly the probability of the event $\{ \nu_\psi(\R) \in [q_1, q_1 + \exp(\gamma (-r+K)/2)], \nu_\psi(\R + i\pi) \in [q_2, q_2 + \exp(\gamma (-r+K)/2)]\}$ conditional on $E'_\beta$; that is, it's roughly speaking the quantum disk equivalent of the LHS. The bound \eqref{eqn: beta far to the left} roughly says that the quantum area to the left of $\tau^h_{-\beta}$ is small; this estimate is not needed for the proof of Proposition~\ref{prop: head of GFF} but will be useful for Proposition~\ref{prop: conditioned on past, future is disk}. Of course~\eqref{eqn: beta far to the left} is stronger than~\eqref{eqn: E' given E}, but we include both for clearer referencing.

\begin{proof}
Just for this proof, we introduce some notation. Write 
\[\begin{gathered}
G_{r, \beta, \eps, N} := \{\tau^h_{-\beta} < \sigma - N \} \cap E'_{r,\beta},\\
E_{\mathsf{U} } := \left\{(\nu_h(\R_+), \nu_h(\R_+ + i\pi)) \in \mathsf{U} \right\} \quad \text{ for open } \mathsf U \subset \R^2_+. 
\end{gathered}\]
Define the square $\mathsf{S} := \left[ \frac12, \frac12 + \eps\right] \times \left[ \frac12, \frac12 + \eps\right]$. 
\medskip

\noindent\textbf{Step 1: Showing $\P[  G_{r, \beta, \eps, N}  \mid E_\mathsf{S}] \approx 1$.} Write $a = Q - \gamma > 0$. We show the following three inequalities:
\eqb\label{eq-three}
\P[ E_\mathsf{S} ] \gtrsim e^{-ar}, \quad \P[(E'_{r,\beta})^c \cap E_\mathsf{S}] = o_\beta(1)e^{-ar}, \quad \P[\{\tau^h_{-\beta} \geq \sigma - N \}\cap E'_{r,\beta} ] = o_\beta(1)e^{-ar}. 
\eqe
With these we obtain Step 1, since the first two inequalities imply that $\P[E'_{r,\beta} | E_\mathsf{S}] = 1-o_\beta(1)$, and the first and third inequalities imply $\P[\{ \tau^h_{-\beta} < \sigma - N\} \mid E'_{r,\beta} \cap E_\mathsf{S}] = 1-o_\beta(1)$, so combining these gives $\P[\{ \tau^h_{-\beta} < \sigma - N\} \cap E'_{r,\beta} \mid E_\mathsf{S}] = 1-o_\beta(1)$ as desired. 

By a standard Brownian motion computation, we have $\P[E'_{r,\beta}] \asymp e^{-a(r-\beta) }$. Since it's clear that $\P[E_{\mathsf{S}} \mid E'_{r,\beta=0}] > 0$ uniformly over $r > 0$, we get $\P[E_{\mathsf{S}}] \geq \P[E_{\mathsf{S}} \mid E'_{r,\beta=0}] \P[E'_{r,\beta = 0}] \asymp e^{-ar}$. This is  the first inequality of~\eqref{eq-three}.

Fix $p \in (0, \frac4{\gamma^2})$. Lemma~\ref{lem: coupling of disk and GFF} and Corollary~\ref{cor-large-boundary} tell us that for any $x < r$ we have $\P[\nu_h(\R_+) + \nu_h(\R_+ + i\pi) > \frac12 \mid E'_{r,x}] \lesssim e^{-\gamma x p /2}$. Thus 
\begin{align*}
&\P[\{ \nu_h(\R_+) +\nu_h(\R_+ + i\pi) > \frac12 \} \cap (E'_{r,\beta})^c] \notag\\
&\qquad \qquad \leq \sum_{x=\lfloor \beta+1 \rfloor}^{\lceil r \rceil} \P[\{ \nu_h(\R_+) +\nu_h(\R_+ + i\pi) > \frac12 \} \cap (E'_{r,x-1})^c \mid E'_{r,x}]\P[E'_{r,x}] \notag \\
&\qquad \qquad \lesssim  \sum_{x=\lfloor \beta+1 \rfloor}^{\lceil r \rceil} e^{-\gamma x p/2} e^{-a(r-x)}.
\end{align*}
Taking $p$ sufficiently close to $\frac{4}{\gamma^2}$, we have $\frac{\gamma p}2 \geq Q - \gamma = a$, and so the above sum contracts to $e^{-ar} e^{-\beta (\frac{\gamma p}2 - a)} = e^{-ar} o_\beta(1)$. Clearly $E_\mathsf{S} \subset \{\nu_h(\R_+) + \nu_h(\R_+ + i\pi) > \frac12 \}$, so this gives the second inequality of~\eqref{eq-three}.

Finally, we have $\{ \tau^h_{-\beta} > \sigma - N\} \subset \{ \nu_h([0,\tau^h_{-\beta} + N] \times \{0,\pi\}) \geq \frac12\}$. Thus the third inequality of~\eqref{eq-three} follows from Lemma~\ref{lem: coupling of disk and GFF} and Corollary~\ref{cor-large-boundary}.
\medskip

\noindent\textbf{Step 2: Showing ``most'' $q_1,q_2$ satisfy~\eqref{eqn: beta far to the left}.} We apply Markov's inequality to show that when we break $\mathsf{S}$ into many small squares, most of them satisfy \eqref{eqn: beta far to the left}.

We rewrite the result of Step 1 to say that for some function $\delta = \delta(\eps, \beta, N)$, for fixed $\eps,N$ we have $\delta \to 0$ as $\beta \to \infty$, and 
\begin{equation}\label{eqn: E' given E large}
\P\left[G_{r, \beta, \eps, N} \: \big| \: E_\mathsf{S}\right] \geq 1- \delta^2 \text{ for all large } r.
\end{equation}
Partition $\mathsf{S}$ into squares $\mathsf{s}_1, \mathsf{s}_2,\dots, \mathsf{s}_N$ of side-length $e^{\gamma (-r + K)/2}$. Call a square $\mathsf{s}$ \textit{bad} if $\P[G_{r, \beta, \eps, N} \mid E_\mathsf{s}] < 1-\delta$, and \textit{good} otherwise. We have
\begin{equation*}
\delta^2\P[E_\mathsf{S}] \geq \P[G_{r, \beta, \eps, N}^c  \mid E_\mathsf{S}]\P[E_\mathsf{S}] = \sum_{\mathsf{s}} \P[G_{r, \beta, \eps, N}^c \mid E_\mathsf{s}]\P[E_\mathsf{s}]\geq \delta \sum_{\mathsf{s} \text{ bad}}  \P[E_{\mathsf{s}}].
\end{equation*}
Dividing through by $\delta P[E'_{r,\beta}]$ and applying \eqref{eqn: E' given E large}, we obtain
\begin{equation}\label{eqn: total mass of bad squares small}
\frac{\delta}{1-\delta^2} \P[E_\mathsf{S} \mid E'_{r,\beta}]\geq \frac{\delta \P[E_\mathsf{S} \mid E'_{r,\beta}]}{\P[E'_{r,\beta} \mid E_\mathsf{S}]} = \frac{\delta \P[E_\mathsf{S}]}{\P[E'_{r,\beta}]} \geq \sum_{\mathsf{s} \text{ bad}} \frac{\P[E_\mathsf{s}]}{\P[E'_{r,\beta}]} \geq \sum_{\mathsf{s} \text{ bad}} \P[E_{\mathsf{s}} \mid E'_{r,\beta}],
\end{equation}
so ``most'' squares $\mathsf{s}$ are good.
\medskip

\noindent\textbf{Step 3: Finding a good square satisfying \eqref{eqn: most phi are improbable}.} Recall the coupling of Lemma~\ref{lem: coupling of disk and GFF}, where we sample a quantum disk $(\cS, \psi, +\infty, -\infty)$ conditioned on $E'_{\beta}$, and for $r > \beta$ set $h(\cdot) = \psi(\cdot + \tau^\psi_{-r})$. Since $\tau^\psi_{-r} \rta -\infty$ as $r\rta\infty$, in this coupling a.s.\ $\lim_{r \to \infty} \nu_h(\R_+) = \nu_\psi(\R)$ and $\lim_{r \to \infty} \nu_h(\R_++i\pi) = \nu_\psi(\R+i\pi)$, so by the bounded convergence theorem we have
\begin{equation*}
\iint_\mathsf{S} d^\beta_{\mathrm{disk}} (y_1,y_2) \ dy_1dy_2 = \lim_{r \to \infty} \P[E_\mathsf{S} \mid E'_{r,\beta}].
\end{equation*}
Let $\delta' = \sqrt{\delta}$, so $\lim_{\beta \to \infty} \delta' = 0$. For sufficiently large $\beta$, for sufficiently large $r$ we have $\frac{1}{1+\delta'} + \frac{\delta}{(1-o_r(1))(1-\delta^2)} < 1$, and so 
\begin{equation}\label{eqn: asdfghj}
\iint_\mathsf{S} d^\beta_{\mathrm{disk}} (y_1,y_2) \ dy_1dy_2 > \left(\frac{1}{1+\delta'} + \frac{\delta}{(1-o_r(1))(1-\delta^2)} \right) \P[E_\mathsf{S} \mid E'_{r,\beta}].
\end{equation}
On the other hand, applying Lemma~\ref{lem: conditioned on start, field looks like disk} to \eqref{eqn: total mass of bad squares small}, we conclude
\begin{equation}\label{eqn: contribution of bad squares}
 \frac{\delta}{(1-o_r(1))(1-\delta^2)} \P[E_\mathsf{S} \mid E'_{r,\beta}] \geq \sum_{\mathsf{s} \text{ bad}}  \iint_\mathsf{s} d^\beta_{\mathrm{disk}}(y_1,y_2) \ dy_1 dy_2 .
\end{equation}
Subtracting~\eqref{eqn: contribution of bad squares} from \eqref{eqn: asdfghj}, we conclude that
\[\sum_{\mathsf{s} \text{ good}} \iint_\mathsf{s} d^\beta_\mathrm{disk} (y_1,y_2) \geq \frac{1}{1+\delta'} \P[E_\mathsf{S} \mid E'_{r,\beta}] \geq \sum_{\mathsf{s} \text{ good}} \frac{1}{1+\delta'} \P[E_\mathsf{s} \mid E'_{r,\beta}]. \]
Thus there exists a good square $\mathsf{s} =[q_1, q_1 + e^{\gamma (-r+K)/2}] \times [q_2, q_2 + e^{\gamma (-r+K)/2}]$ satisfying $(1+\delta')\iint_\mathsf{s} d^\beta_\mathrm{disk} (y_1,y_2) \geq\P[E_\mathsf{s} \mid E'_{r,\beta}]$. Using the uniform continuity of $d^\beta_{\mathrm{disk}}$ on $\mathsf{S}$ and taking $r$ sufficiently large, we see that this square satisfies~\eqref{eqn: most phi are improbable}, and by the definition of good, this square also satisfies~\eqref{eqn: beta far to the left}.
Thus we have found the required $q_1, q_2$. 
\end{proof}

With our matching upper and lower bounds (Lemmas~\ref{lem: choose q1, q2},~\ref{lem: conditioned on start, field looks like disk}) in hand, we are now ready to prove Proposition \ref{prop: head of GFF}.

\begin{proof}[Proof of Proposition \ref{prop: head of GFF}] 
By Lemma~\ref{lem: coupling of disk and GFF}, the conditional law of $h|_R + r$ given $E'_{r,\beta}$ is close in total variation to the field $\phi$ defined in~\eqref{eqn: initial field}. In Step 1, we show that $h|_R + r$ further conditioned on $E_{r,K,q_1,q_2}\cap E'_{r,\beta}$ is close in total variation to $\phi$. In Step 2, we show that when we condition on most realizations of $h|_R$, the boundary lengths $(\nu_h(\R_+), \nu_h(\R_+ + i\pi))$ are close in total variation to independent samples from $[q_j, q_j + \exp(\gamma(-r+K)]$ for $j=1,2$. These two steps are done by carefully combining the bounds of Lemmas~\ref{lem: conditioned on start, field looks like disk} and~\ref{lem: choose q1, q2}.

We remark that by \eqref{eqn: E' given E}, it suffices to prove Proposition \ref{prop: head of GFF} with $h$ conditioned on $E_{r,K,q_1,q_2} \cap E'_{r,\beta}$, then take $\beta \to \infty$. This is our goal for the rest of this proof. To accomplish this, we take limits of three parameters in the following order: first we send $r \to \infty$, then $\beta \to \infty$, and finally $\eps \to 0$. Although $\eps$ has no role in the statement of Proposition~\ref{prop: head of GFF}, recalling that $q_1, q_2$ are chosen from $[\frac12, \frac12 + \eps]$ in Lemma~\ref{lem: choose q1, q2}, we see that $\eps \to 0$ guarantees $q_1, q_2 \to \frac12$.

We first restate a result of Lemma~\ref{lem: coupling of disk and GFF} here:
\begin{equation}\label{eqn: initial field conditioned on E'}
\lim_{r \to \infty} d_{TV}\left(h|_R \text{ conditioned on }E'_{r,\beta}, \phi - r \right) = 0. 
\end{equation}
\medskip

\noindent\textbf{Step 1: further conditioning on $E_{r,K,q_1,q_2}$.}
We want to prove that~\eqref{eqn: initial field conditioned on E'} continues to hold if we instead condition on $E_{r,K,q_1,q_2}\cap E'_{r,\beta}$ for an appropriate choice of $q_1,q_2$.
By Lemma \ref{lem: conditioned on start, field looks like disk}, there exists an error $e_3= e_3(r,\beta)$ so that with probability at least $1-e_3$ over the realization of $\phi$ sampled from \eqref{eqn: initial field} we have  
\begin{equation}\label{eqn: most phi are probable}
d^{\beta,r}_{\GFF}(y_1,y_2 \: \big| \:\phi ) \geq (1-e_3) d_{\op{disk}}^\beta(y_1, y_2)\quad \text{ for all }y_1, y_2 \in [\frac12, 1],
\end{equation}
and the error $e_3$ satisfies $\lim_{r \to \infty} e_3 = 0$ for fixed $\beta$.

Let $\eps>0$ and $q_1, q_2$ as in Lemma \ref{lem: choose q1, q2}. For each fixed distribution $\varphi$ on $R$, define $\P_\varphi$ to be the regular conditional law of $h$ conditioned on $\{h|_R = \varphi - r\}$. Using the uniform continuity of $d^\beta_\mathrm{disk}$ on $[1/2,1]^2$ and integrating \eqref{eqn: most phi are probable} over $y_j \in [q_j, q_j + \exp(\gamma (-r + K)/2)]$ for $j=1,2$, we conclude that with probability $1-e_3$ over the realization of $\phi$,  
\begin{equation}\label{eqn: e4}
\P_\phi\left[E_{r,K,q_1,q_2}\: \Big| \: E'_{r,\beta}\right] \geq (1-e_3 -o_r(1))  e^{\gamma (-r+K)} d_{\mathrm{disk}}^\beta(q_1,q_2).
\end{equation}
Applying \eqref{eqn: initial field conditioned on E'} to \eqref{eqn: e4}, we see that with probability $1-e_3 - o_r(1)$ over the realization of the random field $h|_R$ conditioned on $E'_{r,\beta}$, we have
\begin{equation}\label{eqn: e4 but for h}
\P_{h|_R + r}\left[E_{r,K,q_1,q_2}\: \Big| \: E'_{r,\beta}\right] \geq (1-e_3 -o_r(1))  e^{\gamma (-r+K)} d_{\mathrm{disk}}^\beta(q_1,q_2).
\end{equation}

Write $\cL_{E'}$ for the conditional law of $h|_R$ given $E'_{r,\beta}$, and $\cL_{E\cap E'}$ for the conditional law of $h|_R$ given $E_{r,K,q_1,q_2} \cap E'_{r,\beta}$. If we compare the lower bound \eqref{eqn: e4 but for h} with the upper bound \eqref{eqn: most phi are improbable} and use Bayes' rule under the conditional law given $E_{r,\beta}'$, we get the following lower bound on the Radon-Nikodym derivative with probability $1-e_3-o_r(1)$ over the realization of $h|_R \sim \cL_{E'}$:
\[\frac{d\cL'_{E\cap E'}}{d\cL_{E'}}(h|_R)  = \frac{\P_{h|_R +r} [E_{r,K,q_1,q_2} \mid E'_{r,\beta}]}{\P[E_{r,K,q_1,q_2} \mid E'_{r,\beta}]}  \geq \frac{1-e_3 - o_r(1)}{1+e_2}. \]
This implies that the total variation distance between $h|_R$ conditioned on $E'_{r,\beta}$ and $h|_R$ conditioned on $E_{r,K,q_1,q_2} \cap E'_{r,\beta}$ is $o_\beta(1)+o_r(1) = o_\beta(1)$. Comparing this to \eqref{eqn: initial field conditioned on E'}, we conclude that for all sufficiently large $r$ in terms of $\beta, \eps$, 
\begin{equation}\label{eqn: h|_R close to phi}
d_{TV} \left( h|_R + r \text{ conditioned on } E_{r,K, q_1, q_2} \cap E'_{r,\beta}, \phi \right) \leq e_4
\end{equation}
for some error $e_4(\eps,\beta)$ which goes to zero as $\beta \to \infty$. Thus, we have shown that the two fields in Proposition \ref{prop: head of GFF} are close in total variation.
\medskip

\noindent\textbf{Step 2: near-independence of boundary lengths.}
Finally, we need to show that when we sample $\phi$ via \eqref{eqn: initial field} and then condition on $\{h|_R = \phi -r\}\cap E_{r,K,q_1,q_2} \cap E'_{r,\beta}$, then with high probability over the realization of $\phi$, the side lengths $\nu_h(\R_+)$ and $\nu_h(\R_+ + i\pi)$ are close in total variation to being chosen independently from $\mathrm{Unif}([q_j, q_j + e^{\gamma (-r+K)/2}])$ for $j=1,2$ respectively. By \eqref{eqn: h|_R close to phi}, we can couple the fields ($h|_R$ conditioned on $E_{r,K,q_1,q_2} \cap E'_{r,\beta}$) and $\phi-r$ so that they agree with probability $1-e_4$. In the probability space of this coupling, define the random variable
\eqbn
Y = \P_\phi \left[ E_{r,K,q_1,q_2} \big| E'_{r,\beta}\right] \mathbbm1 \{h|_R = \phi - r\},
\eqen
where, as above, $\P_\varphi$ is defined to be the regular conditional law of $h$ conditioned on $\{h|_R = \varphi - r\}$. Averaging over the realization of $\phi$ and then applying \eqref{eqn: most phi are improbable} tells us that 
\eqb \label{eqn: Y mean}
\E[Y] \leq (1+e_2) e^{\gamma(-r + K)} d^\beta_{\mathrm{disk}} (q_1,q_2).
\eqe 
But \eqref{eqn: e4} together with the fact that $h|_R = \phi - r$ w.p. $1-e_4$ tell us that 
\begin{equation}\label{eqn: Y whp large}
Y   \geq (1-o_r(1)) e^{\gamma(-r + K)} d^\beta_{\mathrm{disk}} (q_1,q_2) \quad \text{ with probability } 1-e_3-e_4. 
\end{equation}
We now prove a high-probability upper bound for $Y$. Indeed, if we set $Y' :=  Y e^{-\gamma(-r + K)} d^\beta_{\mathrm{disk}} (q_1,q_2)^{-1}$, then combining \eqref{eqn: Y mean} and \eqref{eqn: Y whp large} shows that for each $\delta \in (0,1)$, it holds for large enough $\beta > 0$ and $r > r_0(\beta)$ that
\begin{align}
1 + \delta^2
\geq \BB E[Y'] 
&\geq (1+\delta) \BB P\left[ Y' \geq 1+\delta  \right]  + (1-\delta^2) \left( \P[Y' \geq 1-\delta^2] - \P[Y' \geq 1+ \delta] \right)  \notag \\
&\geq (1+\delta) \BB P\left[ Y' \geq 1+\delta  \right]  + (1-\delta^2) \left(1 -  \delta^2 - \BB P\left[  Y' \geq 1+\delta \right] \right)  \quad \text{(by \eqref{eqn: Y whp large})}  .
\end{align}
Re-arranging this gives $\BB P[Y' \geq 1+\delta] \leq 3\delta^2 / (\delta + \delta^2)$, which tends to zero as $\delta \rta 0$.  
Recalling the definitions of $Y$ and $Y'$, we get that for sufficiently large $r$ (depending on $\beta$), it holds with probability $1-o_\beta(1)$ over the realization of $\phi$ that 
\begin{equation}\label{eqn: integrated upper bound}
\P_\phi\left[E_{r,K,q_1,q_2} \: \Big| \: E'_{r,\beta}\right] \leq (1+o_\beta(1))e^{\gamma(-r+K)}d_{\mathrm{disk}}^\beta(q_1,q_2).
\end{equation}

Combining this bound with \eqref{eqn: most phi are probable} and the uniform continuity of $d^\beta_\mathrm{disk}(\cdot, \cdot)$ in $[1/2,1]^2$ gives that for all sufficiently large $r$, with probability $1-o_\beta(1)$ over the realization of $\phi$, 
\eqbn
\frac{d^{\beta,r}_{\GFF} (y_1, y_2 \: \big| \:\phi)}{\P_\phi\left[E_{r,K,q_1,q_2} \Big| E'_{r,\beta} \right]} \geq (1-o_\beta(1)) e^{-\gamma (-r+K)} \quad \text{ for all }(y_1,y_2) \in [q_1,q_1+e^{\gamma(-r+K)/2}]\times [q_2,q_2+e^{\gamma(-r+K)/2}].
\eqen
Observe that RHS is close to the uniform density on $[q_1,q_1+e^{\gamma(-r+K)/2}]\times [q_2,q_2+e^{\gamma(-r+K)/2}]$. 
Thus, with probability $1-o_\beta(1)$ over the realization of $\phi$, if we condition $h$ on $E_{r,K,q_1,q_2} \cap E'_{r,\beta}\cap\{h|_R = \phi - r\}$, the side lengths $\nu_h(\R_+)$ and $\nu_h(\R_+ + i\pi)$ are indeed close, in the total variation sense, to being independently and uniformly drawn from $\mathrm{Unif}([q_j,q_j + e^{\gamma (-r+K)/2}])$ for $j=1,2$.  
 Finally, using \eqref{eqn: h|_R close to phi}, this proves Proposition~\ref{prop: head of GFF} when we condition on $E_{r,K,q_1,q_2} \cap E'_{r,\beta}$ and send first $r \to \infty$ then $\beta \to \infty$ and finally $\eps \to 0$ in that order (sending $\eps \to 0$ guarantees that $\lim_{r\to\infty} q_j(r) = \frac12$ for $j=1,2$). By \eqref{eqn: E' given E} the same holds when we only condition on $E_{r,K, q_1,q_2}$. 
\end{proof}

\subsection{$h$ resembles a quantum disk given $E_{r,K,q_1,q_2}$ and ${(h|_R, \nu_h(\R_+), \nu_h(\R_+ + i\pi))}$} \label{subsection: pinched disk}
In this section, we prove that when we condition on $E_{r,K,q_1,q_2}$ and on a typical realization of $(h|_R, \nu_h(\R_+), \nu_h(\R_++i\pi))$, the field $h$ looks like a $(\frac12,\frac12)$-quantum disk field $\psi$. 

\begin{proposition}\label{prop: conditioned on past, future is disk}
Assume the notation and setup of Proposition \ref{prop: head of GFF}.
Suppose we condition on $E_{r,K,q_1(r),q_2(r)}$. Then the conditional law of $h$ becomes close to the law of the field of a $(\frac12,\frac12)$-length quantum disk as $r \to \infty$, and moreover with high probability (w.r.t.\ the realization of $(h|_R, \nu_h(\R_+), \nu_h(\R_+ + i\pi))$) the same holds if we further condition on $(h|_R, \nu_h(\R_+), \nu_h(\R_+ + i\pi))$. 

We state this more precisely. There exist functions $q_1(r), q_2(r)$ satisfying $\lim_{r\to\infty}q_1(r) = \lim_{r \to \infty}q_2(r) = \frac12$, so that for any fixed $\delta > 0$, $N > 0$, and $R$, for sufficiently large $r$ the following two laws have total variation distance at most $\delta$:
\begin{itemize}
\item Let $\sigma^h \in \R_+$ be chosen so that $\nu_{ h}([\sigma^h, +\infty) \times \{0,\pi\}) = \frac12$ and consider the law of the field $h (\cdot + \sigma^h)|_{\cS_+ - N}$ conditioned on $E_{r,K,q_1(r),q_2(r)}$;

\item 
Let $\psi$ be the field of a $(\frac12, \frac12)$-length quantum disk and let $\sigma^\psi \in \R$ be chosen so that $\nu_\psi([\sigma^\psi, +\infty) \times \{0,\pi\}) = \frac12$, and consider the law of the field $\psi(\cdot + \sigma^\psi)|_{\cS_+ - N}$.
\end{itemize}

Moreover, with conditional probability at least $1-\delta$ given $E_{r,K,q_1(r),q_2(r)}$, the triple $(h|_R, \nu_h(\R_+), \nu_h(\R_++i\pi))$ is such that, if we further condition on $(h|_R, \nu_h(\R_+), \nu_h(\R_++i\pi))$ in the first law, then the above two laws are within $\delta$ in total variation distance.  
\end{proposition}
We briefly summarize this proposition. To compare $h$ conditioned on $E_{r,K,q_1,q_2}$ and the $(\frac12,\frac12)$-quantum disk field $\psi$, we first fix their horizontal translations in a way intrinsic to the quantum surfaces (i.e. the quantum boundary length of $\R_+ \cup (\R_+ + i\pi)$ is $\frac12$), then restrict the fields to $\cS_+ - N$. Proposition \ref{prop: conditioned on past, future is disk} then roughly states that for large $r$, when the field $h$ is conditioned on $E_{r,K,q_1,q_2}$ and on typical realizations of $(h|_R, \nu_h(\R_+), \nu_h(\R_+ + i\pi))$, it is close (in the above sense) in total variation to $\psi$.

As before, for notational convenience we will usually write $q_1, q_2$ as shorthand for $q_1(r), q_2(r)$. 

Here is a quick explanation for why this proposition should hold. Conditioning $h$ only on $E'_{r,\beta}$ (defined in \eqref{eqn: E'}), the fields $h|_R$ and $h|_{\cS_+ + \tau_{-\beta}}$ are almost independent, simply because their respective domains are far apart in Euclidean distance. Since $h|_R$ has a very small effect on the lengths $\nu_h(\R_+)$ and $\nu_h(\R_++ i\pi)$, we expect that even if we condition on $\{\nu_h(\R_+), \nu_h(\R_+ + i\pi) \approx \frac12\}$, the field $h|_{\cS_+ + \tau_{-\beta}}$ is almost independent of $h|_R$. 

Next, by Lemma~\ref{lem: coupling of disk and GFF} the field $h (\cdot + \tau_{-\beta})|_{\cS_+}$ conditioned on $E'_{r,\beta}$ has the same law as $\psi(\cdot + \tau_{-\beta}^\psi)|_{\cS_+}$, where $\psi$ is the field of a quantum disk conditioned on $E'_\beta$ (defined in \eqref{eqn: large quantum disk}). For $\beta$ large, the lengths $\nu_h([0,\tau_{-\beta}] \times \{0,\pi\})$ and $\psi((-\infty, \tau_{-\beta}^\psi] \times \{0,\pi\})$ are small with high probability, so the law of $h (\cdot + \tau_{-\beta})|_{\cS_+}$ conditioned on $\{\nu_h(\R_+), \nu_h(\R_+ + i\pi) \approx \frac12\}$ is close to the law of $\psi|_{\cS_+}$ conditioned on $\{\nu_\psi(\R) = \nu_\psi(\R + i\pi) = \frac12\}$. This concludes our brief explanation of the above proposition. Again, this explanation is only a sketch, but is morally correct, and will be fully fleshed out in this section. Arguments in a similar flavor as in the previous section give upper and lower bounds on various conditional densities (Lemmas \ref{lem: lower bound on future field} and \ref{lem: stronger upper bound}), which combined yield Proposition \ref{prop: conditioned on past, future is disk}.

Once again, we expect that this proposition should hold if we make for all $r$ the choice $q_1 = q_2 = \frac12$, but it seems simpler to avoid proving it for this specific choice.

Now, we prove a couple of lemmas in order to prove Proposition \ref{prop: conditioned on past, future is disk}. 
The proof of the first lemma uses the notation and intermediate steps of Lemma \ref{lem: conditioned on start, field looks like disk}. Recall that $\cL^{\beta,a,b}_\mathrm{disk}$ is the regular conditional law of $(X^\psi,f^\psi)$ conditioned on $E'_\beta\cap\{\nu_\psi(\R) = a, \nu_\psi(\R + i\pi) = b\}$ in the decomposition \eqref{eqn: field decomposition}. Observe that the field $\Phi := \psi( \cdot + \tau^\psi_{-\beta})|_{\cS_+}$ is a function of $(X^\psi, f^\psi)$, so with an abuse of notation, we also write $\cL^{\beta,a,b}_\mathrm{disk}$ to mean the law of $\Phi$ conditioned on $E'_\beta\cap\{(\nu_\psi(\R), \nu_\psi(\R + i\pi) )= (a,b)\}$. Analogously, if $\varphi$ is a field on $R$ with mean zero on $[0,i\pi]$, then we let $\cL^{\beta,r,\varphi,a,b}_{\GFF}$ denote the law of $h(\cdot + \tau^\psi_{-\beta})|_{\cS_+}$ conditioned on $E'_{r,\beta} \cap \{(h|_R, \nu_h(\R_+), \nu_h(\R_+ + i\pi)) = (\varphi -r, a,b)\}$.

In order to prove Proposition~\ref{prop: conditioned on past, future is disk}, we need to show that for ``most'' realizations of the independent triple $(\phi, V_1, V_2)$ (with $\phi$ as in \eqref{eqn: initial field}, and $V_j$ uniform in $[q_j, q_j + \exp(\gamma(-r+K)/2)]$), the laws $ \cL^{\beta, r,\phi, V_1, V_2}_{\GFF}$ and $\cL^{\beta,V_1,V_2}_\mathrm{disk}$ are close in total variation. To that end, using Lemmas~\ref{lem: lower bound on future field} and \ref{lem: stronger upper bound} we will lower bound $\frac{d \cL^{ \beta,r, \phi,V_1, V_2}_{\GFF}}{d\cL^{\beta,V_1,V_2}_\mathrm{disk}}(\Phi)$ for most $(\phi, V_1,V_2), \Phi$. 

\begin{lemma}[Lower bound on $\frac{d \cL^{\beta,r, \varphi, y_1, y_2}_{\GFF}}{d\cL^{\beta,y_1,y_2}_\mathrm{disk}}$ for most $\Phi$]\label{lem: lower bound on future field}
Let $\varphi$ be the restriction of a Neumann GFF on $\cS$ to $R$ normalized to have zero mean on $[0,i\pi]$. Then $\varphi$-a.s. the following is true when we fix $\varphi$. For all $y_1, y_2 \in \left[ \frac12, 1\right]$, with probability $1-e_1$ over the realization of $\Phi \sim \cL^{\beta,y_1,y_2}_\mathrm{disk}$, we have the inequality
\begin{equation}\label{eqn: ASDFGHJ}
\frac{d \cL^{\beta,r, \varphi, y_1, y_2}_{\GFF}}{d\cL^{\beta,y_1,y_2}_\mathrm{disk}}(\Phi) \geq (1-e_1) \frac{d^{\beta}_\mathrm{disk}(y_1,y_2)}{d^{\beta,r}_{\GFF}(y_1,y_2 \: \big| \: \varphi)}.
\end{equation}
Here, the error $e_1 = e_1(\varphi, r,\beta)$ satisfies for each fixed $\varphi, \beta$ the limit $ \lim_{r \to \infty} e_1 = 0$.
\end{lemma}
\begin{proof}
First, recall the setup and steps of Lemma \ref{lem: conditioned on start, field looks like disk}. Let $\cL^\beta_\mathrm{disk}$ denote the law of $\psi(\cdot + \tau^\psi_{-\beta})|_{\cS_+}$ conditioned on $E'_\beta$, and $ \cL^{ \beta,r, \varphi}_{\GFF}$ the law of $h(\cdot + \tau^h_{-\beta})|_{\cS_+}$ conditioned on $E'_{r,\beta}\cap\{h|_R = \varphi -r\}$. Using Bayes' rule, we have
\begin{equation}\label{eqn: ASDFGHJ1}
\frac{d \cL^{ \beta,r,\varphi,y_1,y_2}_{\GFF}}{d \cL^{\beta,y_1,y_2}_\mathrm{disk}}(\Phi) = \frac{d^{\beta, r}_{\GFF} \left(y_1,y_2 \: \big| \: \varphi, \Phi\right)}{d^{\beta}_{\mathrm{disk}} \left(y_1,y_2 \: \big| \: \Phi\right)} \cdot\frac{d  \cL^{\beta,r,\varphi}_{\GFF}}{d \cL^\beta_{\mathrm{disk}}}(\Phi) \cdot \frac{d^{\beta}_\mathrm{disk}(y_1,y_2)}{d^{\beta,r}_{\GFF}(y_1,y_2 \: \big| \: \varphi)},
\end{equation}
where here $d^{\beta, r}_{\GFF} \left(\cdot, \cdot\: \big| \: \varphi, \Phi\right)$ is the density of $(\nu_h(\R_+), \nu_h(\R_+ + i\pi))$ conditioned on $\{h|_R = \varphi - r\}\cap \left\{h(\cdot + \tau_{-\beta}^h)|_{\cS_+} = \Phi \right\}$, and $d^\beta_\mathrm{disk}(\cdot, \cdot \: \big| \: \Phi)$ is the density of $(\nu_\psi(\R), \nu_\psi(\R + i\pi))$ conditioned on $\{\psi(\cdot + \tau_{-\beta}^\psi) = \Phi\}$. Recall that these regular conditional probability densities can be defined by the Markov property of the GFF.

We will lower bound two of the terms on the right of~\eqref{eqn: ASDFGHJ1} to get \eqref{eqn: ASDFGHJ}. Recall that, in the proof of Lemma \ref{lem: conditioned on start, field looks like disk}, we coupled a GFF $h$ conditioned on $E'_{r,\beta}$ with a quantum disk $(\cS, \psi, +\infty, -\infty)$ conditioned on $E'_\beta$, so that the event $A_{r, \delta'}$ defined as in~\eqref{eqn: coupling with head} and~\eqref{eqn: coupling lower bound} occurs with probability $1-\delta'$ for $r$ sufficiently large. Since $\psi(\cdot + \tau^\psi_{-\beta})|_{\cS_+}$ and $h(\cdot + \tau^h_{-\beta})|_{\cS_+}$ only depend on $(X^\psi( \cdot + \tau_{-r/2}^\psi)|_{\R_+}, f^\psi(\cdot + \tau_{-r/2}^\psi)|_{\cS_+})$ and $(X^h(\cdot + \tau_{-r/2}^h)|_{\R_+}, f^h(\cdot + \tau_{-r/2}^h)|_{\cS_+})$ respectively, this means that when $A_{r,\delta'}$ holds we have $\Phi = \psi(\cdot + \tau^\psi_{-\beta})|_{\cS_+} = h(\cdot + \tau^h_{-\beta})|_{\cS_+}$. 
\medskip

\noindent
\textbf{First term:} We will show that for all $y_1, y_2 \in \left[ \frac12, 1\right]$, with probability $1-o_r(1)$ over the realization of $\Phi \sim \cL^{\beta,y_1,y_2}_\mathrm{disk}$, we have the inequality
\begin{equation}\label{eqn: big part of GFF is close to big part of disk}
d^{\beta,r}_{\mathrm{GFF}} (y_1, y_2 \:\big|\: \varphi, \Phi) \geq (1-o_r(1)) d^\beta_\mathrm{disk} (y_1, y_2 \:\big|\: \Phi).
\end{equation}
Here the $o_r(1)$ terms are uniform in $y_1,y_2$.
Pick any $\delta>0$, and recall that $B^1, \dots, B^N$ is a finite collection of balls covering $[1/2,1]^2$ such that for any ball $B^j$ and any pair of points $(y_1,y_2), (y_1',y_2') \in B^j$, the total variation distance between $\cL^{\beta,y_1,y_2}_\mathrm{disk}$ and $\cL^{\beta,y_1',y_2'}_\mathrm{disk}$ is at most $\delta$. For sufficiently small $\delta' > 0$ and sufficiently large $r$ we have \eqref{eqn: balls absolutely continuous} for all balls $B^j$. For each $j$, an application of Markov's inequality to \eqref{eqn: balls absolutely continuous} shows that, with probability $1-\sqrt\delta$ over the realization of $\Phi$ conditioned on $\{(\nu_\psi(\R), \nu_\psi(\R + i\pi)) \in B^j\}$, we have 
\begin{equation}\label{eqn: analog of 5.12}
\P\left[ A_{r, \delta'}\mid (\nu_\psi(\R), \nu_\psi(\R + i\pi)) \in B^j, \psi(\cdot + \tau^\psi_{-\beta})|_{\cS_+} = \Phi\right] \geq 1 - \sqrt\delta. 
\end{equation}
Recall that on $A_{r,\delta'}$, we have $ h(\cdot + \tau^h_{-\beta})|_{\cS_+} = \Phi$. For any $(y_1, y_2) \in B^j$, since the total variation distance between $\Phi \sim \cL^{\beta,y_1,y_2}_\mathrm{disk}$ and $\Phi$ conditioned on $\{(\nu_\psi(\R), \nu_\psi(\R + i\pi)) \in B^j\}$ is at most $\delta$, by integrating \eqref{eqn: coupling lower bound} over $A_{r,\delta'}$ we see that with probability $1 -\delta-\sqrt\delta$ over the realization of $\Phi \sim \cL^{\beta, y_1,y_2}_\mathrm{disk}$, we have 
\[d^{\beta,r}_{\GFF} \left(y_1, y_2 \: \big| \: \varphi, \Phi\right) \geq (1-\delta -\sqrt\delta) d^\beta_\mathrm{disk} \left(y_1,y_2 \: \big| \: \Phi\right) - \delta.\]
(This above step is essentially the same as the rest of the proof of Lemma~\ref{lem: conditioned on start, field looks like disk} after \eqref{eqn: balls absolutely continuous}, with \eqref{eqn: analog of 5.12} in this argument playing the role of \eqref{eqn: balls absolutely continuous}.)
Since the law of the random variable $ d^\beta_\mathrm{disk} \left(y_1,y_2 \: \big| \: \Phi\right) $ (with $\Phi \sim \cL^{\beta, y_1,y_2}_\mathrm{disk}$) does not depend on $\delta$, with high probability we can absorb the additive term $-\delta$ in the RHS to obtain \eqref{eqn: big part of GFF is close to big part of disk}. 
\medskip

\noindent
\textbf{Second term:} We claim that for all $y_1, y_2 \in \left[ \frac12, 1\right]$, we have
\begin{equation} \label{eqn: KALSJDLKAJDLKSJA}
\frac{d  \cL^{\beta,r,\varphi}_{\GFF}}{d \cL^\beta_{\mathrm{disk}}}(\Phi) \geq 1 - o_r(1) \quad \text{ with probability } 1- o_r(1) \text{ over } \Phi \sim \cL^{\beta,y_1,y_2}_\mathrm{disk}. 
\end{equation}
Here the $o_r(1)$ error is uniform for $y_1, y_2 \in [\frac12,1]$. Consider first $\Phi \sim \cL^\beta_\mathrm{disk}$. In the above coupling of $\psi$ conditioned on $E'_\beta$ and $h$ conditioned on $E'_{r,\beta} \cap \{h|_R = \varphi - r\}$, with probability $1-o_r(1)$ the coupling holds, so $\Phi = \psi(\cdot + \tau_{-\beta}^\psi)|_{\cS_+} = h(\cdot + \tau_{-\beta}^h)|_{\cS_+}$ with probability $1-o_r(1)$. If we define for fixed $\delta'$ the event on the coupled probability space
\[\widetilde A_{r,\delta'} = \{\text{coupling holds} \} \cap \left\{\frac{d  \cL^{\beta,r,\varphi}_{\GFF}}{d \cL^\beta_{\mathrm{disk}}} (\Phi) \geq 1 - \delta' \right\},  \]
then for sufficiently large $r$ we have
\begin{equation} \label{eqn: KALSJDLKAJDLKSJA2}
\P[\widetilde A_{r,\delta'}] \geq 1 - \delta'.
\end{equation}
Note that in the above, since we condition $\psi$ on $E'_\beta$, the field $\Phi$ has the law $\cL^\beta_\mathrm{disk}$. The argument of Step 2 of Lemma~\ref{lem: conditioned on start, field looks like disk} (replacing $A_{r,\delta'}$ with $\wt A_{r,\delta'}$) lets us extend this to get \eqref{eqn: KALSJDLKAJDLKSJA}. 

Using the lower bounds of the first and second term in \eqref{eqn: ASDFGHJ1}, we get the desired bound.
\end{proof}

\begin{lemma}[Upper bound on $d^{\beta,r}_{\GFF}(V_1, V_2 \mid \phi)$ for most $(\phi, V_1, V_2)$]\label{lem: stronger upper bound}
Fix $\eps > 0$. Choose $q_1, q_2$ as in Lemma \ref{lem: choose q1, q2}. Let $\phi, V_1, V_2$ be independent random variables with $\phi$ defined in \eqref{eqn: initial field}, and $V_j$ uniformly chosen from $\left[ q_j, q_j + \exp(\gamma (-r+K)/2) \right]$ for $j=1,2$. Then for sufficiently large $r$, with probability $1-o_\beta(1)$ we have
\begin{equation}\label{eqn: upper bound GFF density}
d^{\beta,r}_{\GFF}(V_1, V_2 \:\big|\: \phi) \leq (1+o_\beta(1))d^\beta_{\mathrm{disk}}(V_1,V_2).
\end{equation}
\end{lemma}
\begin{proof}
Pick any $\delta > 0$. Then for any sufficiently large $\beta$ and any $r > r_0(\beta)$, we have, with probability $1-\delta$ over the realization of $\phi$, that the following both hold:
\begin{align*}
\E\left[ d^{\beta, r}_{\GFF}(V_1,V_2 \mid \phi) \: \big| \: \phi \right] &\leq (1+\delta^2)d_{\mathrm{disk}}^\beta(q_1,q_2),\\
d^{\beta, r}_{\GFF}(y_1, y_2 \mid \phi) &\geq (1-\delta^2) d^{\beta}_\mathrm{disk} (q_1,q_2) \text{ for all } y_j \in [q_j, q_j + \exp(\gamma(-r+K)/2)], \: \forall j \in \{1,2\}.
\end{align*}
The first follows by rephrasing \eqref{eqn: integrated upper bound}, and the second from \eqref{eqn: most phi are probable} together with the uniform continuity of $d^\beta_\mathrm{disk}(\cdot, \cdot)$ on $[1/2,1]^2$. 
Fix any realization of $\phi$ for which these inequalities both hold, and consider the random variable $d^{\beta, r}_{\GFF} (V_1,V_2 \mid \phi) - (1-\delta^2)d^\beta_\mathrm{disk} (q_1,q_2)$ (where the randomness is due to $V_j \sim [q_j, q_j + \exp (\gamma (-r+K)/2)]$). By the second inequality this is a.s. nonnegative, and by the first inequality its expectation is at most $2\delta^2 d^\beta_\mathrm{disk}(q_1,q_2)$. By Markov's inequality, for any fixed realization of $\phi$ for which the above inequalities both hold, we have with probability $1-\delta$ over the realization of $V_1, V_2$ that 
\[ d^{\beta, r}_\mathrm{GFF} (V_1,V_2 \mid \phi) - (1-\delta^2)d^\beta_\mathrm{disk}(q_1,q_2) \leq 2\delta d^\beta_\mathrm{disk} (q_1,q_2).\]
Since $\delta$ is arbitrary, this yields~\eqref{eqn: upper bound GFF density}. 
\end{proof}

Finally, we turn to the proof of Proposition \ref{prop: conditioned on past, future is disk}. 
\begin{proof}[Proof of Proposition \ref{prop: conditioned on past, future is disk}]
We first fix $\eps > 0$ and pick $q_1, q_2 \in \left[\frac12, \frac12 + \eps\right]$ via Lemma \ref{lem: choose q1, q2}. Let $(\psi, \cS, -\infty, +\infty)$ be a quantum disk. Let $(\phi, V_1, V_2)$ be a mutually independent triple with $\phi$ as in \eqref{eqn: initial field}, and $V_j \sim \mathrm{Unif}[q_j, q_j + \exp(\gamma (-r+K)/2)]$ for $j=1,2$. First, we prove that for $r$ sufficiently large, with probability $1-o_\beta(1)$ over the realization of $(\phi, V_1,V_2)$ we have
\begin{equation}\label{eqn: condition on side lengths}
\begin{aligned}
d_{TV}\Big( &h ( \cdot + \tau^h_{-\beta})|_{\cS_+} \text{ conditioned on } E_{r,K,q_1,q_2} \cap E'_{r,\beta}\cap \left\{(h|_R, \nu_h(\R_+), \nu_h(\R_+ + i\pi)) = (\phi -r, V_1,V_2) \right\}, \\
 &\psi(\cdot + \tau_{-\beta}^\psi)|_{\cS_+} \text{ conditioned on }E'_\beta \cap \left\{(\nu_\psi(\R), \nu_\psi(\R + i\pi)) = (V_1,V_2) \right\} \Big)  =o_\beta(1).
\end{aligned}
\end{equation}
To see why this holds, we lower bound the Radon-Nikodym derivative $d  \cL^{\beta, r,\phi,V_1,V_2}_{\GFF}/{d \cL^{\beta,V_1,V_2}_\mathrm{disk}} (\Phi)$. For sufficiently large $r$ (how large depends on $\beta$), the error $e_1(\phi,r, \beta)$ of Lemma \ref{lem: lower bound on future field} will be small with high probability over the realization $\phi$. Thus, plugging $(\varphi, y_1, y_2) = (\phi, V_1, V_2)$ into \eqref{eqn: ASDFGHJ} of Lemma~\ref{lem: lower bound on future field}, and using Lemma~\ref{lem: stronger upper bound} to lower bound the RHS of~\eqref{eqn: ASDFGHJ} with high probability, we get that for $r$ sufficiently large, with probability $1-o_\beta(1)$ over the realization of $(\phi, V_1,V_2)$ we have
\[\frac{d  \cL^{\beta, r,\phi,V_1,V_2}_{\GFF}}{d \cL^{\beta,V_1,V_2}_\mathrm{disk}}(\Phi) \geq 1 - o_\beta(1) \quad \text{ with probability } 1-o_\beta(1) \text{ over } \Phi \sim d \cL^{\beta,V_1,V_2}_\mathrm{disk}.\]
This implies \eqref{eqn: condition on side lengths}.
 
Next, we convert \eqref{eqn: condition on side lengths} to a corresponding statement \eqref{eqn: condition on side lengths different field} about the fields $h(\cdot + \sigma^h)|_{\cS_+ - N}$ and $\psi(\cdot + \sigma^\psi)|_{\cS_+ - N}$. For $r$ sufficiently large in terms of $\beta$, from~\eqref{eqn: beta far to the left} we have  
\begin{equation*}
\P[E'_{r,\beta} \cap \{\tau^h_{-\beta} < \sigma - N\} \mid  E_{r,K,q_1,q_2}] = 1 - o_\beta(1). 
\end{equation*}
Write $\P^{r, \varphi,y_1,y_2}_\mathrm{GFF}$ to denote the regular conditional law of $h$ conditioned on $\{(h|_R , \nu_h(\R_+), \nu_h(\R_+ + i\pi)) = (\varphi -r, y_1,y_2)\}$. Using this equation with Markov's inequality, and the fact that conditioned on $E_{r,K,q_1,q_2}$ the triple $(h|_R , \nu_h(\R_+), \nu_h(\R_+ + i\pi))$ is close in total variation to $(\phi -r, V_1,V_2)$ (Proposition~\ref{prop: head of GFF}), we know that for $r$ sufficiently large, with probability $1-o_\beta(1)$ over the realization of $\phi,V_1,V_2$, we have 
\begin{equation}\label{eqn: comparing different field descriptions for h}
\P^{r,\phi,  V_1,V_2}_\mathrm{GFF}[E'_{r,\beta}\cap \{\tau^h_{-\beta} < \sigma^h - N\}] = 1-o_\beta(1).
\end{equation}
We can show a similar statement for the quantum disk by using Proposition~\ref{prop: E' uniformly likely for all y1,y2} in place of \eqref{eqn: E' given E}. Namely, writing $\P^{y_1, y_2}_\mathrm{disk}$ for the law of a quantum disk field $\psi$ conditioned on $\{(\nu_\psi(\R), \nu_\psi(\R + i\pi)) = (y_1, y_2)\}$, using Corollary~\ref{cor: Psi continuous} we see that, for $r$ sufficiently large in terms of $\beta$, a.s. over the realization of $V_1,V_2$ we have
\begin{equation}\label{eqn: comparing different field descriptions for psi}
\P^{V_1,V_2}_\mathrm{disk}[E'_{\beta}\cap \{\tau^\psi_{-\beta} < \sigma^\psi - N\}] = 1-o_\beta(1).
\end{equation}
On the event $E'_{r,\beta} \cap \{ \tau^h_{-\beta} < \sigma^h - N\}$, the field $h(\cdot + \sigma^h)|_{\cS_+ - N}$ is a function of $h(\cdot + \tau^h_{-\beta})|_{\cS_+}$; the corresponding statement holds for $\psi$ also. Using \eqref{eqn: comparing different field descriptions for h} and \eqref{eqn: comparing different field descriptions for psi} with \eqref{eqn: condition on side lengths}, we conclude that for $r$ sufficiently large, with probability $1-o_\beta(1)$ over the realization of $(\phi, V_1,V_2)$ we have
\begin{equation}\label{eqn: condition on side lengths different field}
\begin{aligned}
d_{TV}\Big( &h ( \cdot + \sigma^h)|_{\cS_+ - N} \text{ conditioned on } \left\{(h|_R, \nu_h(\R_+), \nu_h(\R_+ + i\pi)) = (\phi -r, V_1,V_2) \right\}, \\
 &\psi(\cdot + \sigma^\psi)|_{\cS_+ - N} \text{ conditioned on }\left\{(\nu_\psi(\R), \nu_\psi(\R + i\pi)) = (V_1,V_2) \right\} \Big)  =o_\beta(1).
\end{aligned}
\end{equation}

Finally, we prove Proposition~\ref{prop: conditioned on past, future is disk} by sending $r \to \infty, \beta \to \infty, \eps \to 0$ in that order. By Corollary~\ref{cor: Psi continuous}, if we let $\psi$ be the field of a quantum disk, then for $\eps>0$ sufficiently small we have for all $y_1,y_2 \in \left[\frac12, \frac12 + 2\eps\right]$ that
\begin{align*}
d_{TV} \Big( &\psi(\cdot + \sigma^\psi)|_{\cS_+ - N} \text{ conditioned on }\left\{(\nu_\psi(\R), \nu_\psi(\R + i\pi)) = (y_1,y_2) \right\}, \\
&\psi(\cdot + \sigma^\psi)|_{\cS_+ - N} \text{ conditioned on }\{(\nu_\psi(\R), \nu_\psi(\R + i\pi)) = (\frac12,\frac12) \} \Big) < \frac\delta3.
\end{align*}
Next, pick $\beta$ large in terms of $\eps$, and $r$ large in terms of $\beta$ and $\eps$ so that the error of \eqref{eqn: condition on side lengths different field} is less than $\frac\delta3$, and the total variation distance between $(h|_R, \nu_h(\R_+), \nu_h(\R_+ + i\pi))$ and $(\phi - r, V_1, V_2)$ is less than $\frac\delta3$ (Proposition~\ref{prop: head of GFF}). Since $q_1, q_2 \in [\frac12, \frac12+\eps]$, we have proven Proposition~\ref{prop: conditioned on past, future is disk}. 
\end{proof}

\section{Unit boundary length quantum disk as a mating of trees}
\label{section: disk peanosphere}

In this section, we build on the results of the previous sections to prove Theorem \ref{thm: peanosphere disk}. 
Throughout this section $h$ will denote a distribution defined on the whole strip $\mcl S$, rather than just on $\mcl S_+$ as in the previous section. All our arguments in Section~\ref{section: disk peanosphere} do not depend on the choice of equivalence class representative $h$ (i.e. choice of horizontal translation of $h$), but in Section~\ref{subsection: extra conditioning} we specify such a choice for notational convenience.

Our approach is as follows. We sample a counterclockwise space-filling $\SLE_{\kappa'}$ curve $\eta'$ from $-\infty$ to $-\infty$ on an independent $\gamma$-quantum wedge $(\cS, h, +\infty, -\infty)$. Theorem \ref{thm: peanosphere gamma wedge} describes the boundary length process for $\eta'$ on the $\gamma$-quantum wedge, and using this description we restrict our attention to a curve-decorated surface $\cD^* = (\eta'([0,T]), h, \eta')$, where $T$ is a random time such that $\eta'$ explores about 1 unit of quantum boundary length in the time interval $[0,T]$. Doing some careful conditioning on $\cD^*$, we show firstly that $\cD^*$ converges in some sense to a unit boundary length quantum disk decorated by an independent counterclockwise $\SLE_{\kappa'}$ by Proposition \ref{prop: conditioned on past, future is disk}, and secondly that the boundary length process of $\cD^*$ converges to the appropriate excursion in $\R_+^2$ by Proposition~\ref{prop: approx BM}. This yields Theorem \ref{thm: peanosphere disk}.

In the first two sections we focus on the case $\gamma \in (0,\sqrt2]$, because of the simpler topology. In Section~\ref{subsection: decomposing wedge}, we decompose a curve-decorated $\gamma$-quantum wedge into three curve-decorated quantum surfaces $\cD^*, \cW_1^*, \cW_2^*$, and describe their boundary length processes. We also define an event $F_{r,C}$ which roughly corresponds to the boundary length process of $\cD^*$ being close to a Brownian excursion of displacement 1. Roughly speaking, the event $E_{r,K,q_1,q_2}$ of Proposition~\ref{prop: head of GFF} corresponds to finding a bottleneck in the field description of the quantum wedge, and $F_{r,C}$ a bottleneck in the boundary-length description of the curve-decorated quantum wedge. 

In Section~\ref{subsection: equivalence E F}, we show that $\P[F_{r,C} \mid E_{r,K,q_1,q_2}] > 0$ uniformly in $r$, and $\P[E_{r,K,q_1,q_2} \mid F_{r,C}] \approx 1$ for large $C,K$ (with $E_{r,K,q_1,q_2}$ defined in Proposition~\ref{prop: head of GFF}), so if we want to condition on $F_{r,C}$, we can instead condition on $F_{r,C} \cap E_{r,K,q_1,q_2}$. 

In Section~\ref{subsection: gamma large}, we explain the modifications that we need to make to obtain the results of Sections~\ref{subsection: decomposing wedge} and \ref{subsection: equivalence E F} in the regime $\gamma \in (\sqrt2,2)$. Essentially the same arguments apply, but one needs to be careful about the topology of the surfaces.
Finally, in Section~\ref{subsection: extra conditioning}, we complete the proof of Theorem~\ref{thm: peanosphere disk} by first conditioning on $E_{r,K,q_1,q_2}$, then applying Propositions~\ref{prop: head of GFF} and \ref{prop: conditioned on past, future is disk} to show that even after further conditioning on $F_{r,C}$, the quantum surface $\cD^*$ still resembles a quantum disk.

\subsection{Decomposing a $\gamma$-quantum wedge for $\gamma \in (0,\sqrt2]$}
\label{subsection: decomposing wedge}
In this section we sample a $\gamma$-quantum wedge decorated by an independent counterclockwise space-filling $\SLE_{\kappa'}$, and use the boundary-length process of the curve to decompose the wedge into three curve-decorated quantum surfaces $\cD^*, \cW_1^*, \cW_2^*$. We also show that the boundary length process of the curve in $\cD^*$ is close to a Brownian excursion in the cone $\R_+\times \R_+$. We will state our results for \emph{all} $\gamma \in (0,2)$, bur only prove them for $\gamma \in (0,\sqrt2]$ (the proofs for $\gamma \in (\sqrt2, 2)$ are deferred to Section~\ref{subsection: gamma large}). The regime $\gamma \in (0,\sqrt2]$ is topologically simpler because the region explored by a space-filling $\SLE_{\kappa'}$ in an interval of time is almost surely simply connected.

Sample a $\gamma$-quantum wedge $(\cS,h, +\infty, -\infty)$ so that neighborhoods of $+\infty$ (resp. $-\infty$) are finite (resp. infinite), i.e., as $x \to +\infty$ (resp. $-\infty$) the vertical field averages of $h$ on $[x,x+i\pi]$ tend to $-\infty$ (resp. $\infty$). Throughout this section, we define for $t \in \R$ the stopping time
\begin{equation}
\tau_t = \inf_{x \in \R} \{ \text{average of }h \text{ on } [x, x+i\pi] \text{ equals } t\}.
\end{equation}
We emphasize that this definition of $\tau_t$ is different from that of $\tau_t$ in Section~\ref{section: pinching off a disk}. Indeed, for $r \in \R$ the field studied in Section~\ref{section: pinching off a disk} was given by $h(\cdot + \tau_{-r})|_{\cS_+}$, so the stopping times in Section~\ref{section: pinching off a disk} are all to the right of $\tau_{-r}$. In this section, however, when $a > -r$ the stopping time $\tau_{a}$ is to the \emph{left} of $\tau_{-r}$.  

Let $q_1 = q_1(r)$ and $q_2 = q_2(r)$ be the functions in Proposition \ref{prop: head of GFF}, so $\lim_{r \to \infty} q_1(r) = \lim_{r \to \infty} q_2 (r) = \frac12$. For notational simplicity we will usually write $q_1, q_2$, leaving their dependence on $r$ implicit. Let $x_1 \in \R $ and $x_2 \in \R + i\pi$ satisfy
\begin{equation*}
\nu_h([x_1,+\infty)) = q_1, \quad \nu_h([x_2, +\infty)) = q_2.
\end{equation*}

Independently sample a counterclockwise space-filling $\SLE_{\kappa'}$ $\eta'$ from $-\infty$ to $-\infty$, parametrized by quantum area as in Theorem \ref{thm: peanosphere gamma wedge}, and with time recentered so that $\eta'$ hits $x_1$ at time $0$. As counterclockwise space-filling $\SLE_{\kappa'}$ fills the boundary in counterclockwise order, $\eta'$ hits $x_2$ after $x_1$, say at time $T>0$. Define
\begin{equation*}
U = \eta'([0,T]), \quad U_1 = \eta'((-\infty, 0]), \quad U_2 = \eta'([T, +\infty)),
\end{equation*}
and name the intersection point $p = U \cap U_1 \cap U_2$. Note that in the regime $\gamma \in (0,\sqrt2]$, almost surely these three domains are simply connected. Define the curve-decorated quantum surfaces
\begin{equation}
\cD^* := (U,h,\eta',x_1,p,x_2), \quad \cW_1^* := (U_1,h,\eta',x_1,-\infty), \quad \cW_2^* := (U_2,h,\eta', x_2, -\infty).
\end{equation}
Note that $\cW_1^*$ and $\cW_2^*$ each comes with two marked points: one marked point with neighborhoods of finite quantum area, and one with neighborhoods of infinite quantum area. The curve-decorated quantum surface $\cD^*$ comes with three marked points.
See Figure~\ref{fig: surface decomposition} for an illustration of this setup. 

\begin{figure}[ht!]
\begin{center}
\includegraphics[scale=0.75]{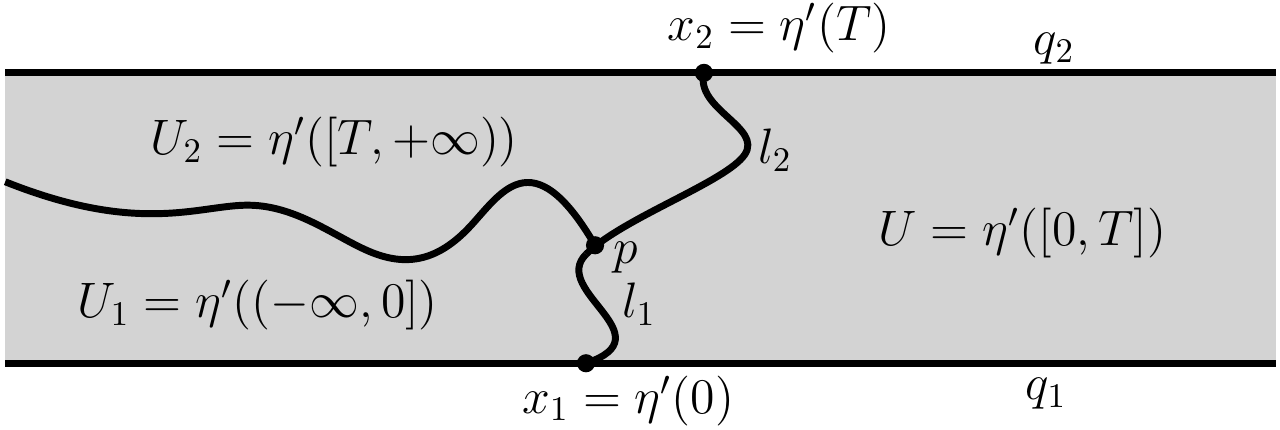}
\end{center}
\caption{\label{fig: surface decomposition} Let ${\gamma \in (0,\sqrt2]}$. By drawing an independent counterclockwise $\SLE_{\kappa'}$ on top of a $\gamma$-wedge $ (\cS, h, +\infty, -\infty)$, we can decompose the curve-decorated $\gamma$-wedge into three curve-decorated surfaces $\cD^*, \cW_1^*, \cW_2^*$ parametrized by the regions $U, U_1, U_2$ respectively.}
\end{figure}

\cite[Proposition 1.7]{shef-zipper} says that if, on a $\gamma$-quantum wedge, one reroots the marked point with finite neighborhoods by moving it some fixed quantum length away on the boundary, the resulting doubly-marked quantum surface is also a $\gamma$-quantum wedge. Thus $(\cS, h, x_j, -\infty)$ is a $\gamma$-quantum wedge for $j=1,2$.
Let $(L_t, R_t)_{t \in \R}$ be the boundary length process of $\eta'$ on $(\cS, h, x_1, -\infty)$, with additive constants fixed so that $L_0 = R_0 = 0$. Using Theorem~\ref{thm: peanosphere gamma wedge}, we can derive descriptions of the boundary length processes of $\eta'$ in each of $\cD^*, \cW_1^*, \cW_2^*$. In each of these descriptions, $(L,R)_I$ is a two-dimensional Brownian motion defined on some time interval $I$ with covariances given by \eqref{eqn: covariance}, and possibly with some conditioning.

\begin{itemize}
\item In $\cD^*$, the boundary length process $(L_t, R_t)_{0 \leq t \leq T}$ of $\eta'$ is Brownian motion started at $(L_0, R_0) = (0,0)$ and stopped the first time $T > 0$ that $R_T = -q_1 - q_2$.

\item In $\cW_1^*$, the boundary length process $(L_t, R_t)_{t \leq 0}$ is the restriction of the boundary length process of Theorem~\ref{thm: peanosphere gamma wedge} to $(-\infty,0]$. 

\item In $\cW_2^*$, the boundary length process $(L_t, R_t)_{t \geq T}$ starts at some random point $(L_T, -q_1-q_2)$, and evolves as a Brownian motion.
\end{itemize}

The only relevant information in these boundary length processes is the change in boundary lengths over time; in particular, we only care about boundary length processes modulo additive constant and translation of time interval. See Section~\ref{sec-MOT} for details. 

By Theorem~\ref{thm: wedge laws} and Remark~\ref{rem-recover}, we see that $\cW_1^*, \cW_2^*$ and $\cD^*$ are a.s. determined by their boundary length processes.
These three boundary length processes (modulo additive constant and translation of time interval) are mutually independent by the strong Markov property of Brownian motion, so we conclude that the curve-decorated surfaces $\cD^*, \cW_1^*, \cW_2^*$ are mutually independent.

By Theorem~\ref{thm: wedge laws}, each of the quantum surfaces $(U_j , h , x_j ,-\infty)$ is a quantum wedge (with some value of $\alpha$).
By the scale-invariance of quantum wedges, for any $c>0$, the curve-decorated quantum surfaces $\cW^*_j + c := (U_j , h + c , \eta', x_j ,-\infty)$ (with $\eta'$ re-parametrized by $\mu_{h+c}$-mass) and $\mcl W_j^*$ agree in law. This scaling property is also easy to see  via the boundary length processes --- adding $c$ to the field increases the quantum area measure by a factor of $e^{\gamma c}$ and the boundary length measure by $e^{\gamma c / 2}$, inducing a Brownian rescaling of the space and time parameterizations of $(L_t, R_t)$.

Let $l_1 = \nu_h(U \cap U_1)$ and $l_2 = \nu_h(U \cap U_2)$. Writing $(L_t, R_t)_{0 \leq t \leq T}$ for the boundary length processes of $\eta'$ in $\cD^*$, we have (with $L_0 = 0$, but this expression for $l_1$ is most natural)
\begin{equation}\label{eqn: interpretation of l1 l2}
l_1 = L_0 - \inf_{t \in [0,T]} L_t, \quad l_2 = L_T - \inf_{t \in [0,T]} L_t.
\end{equation}
Thus, $l_1, l_2$ are measurable with respect to the $\sigma$-algebra $\sigma(\cD^*)$.
Define for $C>0$ the event 
\begin{align} \label{eqn: F def}
F_{r,C} &= \left\{ l_1 < e^{\gamma (-r-C)/2}, \quad l_2 - l_1 \in [1,2] \cdot e^{\gamma(-r-C)/2}\right\}, \notag \\
\cL &= \text{ conditional law of the pair }(h,\eta') \text{ given } F_{r,C}.
\end{align}

Since $F_{r,C}$ is measurable with respect to $\sigma(\cD^*)$, and $\cD^*$ is independent of $(\cW_1^*, \cW_2^*)$, we conclude that the $\cL$-law of $(\cW^*_1, \cW^*_2)$ is the same as in the unconditioned setting, and that under $\cL$ the decorated quantum surfaces $\cD^*, \cW^*_1, \cW^*_2$ are still mutually independent. With slight abuse of notation, we will also say that $(\cD^*, \cW^*_1, \cW^*_2)$ are drawn from $\cL$. 

The strategy of Section~\ref{section: disk peanosphere} is to show that, roughly speaking, if we sample $(h, \eta') \sim \cL$ and send $r \to \infty$, then the curve-decorated quantum surface $\cD^*$ converges in a suitable sense to a $(\frac12,\frac12)$-quantum disk decorated by an independent counterclockwise $\SLE_{\kappa'}$. Since we understand well the boundary length process of $\cD^*$, this will allow us to derive the boundary length process for the $(\frac12,\frac12)$-quantum disk decorated by independent counterclockwise $\SLE_{\kappa'}$. Finally, since the $(\frac12,\frac12)$-quantum disk is just a unit boundary length quantum disk with a pair of antipodal points chosen from boundary measure, we deduce the law of the boundary length process for the unit boundary length quantum disk decorated by an independent counterclockwise $\SLE_{\kappa'}$ from a uniformly chosen boundary point. The reason for working with a $(\frac12,\frac12)$-quantum disk is that a quantum disk with two marked points is easier to relate to a $\gamma$-quantum wedge, since we can compare the two boundary arcs of a doubly-marked quantum disk to the left and right boundary rays of a quantum wedge.

We first show that, in a suitable sense, the boundary length process in $\cD^*$ when we condition on $F_{r,C}$ converges as $r \to \infty$ to the cone excursion of Theorem \ref{thm: peanosphere disk}.
\begin{lemma} \label{lem: boundary length process in D}
For $\gamma \in (0,2)$, consider the law of the boundary length process $(L_t, R_t)_{[0,T]}$ in $\cD^*$, conditioned on $F_{r, C}$. Since the boundary length process is only defined modulo additive constant, we may change our parametrization so that $(L_0, R_0) = (0, q_1+q_2)$, and the process stops at the first time $T$ that $R_T = 0$. As $r \to \infty$, the law of $(L_t, R_t)_{[0,T]}$ converges to that of a sheared normalized boundary-to-boundary Brownian excursion with covariance \eqref{eqn: covariance} in the cone $\R_+ \times \R_+$ from $(L_0,R_0) = (0,1)$ to the origin $(0,0)$. This excursion process is defined in Definition \ref{def: cone BM}, and the convergence is with respect to the Prohorov metric corresponding to \eqref{eqn: metric on space of curves}.
\end{lemma}

\begin{proof}[Proof for ${\gamma \in (0,\sqrt{2}]}$]
Write $\delta = e^{\gamma ( -r-C)/2}$. Recalling the definition of $F_{r,C}$ and using \eqref{eqn: interpretation of l1 l2}, we can exactly describe the law of the process $(L_t, R_t)$. It is given by Brownian motion with covariances \eqref{eqn: covariance} started at the point $(0, q_1 + q_2)$ and conditioned to exit the cone $(\R_+ - \delta) \times \R_+$ in the boundary interval $[\delta, 2\delta] \times \{0\}$. Observe that as we send $r \to \infty$, we have $\delta \to 0$ and $q_1 + q_2 \to 1$. Proposition~\ref{prop: approx BM} tells us that in this limit the Brownian motion converges to the desired cone excursion.
\end{proof}

\subsection{Equivalence of $F_{r,C}$ and $F_{r,C} \cap E_{r,K}$ for $\gamma \in (0,\sqrt2]$} 
\label{subsection: equivalence E F}

In this section, we again focus on the case $\gamma \in (0,\sqrt2]$. We will state our results for \emph{all} $\gamma \in (0,2)$, bur only prove them for $\gamma \in (0,\sqrt2]$. See Section~\ref{subsection: gamma large} for $\gamma \in (\sqrt2,2)$.

In Proposition \ref{prop: head of GFF}, we defined the event $E_{r,K,q_1,q_2}$. In this section, we reuse the notation (suppressing the subscripts $q_1(r), q_2(r)$) for a $\gamma$-quantum wedge $(\cS, h, +\infty, -\infty)$:
\begin{equation} \label{eqn: ErK def}
E_{r,K} = \left\{ \nu_ h (\R_+ + \tau_{-r}) \in \left[ q_1, q_1 + e^{\gamma (-r+K)/2}\right], \nu_ h (\R_+ + \tau_{-r} + i\pi) \in \left[q_2,q_2 + e^{\gamma (-r+K)/2}\right] \right\}. 
\end{equation}
The reason why this is not such a strange choice of notation is that the field $h(\cdot + \tau_{-r})|_{\cS_+}$ has precisely the same law as the field described in Proposition \ref{prop: head of GFF} (see Lemma~\ref{lem-wedge-tip}). 

In this subsection, we show that for any $C, K$ we have $\P[F_{r,C} \mid E_{r,K}] > 0$ uniformly in $r$ (Proposition~\ref{prop:F given E}), and furthermore, as we take $K\to \infty$ then $C\to \infty$, we have $\P[E_{r,K} \mid F_{r,C}] \to 1$ uniformly in $r$ (Proposition~\ref{prop: E given F}). As such, if we want to understand the conditional law of $h$ given $F_{r,C}$, we can instead condition on $E_{r,K} \cap F_{r,C}$.

\begin{proposition} \label{prop:F given E}
Let $\gamma \in (0,2)$, and consider the setup of Section~\ref{subsection: decomposing wedge}, with $E_{r,K}$ defined as in~\eqref{eqn: ErK def}. For each fixed choice of $C,K> 1$, there exists some $p = p(C,K) >0$ such that for all sufficiently large $r$, we have 
\begin{equation}
P[F_{r,C} \mid E_{r,K}] \geq p .
\end{equation}
\end{proposition}
\begin{proof}[Proof for ${\gamma \in (0,\sqrt{2}]}$]
Pick any rectangle $R = [0,S] \times [0,\pi] \subset \cS_+$, and let $\phi$ be the distribution on $R$ defined in~\eqref{eqn: initial field}. Let $d_1, d_2$ be independent samples from $\mathrm{Unif}([0,e^{\gamma K/2}])$, and independently sample a counterclockwise space-filling SLE $\eta'$ in $\cS$ from $-\infty$ to $-\infty$ with arbitrary time-parametrization. Then the following event $A$ occurs with positive probability:
\medskip

\hfill\begin{minipage}{\dimexpr\textwidth-2cm}
There exist $y_1 \in [0,S]$ and $y_2 \in [i\pi, i\pi + S]$ such that $\nu_\phi([0,y_1]) = d_1$ and $\nu_\phi([i\pi, y_2]) = d_2$. Let $V_1$ be the region filled by $\eta'$ before hitting $y_1$, $V_2$ the region filled by $\eta'$ after hitting $y_2$, and $V$ the region filled by $\eta'$ between hitting $y_1$ and $y_2$. The interfaces $V \cap V_1$ and $V \cap V_2$ lie in $R$.
\end{minipage}
\medskip 

\noindent On $A$, analogously to $F_{r,C}$, define the event
\[F =  \left\{ \nu_\phi(V \cap V_1) < e^{-\gamma C/2}, \quad \nu_\phi(V \cap V_2) - \nu_\phi(V \cap V_1) \in [1,2] \cdot e^{-\gamma C/2}  \right\}.\]
Let $\wt p = \P[F \cap A]$. Clearly $\wt p > 0$. We claim that choosing $p= \frac12 \widetilde{p}$ works. 

We recenter the field $h$ so that $\tau_{-r} = 0$. For $r$ sufficiently large, we have by Proposition \ref{prop: head of GFF} that the triple $(h|_R, \nu_h(\R_+), \nu_h(\R_+ + i\pi))$ conditioned on $E_{r,K}$ is within $\frac12 \widetilde{p}$ in total variation of the triple $( \phi|_R - r, q_1 + e^{-\gamma r/2} d_1, q_2 + e^{-\gamma r/2} d_2)$, so we may couple them, and decorate them by the same space-filling curve $\eta'$. On the event $A \cap \{\text{coupling holds}\}$, we have $x_j = y_j$ for $j = 1,2$, and $(U,U_1, U_2) = (V,V_1,V_2)$. Thus, in the coupled probability space we have $F_{r, C} \cap A \cap \{\text{coupling holds}\} = F\cap A \cap \{\text{coupling holds}\}$, so
\[\P[F_{r, C} \mid E_{r,K} ] \geq \P[F \cap A \cap \{ \text{coupling holds} \}] \geq \frac12\widetilde p. \]
\end{proof}

\begin{proposition}\label{prop: E given F}
Let $\gamma \in (0,2)$, and consider the setup of Section~\ref{subsection: decomposing wedge}, with $E_{r,K}$ defined as in~\eqref{eqn: ErK def}. Then for each $\delta > 0$, $C > C_0(\delta)$ and $K > K_0(\delta, C)$, for all $r >0$ we have
\begin{equation}\label{eq-E-given-F}
\P[E_{r,K} \mid F_{r,C}] \geq 1-\delta,
\end{equation}
and moreover, conditioned on $F_{r,C}$, the conditional  probability that $U$ lies to the right of $\tau_{-r}$ goes to 1 as $C \to \infty$ (at a rate uniform for all large $r$).
\end{proposition} 
We sketch why~\eqref{eq-E-given-F} is true. Condition on $F_{r,C}$. Firstly, even though $\cD^*$ is conditioned on a very rare event, the surfaces $\cW_1$ and $\cW_2$ are still (wedges) independent of $\cD^*$, so they should behave in a fairly regular way; as a result, since the boundary lengths $l_1,l_2$ along $\cW_1,\cW_2$ are small ($\frac{2}{\gamma} \log l_j \approx -r - C \ll -r$), we expect the field averages inside $\cW_1, \cW_2$ to also be small. Thus we expect $[\tau_{-r}, \tau_{-r} + i\pi]$ to lie to the left of $U$. 
Secondly, since the field average near $\tau_{-r}$ is close to $-r$, and the surface $\cW_1$ behaves in a fairly regular way, we expect that with high probability the remaining boundary length of $\cW_1$ from $\tau_{-r}$ to its marked point $x_1$ should be within a constant factor of $e^{-\gamma r/2}$. Similarly we expect $\nu_h([\tau_{-r} + i\pi, x_2])$ to be within a constant factor of $e^{-\gamma r/2}$. 
By the definitions of $x_1,x_2$, we conclude that with high probability $\nu_h(\R_+ + \tau_{-r}) \in [q_1, q_1 + \exp(\gamma(-r+K)/2)]$ and $\nu_h(\R_+ + \tau_{-r} + i\pi) \in [q_2, q_2 + \exp(\gamma(-r+K)/2)]$, completing our proof sketch of Proposition \ref{prop: E given F}. We devote the rest of this section to actually proving Proposition \ref{prop: E given F} in the case $\gamma \in (0,\sqrt2]$. 

We will need the following auxiliary random curve-decorated surface. Sample a $\gamma$-quantum wedge $(\cS, \widetilde h,+\infty, -\infty)$ together with an independent counterclockwise space-filling SLE $\widetilde \eta'$. Similar to before, let $\widetilde x_1 \in \R$ and $\widetilde x_2 \in \R + i\pi$ be the unique points satisfying $\nu_{\widetilde h} (\R_+ + x_1) = \nu_{\widetilde h}( \R_+ + x_2) = 1$ (the exact values of these lengths are unimportant). Define the left-to-right stopping times $\widetilde \tau_s$ for $s \in \R$, the regions $\widetilde U, \widetilde U_1, \widetilde U_2$, and the curve-decorated quantum surfaces $\widetilde \cD^*, \widetilde \cW_1^*, \widetilde \cW_2^*$ as above with $(\wt h , \wt\eta')$ in place of $(h,\eta')$. Once again, these three quantum surfaces are mutually independent. Define as before $\widetilde l_j = \nu_{\widetilde h} (\widetilde U \cap \widetilde  U_j)$ for $j= 1,2$, and
\begin{align} \label{eqn: tilde F def}
\widetilde F &= \left\{ \frac{\widetilde l_1}{\widetilde l_2} \in \left[\frac12, \frac34\right], \quad \widetilde l_2 \in \left[\frac12,1\right] \right\}, \notag \\
\widetilde \cL &= \text{ conditional law of } (\widetilde h,\widetilde \eta') \text{ given } \widetilde F . 
\end{align}
Note that $\widetilde F \in \sigma(\widetilde \cD^*)$, so as before, under the law $\widetilde \cL$ the quantum surfaces $\widetilde \cD^*, \widetilde \cW^*_1, \widetilde \cW^*_2$ are still mutually independent, and the marginal laws of $\widetilde \cW^*_1, \widetilde \cW^*_2$ are the same as their unconditional marginal laws, i.e. the respective wedges arising in Theorem~\ref{thm: wedge laws}. 

We will need the following technical lemma in our proof of Proposition~\ref{prop: E given F}.
\begin{lemma}\label{lem: uniform bounds on wedge fields}
Fix $a,b > 0$ and let $\Phi_{a,b}^1$ be the set of smooth functions supported in the rectangle $[0,a] \times [0,\pi]$ with $ \phi \geq 0, \int  \phi(x) \ dx = 1$, and $\| \phi'\|_\infty \leq b$. For $s,x \geq 0$, let
\begin{align*}
M_s(\widetilde h) = \sup_{ \phi \in \Phi_{a,b}^1} (\widetilde h,  \phi ( \cdot + \widetilde\tau_{s})), \qquad
m_x(\widetilde h) = \inf_{ \phi \in \Phi_{a,b}^1} (\widetilde h,  \phi ( \cdot + \widetilde\tau_{x})).
\end{align*}
Then for an unconditioned $\gamma$-quantum wedge $(\cS, \widetilde h,+\infty, -\infty)$, we have
\begin{align*}
\P[M_s(\widetilde h) - s \leq k] \to 1 \text{ as } &k \to +\infty \text{ uniformly in }s > 0,\\
\P[m_x(\widetilde h)  - x \geq -k] \to 1 \text{ as } &k \to +\infty\text{ uniformly in }x > 0.
\end{align*}
Moreover, the same holds for $(\widetilde h, \widetilde \eta')$ sampled from from $\widetilde \cL$.
\end{lemma}
\begin{proof}
The lemma holds for the unconditioned $\gamma$-quantum wedge because $M_s(\widetilde h) - s$ and $m_x(\widetilde h) - x$ have distributions independent of $s,x$, and are almost surely finite. For details see the discussion in the paragraph just after \cite[Proposition 9.19]{wedges}.
Since $\P[\widetilde F] > 0$, the same is true for $(\widetilde h, \widetilde \eta')$ sampled from $\widetilde\cL$. 
\end{proof}

The proof of Proposition \ref{prop: E given F} is long and technical. As such, we present a proof sketch in Figure~\ref{fig:coupling} to convey the main ideas (on a first read one might elect to read just the sketch). 

\begin{proof}[Proof of Proposition \ref{prop: E given F} for ${\gamma \in (0,\sqrt2]}$] We focus on proving~\eqref{eq-E-given-F}. The last assertion of the proposition is Step 3 below; we make no further mention of it.
\medskip

\noindent
\textbf{Step 1: Coupling $(h, \eta')$ with $(\widetilde h, \widetilde \eta')$.}

First, we couple $(\cD^*, \cW^*_1, \cW^*_2)$ and $(\widetilde \cD^*, \widetilde  \cW^*_1, \widetilde \cW^*_2)$ so that their marginal laws are $\cL$ and $\widetilde \cL$ respectively, as follows.
Independently sample $\cD^*$ given $F_{r,C}$ and $\widetilde \cD^*$ given $\widetilde F$,
and define the random variables
\begin{align*}
l &= \frac{l_2\widetilde l_1 - l_1 \widetilde l_2}{\widetilde l_2- \widetilde l_1}, \qquad c = \frac2\gamma \log \frac{l_2 + l}{\widetilde l_2}.
\end{align*}
In other words, $\widetilde l$ and $c$ are the solutions to the system of equations
\begin{align*}
l_j + l = e^{\gamma c / 2} \widetilde l_j\quad  \text{ for }j = 1,2.
\end{align*}

Since $l_1,l_2,\widetilde l_1, \widetilde l_2 \in \sigma(\cD^*, \widetilde \cD^*)$, it is clear that $l$ and $c$ are measurable with respect to $\sigma(\cD^*, \widetilde\cD^*)$. Since we are sampling from the conditional laws $\cL$ and $\widetilde \cL$ given $F_{r,C}$ and $\wt F$, respectively, the definitions~\eqref{eqn: F def} and~\eqref{eqn: tilde F def} of these events imply that a.s. $\wt l \geq 0$ and for some absolute constant $N$ we have 
\[ |(-r-C) - c | < N \text{ almost surely}.\]

Next, given $\cD^*$ and $\widetilde \cD^*$, we produce a coupling of the four curve-decorated quantum surfaces $\cW^*_j$ and $\widetilde \cW^*_j$ for $j=1,2$. Sample $\widetilde\cW^*_1$ and $\widetilde\cW^*_2$ independently from their respective laws. For $j=1,2$, define the quantum surfaces $\cW^*_j = \widetilde \cW^*_j + c$ for $j = 1,2$. By the scale invariance property of the quantum wedges $\cW^*_j$ and the independence of $c$ from $(\widetilde \cW^*_1, \widetilde \cW^*_2)$, we see that given any realization of $\cD^*, \widetilde\cD^*$, the marginal laws of $(\cW^*_1,\cW^*_2)$ and $(\wt \cW^*_1, \wt \cW^*_2)$ are exactly what we want.

This gives us a coupling of $(\cD^*, \cW^*_1, \cW^*_2)$ and $(\widetilde \cD^*, \widetilde  \cW^*_1, \widetilde \cW^*_2)$ with marginal laws $\cL, \widetilde \cL$. Therefore, we can couple $(h, \eta')$ and $(\widetilde h, \widetilde \eta')$ so that $\cW^*_j = \widetilde  \cW^*_j + c$ for $j = 1,2$.
\medskip
\begin{figure}[ht!]
\begin{center}
\includegraphics[scale=0.75]{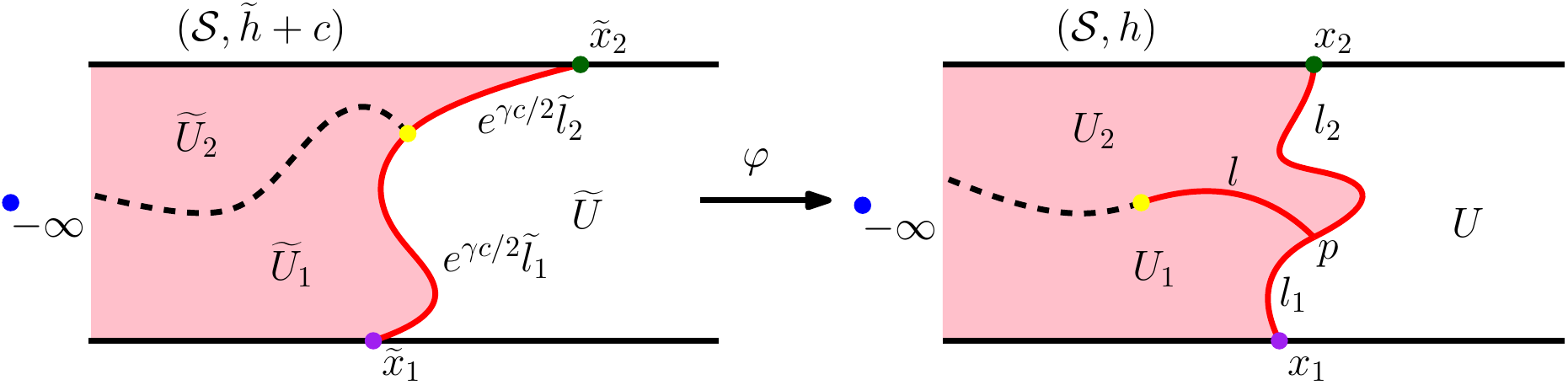}
\end{center}
\caption{\label{fig:coupling} Proof sketch of Proposition \ref{prop: E given F}. In Step 1 we couple $(h, \eta')\sim \cL$ with $(\widetilde h, \widetilde  \eta') \sim \widetilde \cL$ in such a way that for some random $c \approx -r - C$, we have $\widetilde \cW_j^* + c = \cW_j^*$ as curve-decorated quantum surfaces for $j = 1,2$. In Step 2 we make a slit in $U_1 \cup U_2$ producing a slitted pink domain $V$ (right) so that the quantum surfaces parametrized by the pink regions are equivalent. The map $\varphi: \widetilde U_1 \cup \widetilde U_2 \to V$ sends each colored point to the corresponding point of the same color. In Step 3, since in $(\cS, h)$ the lengths $l, l_1, l_2$ are roughly $e^{\gamma (-r-C)/2}$, the nearby region should have field average close to $-r-C$; for large $C$ we thus have that $\tau_{-r}$ lies to the left of $U$. The argument is difficult in $(\cS, h)$ because we condition on the rare event $F_{r,C}$, so we instead perform the argument in $(\cS, \wt h + c)$. In Step 4, we go much further to the left in $(\cS, \wt h + c)$ to a region (corresponding to a region in $(\cS, h)$ with field average greater than $-r$) and analyze the boundary lengths to the right of the region; this gives us the desired bound $e^{\gamma(-r+K)/2}$ on the lengths in $(\cS, h)$ from $\tau_{-r}$ to $x_1$ and from $\tau_{-r} +i\pi$ to $x_2$. How far to the left we have to go depends on $C$, which is why we take $K > K_0(C)$.}
\end{figure}

\noindent
\textbf{Step 2: Showing the equivalence of the quantum surfaces $(\widetilde U_1 \cup \widetilde U_2,  \widetilde h + c)$ and $(V, h)$.} Let $V$ be the domain $ U_1 \cup  U_2$ with a slit along the interface $U_1 \cap U_2$ of quantum length $l$ (see Figure~\ref{fig:coupling}, bottom).

By definition, $\widetilde \cW_j + c = (\widetilde U_j, \widetilde h + c)$ is equivalent as a quantum surface to $\cW_j = (U_j, h)$ for $j = 1,2$. If we write $v_1$ for the point on the right edge of $\widetilde \cW_1 + c$ which is $l_1+l$ units of quantum length from the origin, and $v_2$ for the point on the left edge of $\widetilde \cW_2 + c$ which is $l_2+l$ units from the origin, then $(\widetilde U_1 \cup \widetilde U_2,  \widetilde h + c)$ is a conformal welding of $\widetilde \cW_1+ c$ and $\widetilde \cW_2 + c$ which glues the right boundary ray of $\widetilde\cW_1 + c$ from $v_1$ to $\infty$ to the left boundary ray of $\widetilde\cW_2 + c$ from $v_2$ to $\infty$ according to quantum length. Likewise, we observe that $(V, h)$ is a conformal welding of $\cW_1$ and $\cW_2$ along the analogously defined boundary rays.
By the conformal removability of $\SLE_\kappa$-type curves (with $\kappa \in (0,4)$), such weldings are unique; see \cite[Section 3.5]{wedges} for details. We conclude that the quantum surfaces $(\widetilde U_1 \cup \widetilde U_2,  \widetilde h + c)$ and $(V, h)$ are equivalent. 
\medskip

\noindent
\textbf{Step 3: Showing that $U$ lies to the right of $[\tau_{-r}, \tau_{-r} + i\pi]$ with high probability.}
We want to show that we can choose $C$ large enough so that with probability arbitrarily close to 1 for large $r$, the region $U$ lies to the right of the vertical line $[\tau_{-r}, \tau_{-r} + i\pi]$. In other words, we want to show that when $C$ is large, then $\tau_{-r} \leq u$ with uniformly high probability for large $r$, where
\begin{align} \label{eqn: leftmost pt}
u = \inf \{ \Re z : z \in U\}.
\end{align}

Let $\varphi: \widetilde U_1 \cup \widetilde U_2 \to V$ be the conformal map establishing the equivalence of the quantum surfaces $(\widetilde U_1 \cup \widetilde U_2, \widetilde h + c)$, $(V, h)$, so that on $\widetilde U_1 \cup \widetilde U_2$ we have $h\circ \varphi  + Q\log |\varphi'| = \wt h + c$. See Figure \ref{fig:coupling} for a description of $\varphi$. Our strategy for this step is to construct a function $\phi_{R_2}$ with support to the left of $U$, and use the above change-of-domain formula with the distortion estimate Lemma~\ref{lem: distortion estimate} to show that $(h, \phi_{R_2})$ is very negative with high probability. Thus some vertical field average to the left of $U$ is less than $-r$, so $U$ lies to the right of $[\tau_{-r}, \tau_{-r} + i\pi]$ with high probability.

Let $C_1, C_2$ be the constants in Lemma \ref{lem: distortion estimate}. 
Let $a = 3C_2$ and let $b > 1$ be some large constant to be chosen later that does not depend on $h, \wt h$. Now, we will pick $s$ and $C> 1$, as follows. We can choose $s \gg 0$ such that for $\widetilde h \sim \widetilde \cL$, we have $\P[\widetilde \tau_s + a +C_1 < \inf_{\widetilde U} \Re z]\approx 1$. Let $M_s(\wt h)$ for $s \in \BB R$ be as in Lemma~\ref{lem: uniform bounds on wedge fields}. Since $|(-r-C)-c|< N$,
\begin{equation*}
\P[M_s(\widetilde h) + c < -r - Q \log C_2 ] \geq \P[M_s(\widetilde h) - s < -N + C - s- Q \log C_2 ].
\end{equation*}
By Lemma \ref{lem: uniform bounds on wedge fields}, for $C \geq C_0(s)$ chosen sufficiently large, we have $\P[M_s(\widetilde h) + c < -r - Q \log C_2] \approx 1 $ uniformly for all $r>\beta$.  

For the rest of this step, we truncate on the intersection of the following high probability events:
\begin{gather}
\widetilde \tau_s + a + C_1 < \inf_{\widetilde U} \Re z, \label{eqn: tau_s occurs before interface} \\
M_s(\widetilde h) + c < - r - Q\log C_2 \label{eqn: field small near tau_s} .
\end{gather}
Let $\psi: [0, C_2] \times [0,\pi] \to \R$ be any function such that: 
\begin{itemize}
\item $\psi|_{[0,C_2]}$ is a nonnegative compactly supported bump function;
\item $\psi$ is constant on vertical lines;
\item $\int_{[0,C_2] \times [0,\pi]} \psi(z) \ dz  =1$. 
\end{itemize}
Consider the rectangle $R_1 = [\widetilde \tau_s , \widetilde \tau_s + a] \times [0,\pi]$. By \eqref{eqn: tau_s occurs before interface}, we see that the rectangle $R_1$ is at least a distance of $C_1$ away from $\widetilde U$. Consequently, Lemma \ref{lem: distortion estimate} tells us that the image $\varphi(R_1)\subset U_1 \cup U_2$ contains a rectangle $R_2$ of width $C_2$. Let $\phi_{R_2}$ be given by the composition of $\psi$ with a horizontal translation of the strip $\cS$ such that $\phi_{R_2}$ is supported in $R_2$, so its pullback $\phi =  |\varphi'|^2 \phi_{R_2} \circ \varphi$ is supported in $R_1$. By Lemma \ref{lem: distortion estimate} we can bound $\|\phi'\|_\infty$ above in terms of $C_1, C_2, \| \psi\|_\infty, \|\psi'\|_\infty$, and choosing $b$ large in terms of these constants, we get $\|\phi'\|_\infty < b$. Thus $\phi \in \Phi_{a,b}^1$ (with $\Phi_{a,b}^1$ defined in Lemma \ref{lem: uniform bounds on wedge fields}), and so \eqref{eqn: field small near tau_s} tells us that $(\widetilde h + c, \phi) < - r - Q \log C_2$. Therefore, bounding $|(\varphi^{-1})'|$ via Lemma \ref{lem: distortion estimate}, we get
\[(h,\phi_{R_2})  = (Q \log | (\varphi^{-1})'|, \phi_{R_2}) + (\widetilde h + c, \phi) <- r.\]
Since $\phi_{R_2}$ is constant on vertical lines and has integral against Lebesgue measure equal to 1, we conclude that for some vertical line in $R_2$ the field $h$ has average value less than $-r$. Therefore $\tau_{-r} < u$ with probability approaching 1 as $C \to \infty$.
\medskip

\noindent
\textbf{Step 4: Showing that $\nu_h ([\tau_{-r}, +\infty)) < q_1 + e^{\gamma (-r+K)/2}$ and $ \nu_h([\tau_{-r}, +\infty) + i\pi]) < q_2 + e^{\gamma (-r+K)/2}$ with high probability for $K$ large, uniformly for $r > \beta$.} We return to the setting of Steps 1 and 2, where we have a coupling of $(h, \eta')$ and $(\widetilde h, \widetilde \eta')$, and we are \emph{not} truncating on the events \eqref{eqn: tau_s occurs before interface}, \eqref{eqn: field small near tau_s}.

In the past, we have been exploring the field from left to right. For this step, we instead explore the field $h$ from right to left, stopping when we have discovered the whole domain $U$. More precisely, given the realization of the space-filling curve $\eta'$ (modulo time-parametrization), we discover the field $h|_{\cS_+ +M}$ and decrease the value of $M$ until we discover the points $x_1$ and $x_2$. Given these points and $\eta'$, we know the value of $u$ (defined in~\eqref{eqn: leftmost pt}), and discover $h|_{\cS_+ + u}$. We claim that conditioned on $h|_{\cS_+ + u}$ and $\eta'$, the conditional law of $h|_{\cS_- + u}$ is given by a GFF on $\cS_- + u$ with Neumann boundary conditions on $\R_- +u$ and $\R_- + u + i\pi$, and Dirichlet boundary conditions on $[u,u+i\pi]$ specified by $h|_{\cS_+ + u}$, with an added linear drift in the field average. This is true because it holds for a $\gamma$-quantum wedge field $\wh h$ conditioned on $\wh h|_{\cS_+ + u}, \eta'$ by Lemma~\ref{lem: rest of thick wedge}, and the event $F_{r,C}$ is measurable w.r.t. the $\sigma$-algebra generated by $\wh h |_{\cS_+ + u}$ and $\eta'$. Thus, defining the right-to-left field average process of $h$
\begin{equation*}
Y_t = \text{average of }h \text{ along } [u - t, u - t + i\pi],
\end{equation*}
$(Y_t - Y_0)_{t\geq 0}$ evolves as Brownian motion with variance 2 and upward linear drift of $(Q-\gamma)$ independently of $h|_{\cS_+ + u}$ and $\eta'$.

We henceforth condition on the high probability event $\{\tau_{-r} < u\}$ (see Step 3), so $(Y_t)_{t \geq 0}$ is a Brownian motion with random starting value and upward drift of $(Q - \gamma)$ conditioned to take the value $-r$ at some point. Write \[ \sigma_{-r} = \sup \{t \in (-\infty, u) \: \mid \: \text{the average of }h \text{ along } [t,t+i\pi] \text{ is } -r\}.\] By the strong Markov property of Brownian motion, the law of $|\tau_{-r} - \sigma_{-r}|$ conditioned on $h|_{\cS_+ + u}$ and $\eta'$ is given by the law of the last hitting time of 0 of a Brownian motion started at 0 with variance 2 and upward drift of $(Q-\gamma)$. As such, we can find some absolute constant $d$ such that with high probability we have $|\tau_{-r} - \sigma_{-r}| < d$. 

Next, we choose $x,K > 1$. Since $|(-r-C) - c| \leq N$, we have (with $m_x(\wt h)$ as in Lemma~\ref{lem: uniform bounds on wedge fields})
\begin{equation*}
\P[m_x(\widetilde h) + c > -r + Q \log C_2 ] \geq \P[m_x(\widetilde h) - x > C+N-x+ Q \log C_2],
\end{equation*}
so by Lemma \ref{lem: uniform bounds on wedge fields}, for $x \geq x_0(C)$ we have $\P[m_x(\widetilde h) + c > -r +  Q \log C_2] \approx 1$ uniformly for all $r$. Choose also $x > s$. 

Now, since $|c - (-r-C)|<N$ a.s., we have
\[ \nu_{\widetilde h + c} \left( [\widetilde \tau_x - d - C_2, +\infty) \times \{0,\pi\}\right) < e^{\gamma (-r-C+N)/2} \nu_{\widetilde h} \left( [\widetilde \tau_x - d - C_2, +\infty) \times \{0,\pi\}\right), \]
and so for all $K \geq K_0(x, d)$, 
\begin{equation*}
\P\left[ \nu_{\widetilde h + c} \left( [\widetilde \tau_x - d - C_2, +\infty) \times \{0,\pi\} \right)< e^{\gamma (-r + K)/2}\right] \approx 1 \text{ uniformly for all } r.
\end{equation*}
Thus, uniformly in $r$, the following three events have probability arbitrarily close to 1. We further truncate on them:
\begin{gather}
\sigma_{-r} - \tau_{-r} < d, \label{eqn: first and last hitting times of r are close}\\
m_x(\widetilde h) + c > -r+ Q \log C_2, \label{eqn: field large near tau_x} \\
\nu_{\widetilde h + c} \left( [\widetilde \tau_x - d - C_2, +\infty) \times \{0,\pi\} \right)< e^{\gamma (-r + K)/2}. \label{eqn: boundary length not massive after tau_r}
\end{gather}

Next, we repeat what we did in Step 3. Define $R_1' = [\widetilde \tau_x, \widetilde \tau_x + a] \times [0,\pi]$ and again let $R_2' \subset \varphi(R_1')$ be a rectangle of width $C_2$. Using \eqref{eqn: field large near tau_x}, the same argument from Step 3 shows that $h$ has average greater than $-r$ on some vertical line in $R_2'$. Note that since $x > s$, the rectangle $R_2'$ lies to the left of $R_2$, so by the intermediate value theorem $\sigma_{-r}$ lies to the right of $R_2'$. Therefore, \eqref{eqn: first and last hitting times of r are close} and the distortion estimate Lemma \ref{lem: distortion estimate} tell us that $\varphi^{-1} ([\tau_{-r}, \tau_{-r} + i\pi])$ lies in $\cS_+ + \widetilde \tau_x - d - C_2$. Hence \eqref{eqn: boundary length not massive after tau_r} implies
\begin{align*}
\nu_h([\tau_{-r}, x_1]) &\leq \nu_{\widetilde h + c} ([\varphi^{-1} (\tau_{-r}), +\infty)) < e^{\gamma (-r+K)/2}, \\
\nu_h([\tau_{-r}+i\pi, x_2]) &\leq \nu_{\widetilde h + c} ([\varphi^{-1} (\tau_{-r} + i\pi), +\infty + i\pi)) < e^{\gamma (-r+K)/2}.
\end{align*}
As $\nu_h(\R_+ + x_1) = q_1$ and $\nu_h(\R_+ + x_2) = q_2$, we conclude that with high probability 
\begin{align*}
\nu_h([\tau_{-r}, +\infty)) &< q_1 + e^{\gamma (-r+K)/2}, \\
\nu_h([\tau_{-r}, +\infty) + i\pi) &< q_2 + e^{\gamma (-r+K)/2}.
\end{align*}
This concludes Step 4.

Now we finish the proof. For $(h, \eta') \sim \cL$ (i.e. conditioned on $F_{r,C}$), by Step 3 we have with high probability that $[\tau_{-r}, \tau_{-r} + i\pi]$ lies to the left of $U$, and hence  
\begin{align*}
\nu_h([\tau_{-r}, +\infty)) &> \nu_h([x_i, +\infty)) =  q_1, \\
\nu_h([\tau_{-r}, +\infty) + i\pi) &> \nu_h([x_2, +\infty +i\pi)) =  q_2.
\end{align*}
Comparing this with Step 4, we conclude that $\P[E_{r,K} | F_{r,C}] \approx 1$, as needed. 
See Section~\ref{subsection: gamma large} for the regime $\gamma \in (\sqrt2,2)$.
\end{proof}

\subsection{Adaptations for the regime $\gamma \in (\sqrt2, 2)$}\label{subsection: gamma large}
In this section, we adapt the methods of the previous two sections to the regime $\gamma \in (\sqrt2, 2)$. This regime is more complicated because the space-filling $\SLE_{\kappa'}$ bounces off of itself and boundary segments, so when we perform the surface decomposition in Section~\ref{subsection: decomposing wedge}, we get countably many pieces (rather than three pieces). Fortunately, the difficulty is mostly psychological; most of the arguments carry over directly. We emphasize that for this regime $\gamma \in (\sqrt2,2)$, Theorem~\ref{thm: peanosphere disk} was earlier proved in \cite[Theorem 2.1]{sphere-constructions} by simpler methods. Regardless, we extend our proof to this setting to provide a unified treatment of the mating of trees on a quantum disk.

\medskip

\noindent
\textbf{Section~\ref{subsection: decomposing wedge}: Decomposing a $\gamma$-quantum wedge.} Define exactly as in Section~\ref{subsection: decomposing wedge} the $\gamma$-quantum wedge $(\cS,h,+\infty, -\infty)$, stopping time $\tau_t$ for the left-to-right exploration of the wedge, lengths $q_1 = q_1(r)$ and $q_2 = q_2(r)$, points $x_1, x_2$, counterclockwise space-filling $\SLE_{\kappa'}$ $\eta'$ with time recentered so $\eta'(0) = x_1$, and let $T$ be the time $\eta'$ hits $x_2$. 

Define the regions 
\[U = \eta'([0,T]), \quad U_1 = \eta'((-\infty, 0]), \quad U_2 = \eta'([T,+\infty)).\]
See Figure~\ref{fig: surface decomposition complicated}. When we discussed $\gamma \in (0,\sqrt2]$ earlier, each of these regions had simply connected interior. Now, in the $\gamma \in (\sqrt2,2)$ regime, only $U_1$ has simply connected interior. The interiors of $U, U_2$ each have countably many connected components which are totally ordered; see Remark~\ref{remark: topology of wedges} for details. Regardless, we can define the curve-decorated surfaces 
\begin{equation}
\cD^* := (U,h,\eta',x_1,p,x_2), \quad \cW_1^* := (U_1,h,\eta',x_1,-\infty), \quad \cW_2^* := (U_2,h,\eta', x_2, -\infty).
\end{equation}

\begin{figure}[ht!]
\begin{center}
\includegraphics[scale=0.75]{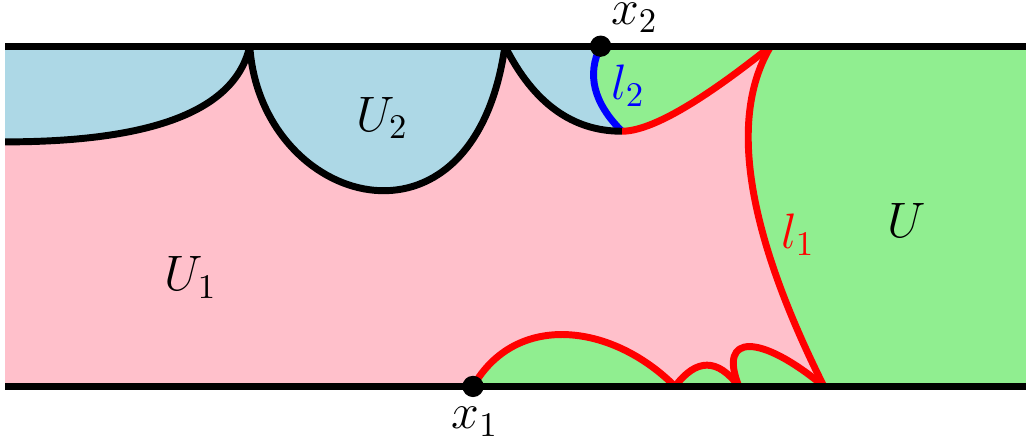}
\end{center}
\caption{\label{fig: surface decomposition complicated} Let $\gamma \in (\sqrt2, 2)$. By drawing an independent counterclockwise $\SLE_{\kappa'}$ on top of a $\gamma$-wedge $ (\cS, h, +\infty, -\infty)$, we can decompose the curve-decorated $\gamma$-wedge into three curve-decorated surfaces $\cD^*, \cW_1^*, \cW_2^*$ parametrized by the color-coded regions $U, U_1, U_2$ respectively. While $U_1$ has simply connected interior, the interiors of $U$ and $U_2$ each have countably many connected components. Define $l_1$ and $l_2$ to be the quantum lengths of the red and blue interfaces respectively.}
\end{figure}

The boundary length processes in each of $\cD^*, \cW_1^*, \cW_2^*$ are precisely those described in Section~\ref{subsection: decomposing wedge}, because those processes were derived purely from working with the boundary length processes in $(\cS, h, +\infty, -\infty, \eta')$, and this derivation involved no topological considerations. Moreover, these surfaces can a.s. be recovered given their boundary length processes, they are mutually independent, and scale invariant. 

As before, we define $l_1 = \nu_h(U \cap U_1)$ and $l_2 = \nu_h(U \cap U_2)$. The curves $U \cap U_j$ \emph{are} allowed to bounce off of the boundary $\partial \cS$, but do not intersect themselves -- see Figure~\ref{fig: surface decomposition complicated}. With these definitions, we define the event $F_{r,C}$ and the conditional law $\cL$ as in~\eqref{eqn: F def}. The proof of Lemma~\ref{lem: boundary length process in D} for $\gamma \in (\sqrt2,2)$ is identical to that of the regime $\gamma \in (0,\sqrt2]$.
\medskip

\noindent
\textbf{Section~\ref{subsection: equivalence E F}: Equivalence of $F_{r,C}$ and $F_{r,C} \cap E_{r,K}$.} All the results of this section still hold in the regime $\gamma \in (\sqrt2,2)$, and most of the arguments carry over directly. 

The only modification needed is in Step 2 of the proof of Proposition~\ref{prop: E given F}. Because of the complicated topology of the quantum surfaces $\cW_1^*, \cW_2^*$, we have to slightly modify the surface-cutting procedure. For $\gamma \in (0,\sqrt2]$, we defined $V$ to be the region $U_1 \cup  U_2$ with a slit of length $\nu_{ h} (\text{slit}) = l$. For $\gamma \in (\sqrt 2, 2)$, the slitted domain comprises countably many components; let $V$ be the component with $-\infty$ on its boundary. Similarly, the interior of the region $\widetilde U_1 \cup \widetilde U_2$ has countably many connected components; let $\widetilde V$ be the component whose boundary contains $-\infty$. Then the surfaces $(\widetilde V, \widetilde h + c)$ and $(V, h)$ are conformal weldings of the surfaces $\widetilde \cW_1^* + c = \cW_1^*$ and $\widetilde \cW_2^* + c = \cW_2^*$ by quantum length along certain boundary rays, after discarding all connected components except for the one containing $-\infty$ on its boundary. By Lemma~\ref{lem: identify-wedges}, we see that the welding interfaces are segments of an $\SLE_{\kappa}(\rho_L; \rho_R)$ process with $\rho_L, \rho_R > -2$. These interfaces are removable \cite[Proposition 3.16]{wedges}, so $(\widetilde V, \widetilde h + c)$ is equivalent to $(V, h)$ as a quantum surface.

\subsection{Extra conditioning does not affect the limit law}\label{subsection: extra conditioning}
Consider the full range $\gamma \in (0,2)$. The results of the previous sections tell us that, if we want to understand the law of $h$ conditioned on $F_{r,C}$, we can first condition on $E_{r,K}$ and then further condition on $F_{r,C}$. 
In this section, we check that this second conditioning does not change the limit law of $h$ --- conditioning $h$ on $E_{r,K} \cap F_{r,C}$ still gives a surface close to a quantum disk. 
Since $\P[E_{r,K} \mid F_{r, C}] \approx 1$ (Proposition~\ref{prop: E given F}) and we understand the law of the boundary length process given $F_{r,C}$ (Lemma~\ref{lem: boundary length process in D}), we conclude that the conditional law of the boundary length process given $E_{r,K} \cap F_{r,C}$ is close to  the Brownian cone excursion of Theorem~\ref{thm: peanosphere disk}. Sending $r, K, C \to \infty$ in that order, we obtain Theorem~\ref{thm: peanosphere disk}.
For notational convenience, in this section we will choose the horizontal translation of $h$ so that $\nu_h(\R_+) + \nu_h(\R_+ + i\pi) = \frac12$.

\begin{definition}
Given a space-filling curve $\eta'$ in $\cS$ and any $N \in \R$, define the \emph{restriction of $\eta'$ to $\cS_+ - N$} (denoted $\eta'|_{\cS_+ - N}$) to be the curve $\eta'|_{[s,t]}$, where $s$ is the first time that $\eta'$ enters $\cS_+ - N$, and $t$ is the last time that $\eta'$ exits $\cS_+ -N$.  
Note that $\eta'|_{\cS_+ - N}$ is typically not contained in $\cS_+-N$.
\end{definition}

\begin{lemma} \label{lemma: quantum disk given F and E'}
Fix $N \gg 0$. Consider the setup in Sections \ref{subsection: decomposing wedge} and \ref{subsection: equivalence E F}, so $(\cS, h, -\infty, +\infty)$ is a $\gamma$-quantum wedge, and parametrize so $\nu_h(\R_+) + \nu_h(\R_+ + i\pi) = \frac12$. Condition $(h, \eta')$ on $F_{r,C}$.

If we first send $r \to \infty$ and then $C \to \infty$, the field-curve pair given by $( h|_{\cS_+ - 2N}, \eta'|_{\cS_+ - N} )$ converges in total variation to a field-curve pair, with field $\psi$ given by a $(\frac12, \frac12)$-length quantum disk parametrized so $\nu_\psi(\R_+) + \nu_\psi(\R_+ + i\pi) = \frac12$ and then restricted to $\cS_+ - 2N$, and curve given by the restriction to $\cS_+ - N$ of an independent counterclockwise $\SLE_{\kappa'}$.
\end{lemma}

\begin{proof}
By Proposition \ref{prop: E given F}, it suffices to prove that the statement of the lemma holds for $(h, \eta')$ conditioned on $E_{r,K} \cap F_{r,C}$ when sending $r \to \infty, K \to \infty, C \to \infty$ in that order.

Let $h^{\op{IG}}$ be the imaginary geometry GFF used to construct $\eta'$ as in Section~\ref{subsection: space-filling SLE}, so $h^{\op{IG}}$ is independent of $h$.
Recall Proposition \ref{prop: head of GFF}, which roughly tells us the law of $(h|_R, x_1, x_2)$ conditional on $E_{r,K}$, where $R = [\tau_{-r},\tau_{-r} + S] \times [0,\pi]$. Pick $S$ large so that, when conditioned on $E_{r,K}$, with high probability the interfaces $U_1 \cap U$ and $U_2 \cap U$ (including their endpoints $x_1,x_2$) lie within $R$. When this occurs, given the realizations of $h|_R, h^{\op{IG}}|_R, x_1, x_2$, by \cite[Lemma 2.4]{gms-harmonic} we can check whether $F_{r,C}$ holds. In other words, when we condition on $E_{r,K}$, with high probability $F_{r,C}$ is determined by the restrictions of $h$ and $h^{\op{IG}}$ to $[0,S] \times [0,\pi]$ and the realizations $x_1, x_2 \in R$. Thus, by Propositions \ref{prop: conditioned on past, future is disk} and \ref{prop:F given E}, when we condition on $E_{r,K} \cap F_{r,C} $ and send $r\to \infty, K \to \infty, C \to \infty$ in that order, the field $ h|_{\cS_+ - 2N}$ is close in total variation to $\psi|_{\cS_+ - 2N}$. Furthermore, by Proposition \ref{prop: zoom in Dirichlet boundary}, conditioned on $h^{\op{IG}}|_R$, the field $h^{\op{IG}}$ restricted to $\cS_+ -r/2 $ is close in total variation to its unconditioned law, so by \cite[Lemma 2.4]{gms-harmonic}, the curve $\eta'$ restricted to $\cS_+ - N$ is close in total variation to the restriction of an independent counterclockwise $\SLE$ restricted to $\cS_+ - N$. 
\end{proof}

\begin{proof}[Proof of Theorem~\ref{thm: peanosphere disk}]
Let $N \gg 0$ be large. As in Section~\ref{subsection: decomposing wedge}, sample a $\gamma$-quantum wedge $(\cS, h, -\infty, +\infty)$ parametrized so $\nu_h(\R_+) + \nu_h(\R_+ + i\pi) = \frac12$, decorate it by an independent counterclockwise space-filling SLE $ \eta'$ from $-\infty$ to $-\infty$, and condition on $F_{r,C}$. Let $(L_t, R_t)_{[0,T]}$ be the boundary length process from the time $\eta'$ hits $x_1$ until the time $\eta'$ hits $x_2$, with $(L_0, R_0) = (0,q_1+q_2)$. 

Let $s_N$ and $t_N$ be the first and last times that $\eta'$ lies inside $\cS_+ - N$. As $r \to \infty$ then $C \to \infty$, by Proposition~\ref{prop: E given F} the curve segments $\eta'([0,s_N])$ and $\eta'([t_N, T])$ lie in the rectangle $[\tau_{-r}, -N]\times[0,\pi]$ with probability $1-o_C(1)$. Since $\eta'$ is parametrized by quantum area, we see that with probability $1-o_C(1)$ we have $s_N, T - t_N < \mu_h([\tau_{-r}, -N]\times[0,\pi])$. 
Taking $r \to \infty$, $C \to \infty$, $N \to \infty$ in that order, since $\mu_h([\tau_{-r}, -N]\times[0,\pi])\to 0$ in probability, we have $s_N, T - t_N \to 0$ in probability. Thus the $d_{\mathcal K}$ distance between the curves $(L_t, R_t)_{[0,T]}$ and $(L_t, R_t)_{[s_N, t_N]}$ in $\cS$ converges to zero in probability. (Recall that $d_{\mathcal K}$ is a metric on the space of curves, defined in \eqref{eqn: metric on space of curves}).

Let $(\cS, \psi, +\infty, -\infty)$ be a $(\frac12, \frac12)$-length quantum disk, decorated with an independent counterclockwise space-filling SLE $\widetilde \eta'$ from $-\infty$ to $-\infty$, and with the field $\psi$ horizontally translated so that $\nu_\psi(\R_+) + \nu_\psi(\R_+ + i\pi) = \frac12$. Let $\widetilde T$ be the duration of $\widetilde \eta'$ when parametrized by quantum area (so $\widetilde \eta'(0) = \widetilde \eta' (\widetilde T) = -\infty$). Let the boundary length process be $(\widetilde L_t, \widetilde R_t)_{[0,\widetilde T]}$, and define $\widetilde s_N, \widetilde t_N$ in the same way as above. As before, as $N \to \infty$, the $d_{\mathcal K}$ distance between the curves $(\widetilde L_t, \widetilde R_t)_{[0,\widetilde T]}$ and $(\widetilde L_t, \widetilde R_t)_{[\widetilde s_N,\widetilde t_N]}$ goes to zero in probability.

By Lemma \ref{lemma: quantum disk given F and E'}, we can couple the field/curve pairs $(\psi, \widetilde \eta')$ and $(h,  \eta')$ so that with high probability we have $h|_{\cS_+ - 2N} = \psi|_{\cS_+ - 2N}$, and the restrictions of $\eta', \widetilde \eta'$ to $\cS_+ -N$ agree (recall that we translated the field $h$ so that $\nu_h(\R_+) + \nu_h(\R_+ + i\pi) = \frac12$). Note that as $N \to \infty$, the probability that the curve $\eta'|_{[s_N, t_N]}$ stays inside $\cS_+ - 2N$ tends to 1; this means in particular that with probability approaching 1 the processes $(L_t, R_t)|_{[s_N, t_N]}$ and $(\widetilde L_t, \widetilde  R_t)|_{[\widetilde s_N, \widetilde t_N]}$ agree exactly modulo additive constant.

Now we check that this additive constant is small with high probability. As $N \to \infty$, because $s_N , \widetilde s_N \to 0$ in probability, with probability tending to 1 the ``starts'' of the curves $(L_t, R_t)|_{[0,s_N]}$ and $(\widetilde L_t, \widetilde R_t)|_{[0,\widetilde s_N]}$ stay uniformly close to $(0,1)$. Likewise, the ``ends'' of the curves stay uniformly close to $(0,0)$. Thus in our coupling, with probability approaching 1, the $d_{\mathcal K}$ distance between $(L_t, R_t)|_{[s_N, t_N]}$ and $(\widetilde L_t, \widetilde  R_t)|_{[\widetilde s_N, \widetilde t_N]}$ is arbitrarily small.

Combining all this, we see that for any $\delta > 0$, we can choose $N,C,r \gg 0$, so that we can couple the processes $(L_t, R_t)_{[0,T]}$ and $(\widetilde L_t, \widetilde  R_t)_{[0, \widetilde T]}$ such that with probability $1-\delta$ the $d_{\mathcal K}$ distance between $(L_t, R_t)_{[0,T]}$ and $(\widetilde L_t, \widetilde  R_t)_{[0, \widetilde T]}$ is at most $\delta$. We conclude that the boundary length process $(\widetilde L_t, \widetilde  R_t)_{[0, \widetilde T]}$ is precisely given by the cone excursion process described in Lemma \ref{lem: boundary length process in D}. Forgetting the marked point $+\infty$ on the quantum disk (indeed, it is determined by the quantum surface $(\cS, \psi, -\infty)$ since the arcs separating the two marked points each have $\nu_\psi$-length $1/2$), we see that when we sample an independent counterclockwise space-filling $\SLE_{\kappa'}$ $\eta'$ from $-1$ to $-1$ on a quantum disk $(\D, \psi, -1)$, the boundary length process is as described in Theorem~\ref{thm: peanosphere disk}.

Finally, we check that the curve-decorated $(\frac12, \frac12)$-quantum disk $(\cS, \psi, \eta', +\infty, -\infty)$ is a.s. determined by its boundary length process $(\widetilde L_t, \widetilde R_t)_{t \in [0,\mu_\psi(\D)]}$. Write $\cW^*$ for the curve-decorated $\gamma$-quantum wedge of Theorem~\ref{thm: peanosphere gamma wedge}. In the above proof, we obtained approximations (in total variation) of $(\cS, \psi, \eta', +\infty, -\infty)$ on neighborhoods bounded away from $-\infty$ by looking at neighborhoods of $\cW^*$ conditioned on positive probability events. Thus by Remark~\ref{rem-recover} we see that for any $\ep > 0$, the process $(\widetilde L_t - \widetilde L_\ep ,  \widetilde R_t - \widetilde R_\ep)_{t \in [\ep , \mu_\psi(\cS) - \ep]}$ a.s.\ determines the curve-decorated quantum surface parametrized by the domain $\eta'([\eps, \mu_\psi(\cS) -\eps])$.
Sending $\ep \rta 0$ and forgetting the extra marked point $+\infty$ concludes the proof of Theorem~\ref{thm: peanosphere disk}.
\end{proof}

In the case $\gamma \in (\sqrt2,2)$, Theorem~\ref{thm: peanosphere disk} is equivalent to \cite[Theorem 2.1]{sphere-constructions}, but due to notational differences this may not be immediately apparent. We provide here a restatement of Theorem~\ref{thm: peanosphere disk} of our paper to show the above equivalence. We note that the space-filling SLE$_{\kappa'}$ in the statement of \cite[Theorem 2.1]{sphere-constructions} is the time-reversal of the space-filling $\SLE_{\kappa'}$ considered in this paper.

\begin{corollary}\label{cor: restate disk peanosphere}
Suppose that $\gamma \in (0,2)$, and that $(\D, \psi, -1)$ is a unit boundary length quantum disk. Let $\eta'$ be a counterclockwise space-filling $\SLE_{\kappa'}$ process from $-1$ to $-1$ sampled independently from $h$ and then reparametrized by quantum area. Let $\widehat \eta'$ be the time-reversal of $\eta'$, and let $T$ denote its random duration. Let $\widehat L_t$ and $\widehat R_t$ denote the quantum lengths of the left and right sides of $\widehat\eta'([0,t])$, normalized so that $L_0 = R_0 = 0$; see Figure \ref{fig: peanosphere_disk_boundary_convention} (left). Then $(\widehat L_t, \widehat R_t)_{ 0 \leq t \leq T}$ is a finite-time Brownian motion started from $(0,0)$ and conditioned to stay in the first quadrant $\R^+ \times \R^+$ until it exits at $(1,0)$.
\end{corollary}

\begin{figure}[ht!]
\begin{center}
\includegraphics[scale=0.85]{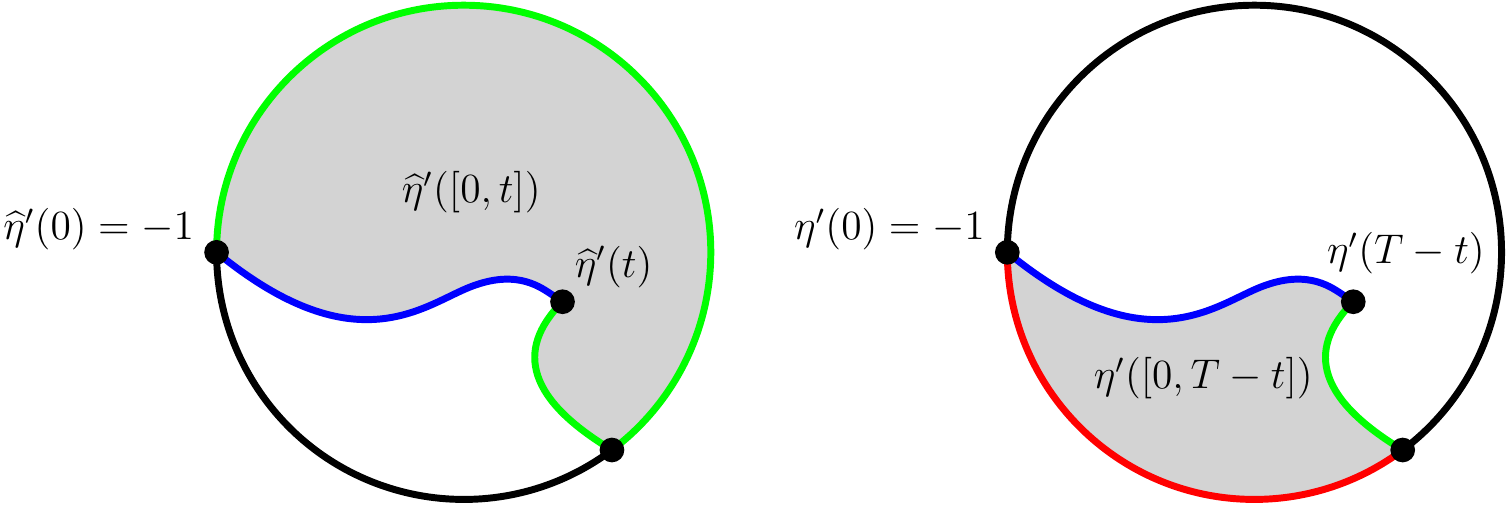}
\end{center}
\caption{\label{fig: peanosphere_disk_boundary_convention} For $\gamma \in (0, \sqrt2]$, consider a unit boundary length quantum disk $(\D,h,-1)$ with an independently drawn counterclockwise space-filling $\SLE_{\kappa'}$ curve $\eta'$ from $-1$ to $-1$ parametrized by quantum area. \textbf{Left:} Let $T$ be the duration of $\eta'$, and let $\widehat \eta'$ be the time-reversal of $\eta'$. We define $\widehat L_t = \nu_h(\text{green})$ and $\widehat R_t =  \nu_h(\text{blue})$. \textbf{Right:} For any time $t \in [0,T]$, by the boundary length definitions of Theorem~\ref{thm: peanosphere disk}, we have $L_{T-t} = \nu_h(\text{blue})$ and $R_{T-t} =1+ \nu_h(\text{green}) - \nu_h(\text{red})$. \textbf{Both:} By comparing diagrams and recalling that the boundary of the disk has quantum length 1, we have $\widehat L_t = R_{T-t}$ and $\widehat R_t = L_{T-t}$. Thus the process $(\widehat L_t, \widehat R_t)$ is the time-reversal of $(R_t, L_t)$. By Theorem~\ref{thm: peanosphere disk} and the reversibility of Brownian excursions, $(\widehat L_t, \widehat R_t)_{0\leq t\leq T}$ is a Brownian cone excursion from $(0,0)$ to $(1,0)$ with covariances given by~\eqref{eqn: covariance}. For $\gamma \in (\sqrt2, 2)$, the topology of the diagram is more complicated, but nevertheless analogous statements hold.}
\end{figure}

\bibliographystyle{hmralphaabbrv}
\bibliography{cibib,bibmore}

\end{document}